\documentclass[twoside,12pt]{article}
\usepackage{amsmath,amssymb,amsthm,amsfonts}
\usepackage{paralist}
\usepackage{graphics} 
\usepackage{epsfig} 
\usepackage{graphicx}
\usepackage{caption}
\usepackage{subcaption}
\usepackage{mathrsfs}
\usepackage{epstopdf}
\usepackage[colorlinks=true]{hyperref}
\usepackage{txfonts}
\usepackage{enumerate}
 \usepackage{verbatim}

\usepackage[numbers,sort&compress]{natbib}
\usepackage[font=scriptsize,labelfont=bf,width=0.85\textwidth]{caption}

\numberwithin{equation}{section}

\newcommand{\be}{\begin{equation}}
\newcommand{\ee}{\end{equation}}

\newcommand{\ben}{\begin{eqnarray*}}
\newcommand{\enn}{\end{eqnarray*}}

\newtheorem{proposition}{Proposition}[section]
\newtheorem{theorem}{\textbf Theorem}[section]
\newtheorem{lemma}{\textbf Lemma}[section]

 \numberwithin{equation}{section}
\newtheorem{remark}{Remark}[section]
\begin{document}

\title{{\textbf{Existence and Stability of Localized Patterns in the Population Models with Large Advection and Strong Allee Effect}}}

\author{Fanze Kong \thanks{Department of Mathematics, University of British Columbia, Vancouver, BC, V6T 1Z2, Canada; fzkong@math.ubc.ca} and
Jun-cheng Wei \thanks{Department of Mathematics, University of British Columbia, Vancouver, BC, V6T 1Z2, Canada; jcwei@math.ubc.ca}
        }

        \date{\today}
\maketitle
\vspace{-0.44in}
\begin{abstract}
 The strong Allee effect plays an important role on the evolution of population in ecological systems.  One important concept is the Allee threshold that determines the persistence or extinction of the population in a long time.  In general, a small initial population size is harmful to the survival of a species since when the initial data is below the Allee threshold the population tends to extinction, rather than persistence.  Another interesting feature of population evolution is that a species whose movement strategy follows a conditional dispersal strategy is more likely to persist.  In other words, the biased movement can be a benefit for the persistence of the population.  The coexistence of the above two conflicting mechanisms makes the dynamics rather intricate.  However, some numerical results obtained by Cosner et. al. (SIAM J. Appl. Math., Vol. 81, No. 2, 2021) show that the directed movement can invalidate the strong Allee effect and help the population survive.  To study this intriguing phenomenon, we consider the pattern formation and local dynamics for a class of single species population models that is subject to the strong Allee effect.  We first rigorously show the existence of multiple localized solutions when the directed movement is strong enough.  Next, the spectrum analysis of the associated linear eigenvalue problem is established and used to investigate the stability properties of these interior spikes.  This analysis proves that there exists not only unstable but also linear stable steady states.  Then, we extend results of the single equation to coupled systems, and also construct several non-constant steady states and analyze their stability.  Finally, numerical simulations are performed to illustrate the theoretical results.
 \end{abstract}

\section{Introduction}
In this paper, we mainly investigate the aggregation phenomena and dynamics of the following single reaction-diffusion equation for $u=u(x,t)$ with the no-flux boundary condition:
\begin{align}\label{origin}
\left\{\begin{array}{ll}
u_t=\nabla\cdot(d_1\nabla u-\chi u\nabla A)+\mu u(1-u)(u-\theta),&x\in\Omega,t>0,\\
(d_1\nabla u-\chi u\nabla A)\cdot \textbf{n}=0,&x\in\partial\Omega, t>0,\\
u(x,0)=u_0(x)\geq 0,&x\in\Omega.
\end{array}
\right.
\end{align}
Here $x$ and $t$ are space and time variables, $d_1$, $\chi$ and $\mu$ are arbitrary positive constants, and $\textbf{n}$ is the unit outer normal on the boundary $\partial\Omega$.  To construct the non-constant steady states, we also focus on the stationary problem of (\ref{origin}), which is:
\begin{align}\label{ss}
\left\{\begin{array}{ll}
0=\nabla\cdot(d_1\nabla u-\chi u\nabla A)+\mu u(1-u)(u-\theta),&x\in\Omega,\\
(d_1\nabla u-\chi u\nabla A)\cdot \textbf{n}=0,&x\in\partial\Omega.
\end{array}
\right.
\end{align}

Equation (\ref{origin}) serves as a paradigm to describe the dynamics of one population with the effect of some known signal subject to the Allee Principle \cite{Allee,SSF}, where $u:\Omega\times [0,\infty) \mapsto [0,\infty)$ denotes the density of a population and $A$ is a known stimulus that governs the directed movement; the constant $d_1$ represents the population diffusion rate, $\chi$ reflects the strength of the biased movement, while the source $f(u):=u(1-u)(u-\theta)$ models the Allee effect and $\theta\in(0,1)$ is the Allee threshold.   

 The general form of system (\ref{origin}) was proposed by Cosner and Rodriguez \cite{Nancy}, which reads:
 \begin{align}\label{Cosner}
\left\{\begin{array}{ll}
u_t=\nabla\cdot(\overbrace{d_1\nabla u}^{\text{random (flux)}}-\overbrace{\chi u\nabla A}^{\text{drift (flux)}})+\overbrace{\mu u(1-u)(u-\theta)}^{\text{Allee Effect Source}},&x\in\Omega,t>0,\\
\mathcal B[u]=0,&x\in\partial\Omega,t>0,\\
u(x,0)=u_0(x)\geq 0,&x\in\Omega,
\end{array}
\right.
 \end{align}
where $\mathcal B[u]=0$ represents either homogeneous Dirichlet or no-flux boundary conditions.  In particular, they obtain a set of qualitative and numerical results concerning the short time dynamics and steady states of system (\ref{Cosner}).  Moreover, to study the interaction between two species, they extended equation (\ref{Cosner}) to the following system:
\begin{align}\label{14}
\left\{\begin{array}{ll}
u_t=\mathcal M_u u+ug(x,u+v),\\
v_t=\mathcal M_v v+vg(x,u+v),
\end{array}
\right.
\end{align}
where $(u,v)$ are the population densities of two species and the dispersal operators are defined by 
$$\mathcal M_u u:=\nabla\cdot(\nabla u-\chi_1 u\nabla A),~\quad\chi_1>0,$$
and 
$$\mathcal M_v v:=\nabla\cdot(\nabla v-\chi_2 v\nabla A),~\quad \chi_2>0;$$
while the growth pattern is $g(x,u+v):=(r(x)-u-v)(u-v-\theta)$ and where the resource distribution is given by $r(x)=e^{A(x)}.$  Some numerical results for (\ref{14}) presented in \cite{Nancy} demonstrated that two populations cooperate at low densities and compete at high densities. 

To study this phenomenon, we consider the coupled system (\ref{14}) in the following two cases of $\chi_1$ and $\chi_2$:
\begin{itemize}
    \item[(i).] $\chi_1=\chi$, $\chi_2=1$, where $\chi>0$ represents the speed of the established species;
    \item[(ii).] $\chi_1=\chi$, $\chi_2=c\chi,$ where constant $c>1$ implies the invading species is faster.
\end{itemize}
In particular, we prove the existence of non-constant steady states for system (\ref{14}) in case (i) and case (ii), then study their stability properties.  
\subsection{Allee Effect}

The well-accepted definition of Allee effect is the positive relationship between population density and individual fitness.  This effect often occurs under situations involving the survival and reproduction of animals, such as habitat alteration, mate-finding \cite{Dennis,Gomes}, etc.

In terms of the scale, the Allee principle is typically decomposed into the component Allee effect and the demographic Allee effect.  The former emphasizes the relationship between any measurable component of survival rates and density size \cite{BAC}, while the latter highlights the overall correlation between them \cite{KDL}.  Many researchers tend to consider macro-population problems, and thereby the demographic Allee effect is more popular.  Some significant concept therein is the critical population size.  When a population threshold exists, the demographic Allee effect is the so-called strong Allee effect; otherwise it is named the weak Allee effect.  In general, when the initial density is below (above) the critical threshold, the population tends to be extinct (persistent).   The critical population size is called the Allee threshold and the relevant models have been intensively studied, see \cite{Liu,WSW2011,WSWJDE,WSW}.  

The most popular and simplest equation used to model the population dynamics subject to the strong Allee effect is
\begin{align*}
u_t=u(r-u)(u-\theta),
\end{align*}
where $r$ represents the environmental resources and $\theta\in (0,1)$ is the Allee threshold.  Here we define $g(u):=(r-u)(u-\theta)$ which admits the bistable growth pattern.  It can be seen that when the environment is homogeneous, $u\equiv \theta$ and $u\equiv r$ are two constant equilibria.  In particular, $u\equiv\theta$ is unstable and $u\equiv r$ is stable.  

\subsection{Directed Movement: Taxis and Advection}
A taxis is the mechanism by which organisms direct their movements in response to the environmental stimulus gradient.  In terms of stimulus such as wind, light, chemical signal, etc., taxis can be identified as Anemotaxis, Phototaxis, Chemotaxis and so on.  In particular, the effect of taxis on population dynamics is often interpreted as the conditional dispersal of species \cite{MH} and the advection term presents a paradigm to model it mathematically.

Combining the biased and unbiased dispersal, many reaction-diffusion-advection models have been proposed in the literature to analyze biological problems involving population dynamics.  The survey paper \cite{Cosner} summarizes a class of such systems and their applications.  The conditional dispersal in general is a benefit for the persistence of a species \cite{BC1995}, with the sensible explanation that individuals can perceive the favorable environmental signals such as the presence of food, and then move towards the stimulus and finally aggregate. 

There have been many previous results for the case where the population dynamics follows a logistic growth \cite{CC1991,CL2003,BC1995,CS2012,HW2016,lam2010,lam2011,lam2012,lam2014}.  In particular, Belgacem and Cosner \cite{BC1995} considered the following reaction-diffusion-advection model:
 \begin{align}\label{logistic}
\left\{\begin{array}{ll}
u_t=\nabla\cdot(d_1\nabla u-\chi u\nabla A)+\mu u(A-u),&x\in\Omega,t>0,\\
(d_1\nabla u-\chi u\nabla A)\cdot \textbf{n}=0,&x\in\partial\Omega,t>0,\\
u(x,0)=u_0(x)\geq 0,&x\in\Omega,
\end{array}
\right.
 \end{align}
where the environment is spatially heterogeneous and the boundary acts as a reflecting barrier.  They proved the population tends to be persistent if $\chi$ is large, which implies that the strong advection effect is beneficial.  Moreover, they showed that there exists some unique non-negative constant $\bar\mu_*$ depending on $\chi$ such that when $\mu>\bar\mu_*$, (\ref{logistic}) admits a unique positive global attractor.  Cosner and Lou \cite{CL2003}  further showed that the effect of the biased movement is not always beneficial and depends crucially on the shape of the domain, where it was established that non-convex domains can be harmful to the persistence of the population (See also  interesting related results in Chen and Lou \cite{ChenLou2012}).  It is worthwhile to mention the interesting work of Lam and Ni on the Lotka-Volterra competition models with the logistic growth.  They considered the competition between the ``smarter'' species with a large advection and the other who just follow random diffusion strategies in the heterogeneous environment.   Under some mild assumption of the resources function $A(x)$, the authors \cite{lam2010,lam2011,lam2012} showed that the densities of species must concentrate at all non-degenerate local maximum points of the resources not only in 1D but also in any dimensions.  Lam and Ni \cite{lam2014} further focused on the general environment function $A(x)$ and established many interesting results.  They found that there exists the competitive exclusion phenomenon when $A(x)$ attains its maximum everywhere on some open set, which is different from the well-known results shown in the case that $A(x)$ only has finitely many non-degenerate maxima.

There are also many different results when the boundary condition is assumed to be Dirichlet:
$$u(x,t)=0,\quad x\in\partial\Omega,~~ t>0,$$
which is the so-called lethal boundary.  For instance, a strong drift term may be harmful rather than helpful \cite{BC1995}, and more interesting results were shown in \cite{CC1991,Hess}.  Similar to the logistic growth, Allee effects also have rich applications in modelling population dynamics.  There are a few references focused on discussing the models subject to Allee effects \cite{CCH1996,NJRB2021,Zhang}. 


\subsection{Motivations and Main Results}
Cosner and Rodriguez \cite{Nancy} combined the free and conditional dispersal to model the movement of a population with the assumption that its dynamics is governed by the strong Allee Principle.  They proposed (\ref{Cosner}) and studied the existence of equilibrium subject to the lethal boundary or reflecting boundary.  Furthermore, some numerical simulations were presented to illustrate that the biased movement plays a vital role on overcoming a strong Allee effect.  The figures in \cite{Nancy} show if $\chi$ is large, i.e. the advection effect is strong, the population will persist rather than disappear even though the initial size is below the Allee threshold $\theta$.

To confirm this numerical experimental finding, we perform theoretical studies by considering system (\ref{origin}) and (\ref{ss}).  Our main goal is to rigorously construct non-constant solutions of (\ref{ss}), and then investigate their stability properties within (\ref{origin}).  In particular, since we focus in understanding the influence of the conditional dispersal rate $\chi$ on the strong Allee effect, we set the remaining parameters $d_1$ and $\mu$ to one.

An immediate consequence of the no-flux boundary condition is the following integral constraint satisfied by all classical solutions of (\ref{ss}):
\begin{align}\label{constraint}
\int_{\Omega} u(1-u)(u-\theta)dx=0.
\end{align}
It can be observed from (\ref{constraint}) that system (\ref{origin}) can admit different nontrivial patterns.  Indeed, some formal analysis implies this integral constraint determines the height of each local interior spike.  We suppose $A$ is smooth and radial with only one non-degenerate local maximum point at $0$. Then we expand $A$ as $A=A_0-\frac{a}{2}\vert x\vert^2+O(\vert x\vert^3)$, where $A_0:=A(0)$ is the local maximum of $A$ and $a:=A_{rr}(0)>0$.  Set $U=U_0+\epsilon U_1+\cdots$ and let
$\chi:=\frac{1}{\epsilon^2}$, $y:=\frac{x}{\epsilon}$, $U(y):=u(x)$, $F(U):=f(u)$ to obtain the following leading order equation:
\begin{align}\label{core}
\left\{\begin{array}{ll}
0=\nabla_y\cdot(\nabla_y U_0+ aU_0\cdot y),&y\in\mathbb R^n,\\
U_0(y)\rightarrow 0, ~~as~~ \vert y\vert\rightarrow\infty,\\
\int_{\mathbb R^n}F(U_0)dy=0.
\end{array}
\right.
\end{align}
By solving (\ref{core}), we have
\begin{align}\label{formalu0}
U_0=c_0e^{-\frac{a}{2}\vert y\vert^2}.
\end{align}
Upon substituting $U_0$ into (\ref{constraint}), we find $c_0$ satisfies
\begin{align}\label{17}
\int_{\mathbb R^n} e^{-\vert y\vert^2}(c_0 e^{-\vert y\vert^2}-\theta)(1-c_0e^{-\vert y\vert^2})dy=0.
\end{align}
By a straightforward calculation, we readily find that (\ref{17}) is equivalent to the following quadratic algebraic equation:
\begin{align}\label{quad}
2^{\frac{n}{2}}c_0^2-(1+\theta)3^{\frac{n}{2}}c_0+6^{\frac{n}{2}}\theta=0.
\end{align}
It is easy to check that there exists 
\begin{align}\label{theta1}
\theta_1:=\frac{2^{n+1}-2\sqrt{4^n-2^n\cdot 3^{\frac{n}{2}}}-3^{\frac{n}{2}}}{3^{\frac{n}{2}}}\in(0,1)
\end{align}
such that for $\theta\in(0,\theta_1)$, there are two values for $c_0$ given by
\begin{align}\label{c0}
c_{01}=\frac{(1+\theta)3^{\frac{n}{2}}+\sqrt{\delta}}{2^{\frac{n}{2}+1}},\quad c_{02}=\frac{(1+\theta)3^{\frac{n}{2}}-\sqrt{\delta}}{2^{\frac{n}{2}+1}}, 
\end{align}
where $\delta:=3^n\theta^2+2\cdot 3^n\theta-4\cdot 12^{\frac{n}{2}}\theta+3^n.$  Thanks to (\ref{c0}) and (\ref{formalu0}), the asymptotic profiles of single interior spikes can be expressed explicitly and are shown in Figure \ref{singlespike}.  
\begin{figure}[h!]
\centering
\begin{subfigure}[t]{0.5\textwidth}
    \includegraphics[width=\linewidth]{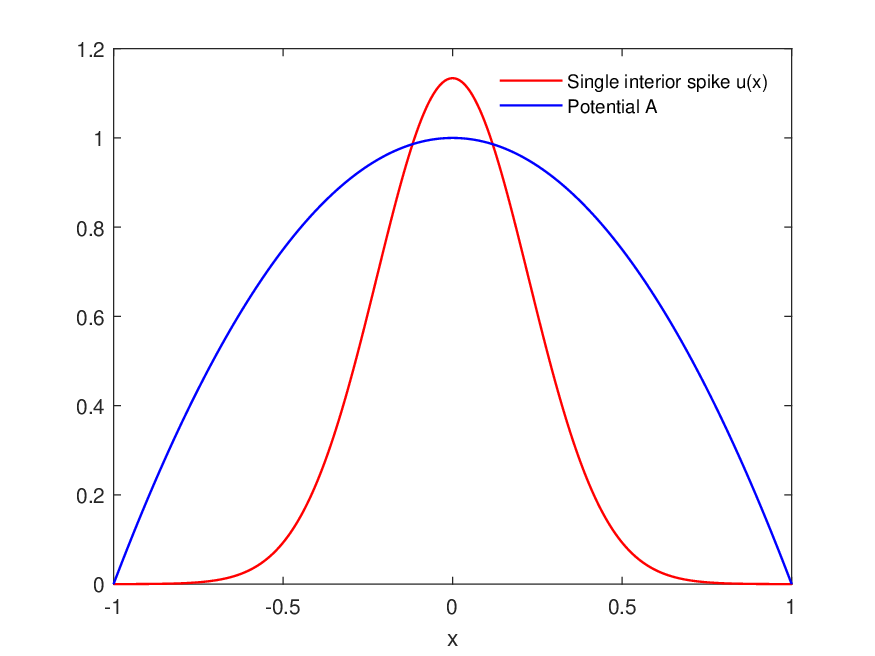}
\end{subfigure}\hspace{-0.25in}
\begin{subfigure}[t]{0.5\textwidth}
  \includegraphics[width=\linewidth]{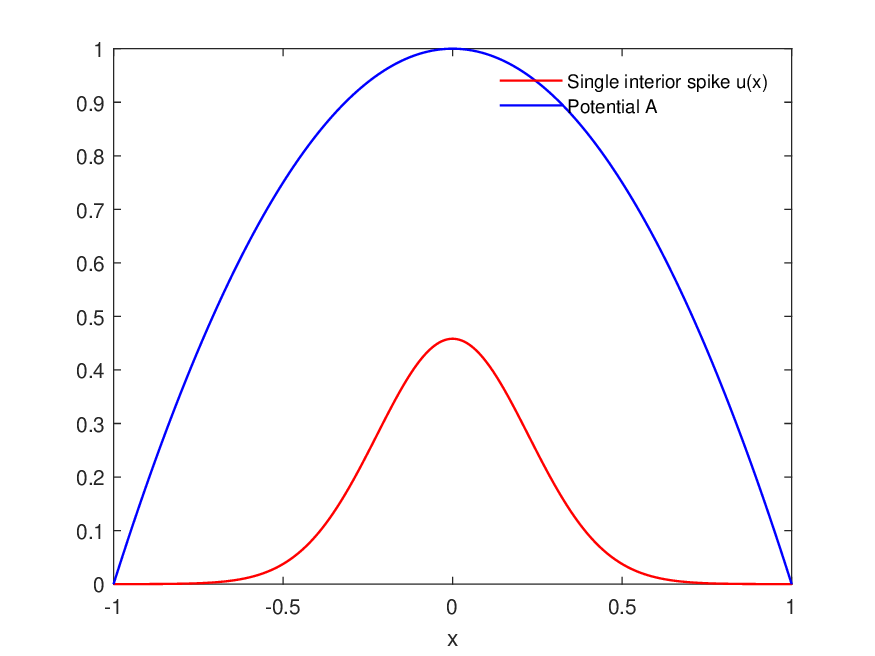}
\end{subfigure}
\\ \vspace{-0.15in}
\caption{\textit{For a 1-D domain with given potential $A=1-x^2$, Allee threshold $\theta=0.3$ and conditional dispersal rate $\chi=10$, we show the leading order profiles of single interior spikes defined by (\ref{formalu0}) with $c_0=c_{01}=1.1339$ (left) and $c_0=c_{02}=0.4582$ (right).} We shall prove that the left single interior spike is stable but the right one is unstable.}
\label{singlespike}
\vspace{-0.0in}
\end{figure}

 We would like to point out that when $\theta\in(\theta_1,1)$, (\ref{core}) only admits the solution $U_0=0$ since the quadratic equation (\ref{quad}) does not have any real solution.  As a consequence, when $\theta\in(\theta_1,1)$, there only exists a trivial pattern or the only one non-trivial spatially homogeneous pattern to (\ref{origin}) that we are not interested in.  Therefore, we only focus on the case $\theta\in(0,\theta_1)$ rather than $\theta\in[\theta_1,1).$  The above formal argument supports our claim that the height of a spike, given by $c_0$, is determined by the integral constraint (\ref{constraint}) and given by (\ref{c0}).  Moreover, our forthcoming rigorous analysis will prove that this statement holds for not only this special form for $A$ but also for a more general class of $A$.

Before stating our main results for the pattern formation of (\ref{origin}), we discuss the properties of the signal $A$.  Indeed, it plays the vital role for the formation of nontrivial patterns within (\ref{origin}).  Numerical simulations exhibited in \cite{Nancy} show that the non-constant steady states to (\ref{origin}) tend to be concentrated at the local non-degenerate maximum points of $A$.  The formal asymptotic analysis given above also confirms this fact.  Now, we recall the assumptions satisfied by the admissible signal $A$ in \cite{Nancy}, which are as follows:
\begin{itemize}
\item[(A1).] $A\in C^2(\bar\Omega)$ is time independent and spatially heterogeneous;
\item[(A2).] $\Vert\Delta A\Vert_{L^\infty(\Omega)}\leq M$ for some constant $M>0$.
\end{itemize}
Assumption (A1) and (A2) are technical assumptions needed for the analysis.  For our analysis below, we also propose several new hypotheses on $A$:
\begin{itemize}
\item[(H1).] all critical points of A are either local non-degenerate maximum points, or critical
points with $\Delta A>0$;
\item[(H2).] $\frac{\partial A}{\partial \textbf{n}}<0$ holds for all $x\in\partial\Omega$.
\end{itemize}
 After supposing $A$ admits exactly $k$ non-degenerate local maximum points $x_1,\cdots,x_k$, where $x_m:=\big(x_m^{(1)},\cdots,x_m^{(n)}\big)^T$, $m=1,\cdots,k.$  we have from assumption (A1), (A2) and hypothesis (H1), (H2) that $A$ can be expanded at $x_m$ as
$$A=A_m-\frac{1}{2}\sum\limits_{i,j=1}^n\big(x^{(i)}-x^{i}_m\big)^Th_m^{(ij)}\big(x^{(i)}-x_m^{(i)}\big)+o\big(\vert x-x_m\vert^2\big),$$
where $A_m:=A(x_m)$ and $-h_m^{ij}$ is the $ij$-th entry of the Hessian matrix of $A$ at $x_m.$  It is necessary to point out that the Hessian matrix of $A$ at every local non-degenerate maximum point $x_m$ is negative definite.  To simplify our subsequent analysis, one can utilize the rotation transform to write the expansion of $A$ as
\begin{align}\label{expansionA1}
A=A_m-\frac{1}{2}\sum\limits_{i=1}^n\hat h_m^{(i)}\big(\hat x^{(i)}-\hat x_m^{(i)}\big)^2+o\big(\vert x-x_m\vert^2\big),
\end{align}
where $-\hat h_m^{(i)}<0$ is the $i$-th eigenvalue of the Hessian matrix of $A$ at $\hat x_m$ and $\hat x$, $\hat x_m$ are rotated vectors.  We further rewrite (\ref{expansionA1}) as the following form
\begin{align}\label{expansionA}
A=A_m-\frac{1}{2}\sum\limits_{i=1}^n h_m^{(i)}\big( x^{(i)}- x_m^{(i)}\big)^2+o\big(\vert x-x_m\vert^2\big),
\end{align}
where the notations $h^{(i)}_{m}$, $x^{(i)}$ and $x^{(i)}_m$ are used to substitute $\hat h^{(i)}_{m}$, $\hat x^{(i)}$ and $\hat x^{(i)}_m$ in (\ref{expansionA}), respectively without confusing readers.

With the help of the above discussion, now we summarize the first set of our results regarding the stationary problem in the following theorem:
\begin{theorem}\label{thm11}
Under the assumptions (A1)-(A2) and hypotheses (H1)-(H2), define $k$ as any positive but fixed integer, let $\mathcal I_s\subseteq\{1,\cdots,k\}$ be any nonempty subset and $\mathcal I_b$ be any subset of $ \mathcal I_s,$ then we have there exists $\chi_0$ such that when $\chi>\chi_0$ and $\theta\in(0,\theta_1)$ with $\theta_1$ being given in (\ref{theta1}), (\ref{ss}) admit solutions having the following forms:
\begin{align}\label{u0}
 u_{s}(x;\chi)=\sum_{m=1}^kc_{m,\chi} e^{-\frac{1}{2}\sum\limits_{i=1}^nh_m^{(i)}\chi(x-x^{(i)}_m)^2}+o(1),
 \end{align}
 where $o(1)\rightarrow 0$ uniformly as $\chi \rightarrow \infty$.  In particular, if $m\in \mathcal I_b$, $c_{m,\chi}=c_{01}+O(\frac{1}{\sqrt{\chi}})$; if $m\in \mathcal I_s\backslash \mathcal I_b$, $c_{m,\chi}=c_{02}+O(\frac{1}{\sqrt{\chi}})$; if $m\not\in \mathcal I_s$, $c_{m,\chi}=O(\frac{1}{\sqrt{\chi}})$, where $c_{01}$ and $c_{02}$ are defined in (\ref{c0}).
\end{theorem}
\begin{remark}
As is shown in Theorem (\ref{thm11}), in contrast to many reaction-diffusion-advection models such as the minimal Keller--Segel model \cite{Keller1970} where the single interior spiky solution is unique \cite{Chen2014}, (\ref{ss}) has a variety of single interior spikes located at each local non-degenerate maximum point of $A$.
\end{remark}
\begin{figure}[h!]
\centering
\begin{subfigure}[t]{0.3\textwidth}
    \includegraphics[width=\linewidth]{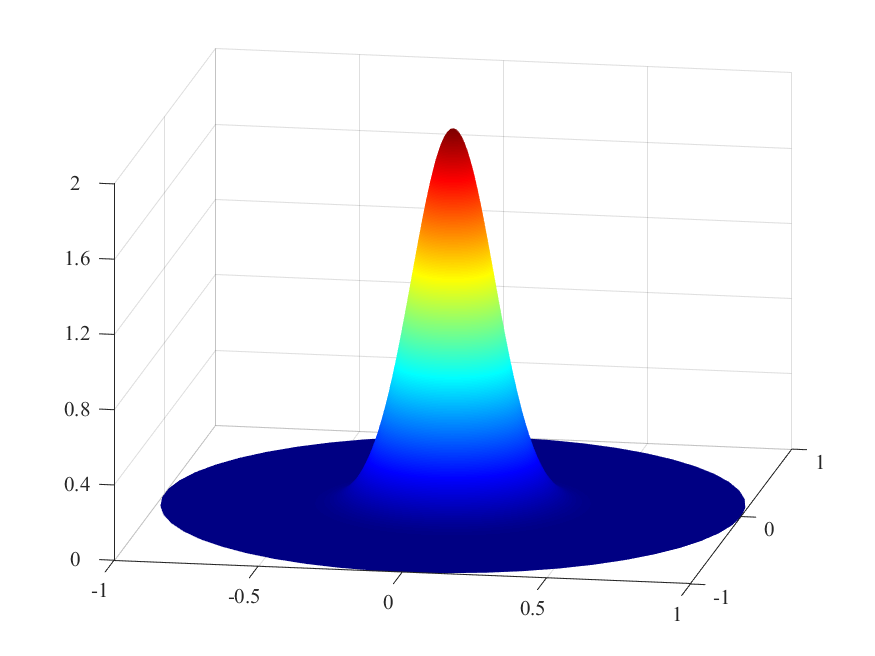}
\caption*{\textbf{(L-1)}.  Potential $A$}
\end{subfigure}\hspace{-0.25in}
\begin{subfigure}[t]{0.3\textwidth}
  \includegraphics[width=\linewidth]{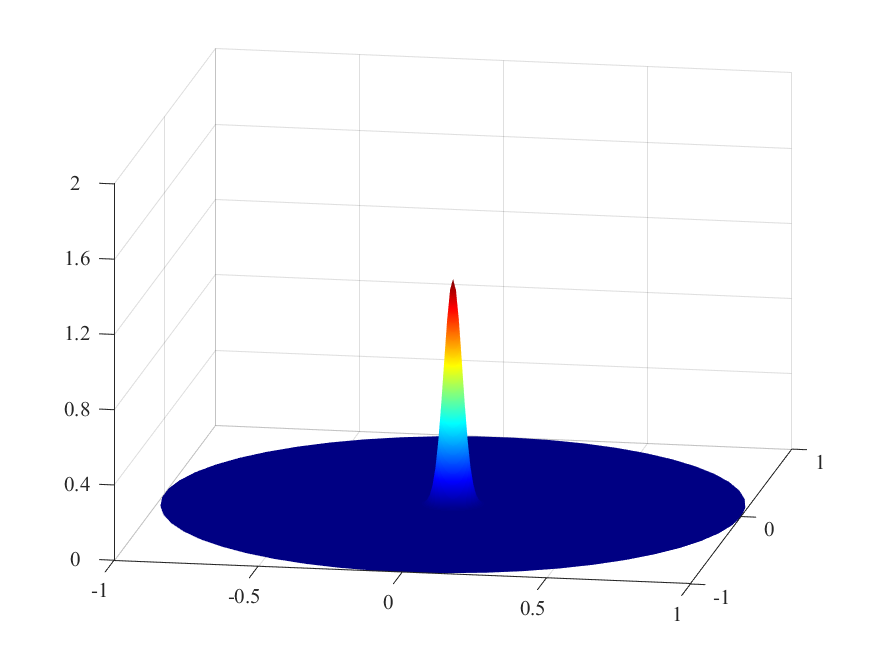}
\caption*{\textbf{(C-1)}.  Single interior spike with $c_{0}=1.2$}
\end{subfigure}\hspace{-0.25in}
\begin{subfigure}[t]{0.3\textwidth}
    \includegraphics[width=\linewidth]{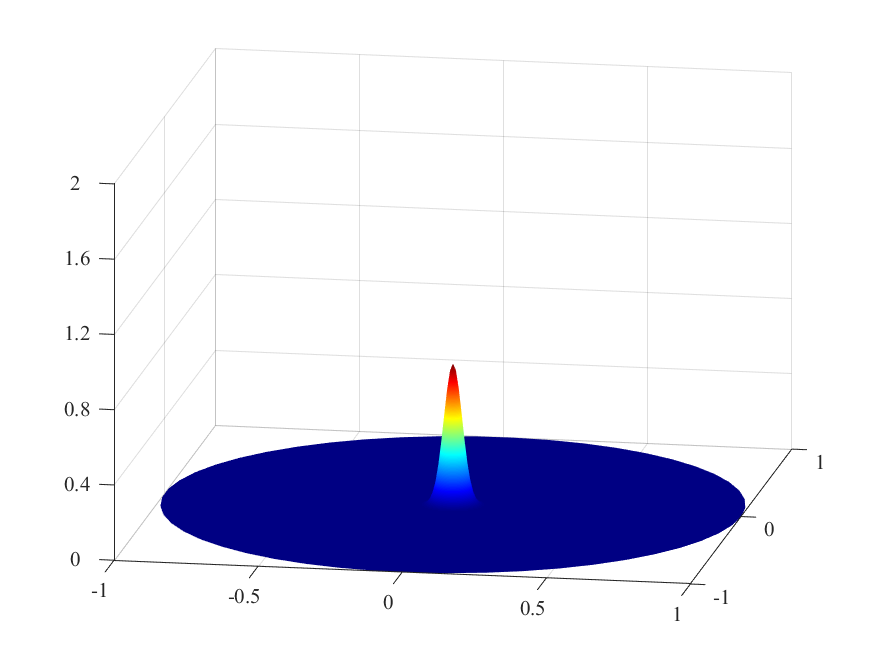}
\caption*{\textbf{(R-1)}.  Single interior spike with $c_{0}=0.75$}
\end{subfigure}

\begin{subfigure}[t]{0.3\textwidth}
    \includegraphics[width=\linewidth]{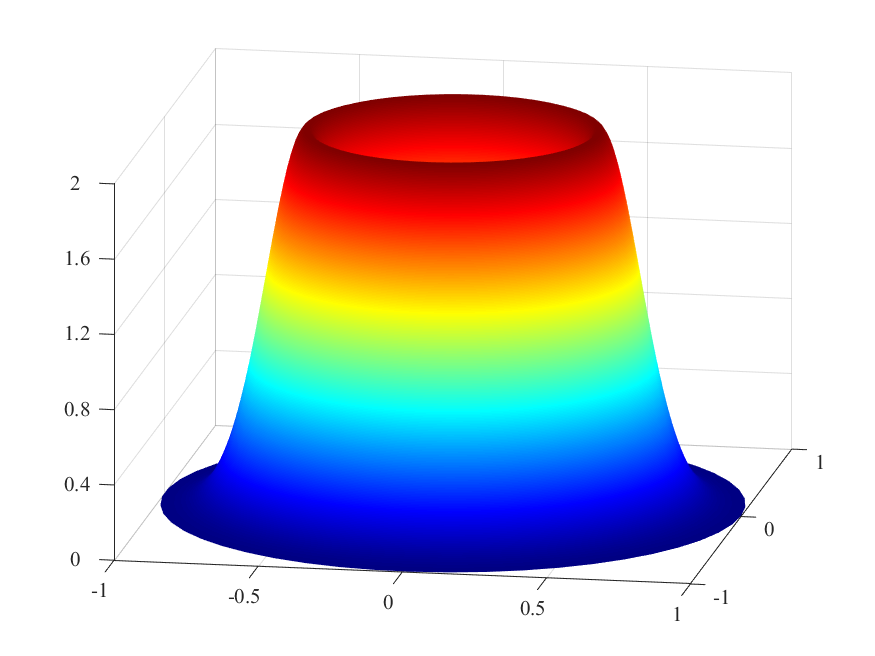}
\caption*{\textbf{(L-2)}.  Potential $A$}
\end{subfigure}\hspace{-0.25in}
\begin{subfigure}[t]{0.3\textwidth}
  \includegraphics[width=\linewidth]{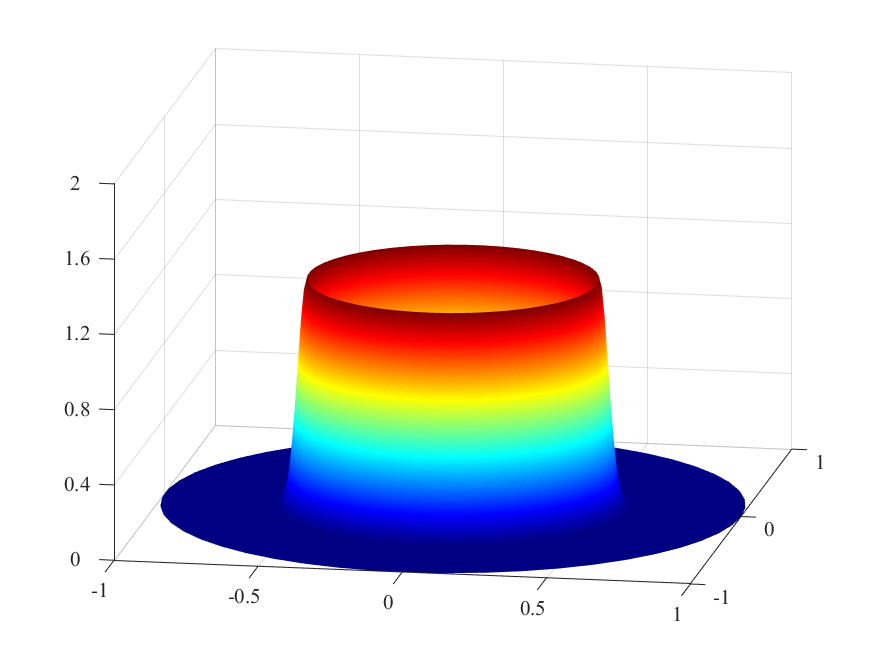}
\caption*{\textbf{(C-2)}.  Single interior spike with $c_{0}=1.2$}
\end{subfigure}\hspace{-0.25in}
\begin{subfigure}[t]{0.3\textwidth}
    \includegraphics[width=\linewidth]{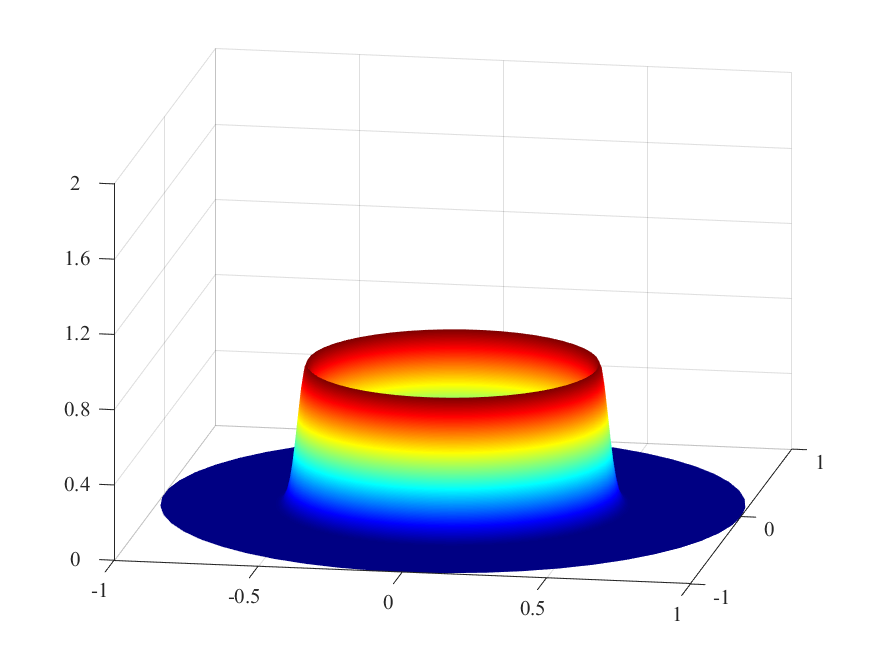}
\caption*{\textbf{(R-2)}.  Single interior spike with $c_{0}=0.75$}
\end{subfigure}
\caption{\textit{Top: Given signal $A:=2e^{-\frac{x^2+y^2}{\sigma^2}}$ with $\sigma=0.2$, the leading order profiles of single spikes defined in Theorem \ref{thm11} presented on the middle and the right with $\theta=0.3$ and $\chi=10$.  Bottom: Profiles of the signal and single interior spikes with the same condition except $A:=2e^{-\frac{\big(\sqrt{x^2+y^2}-0.5\big)^2}{\sigma^2}}$.}  We find that the shapes of spiky solutions follow those of signals $A$. }
\label{figure2}
\end{figure}
Theorem \ref{thm11} states that there are many single and multiple interior spikes to (\ref{ss}), and the asymptotic profiles of some one-spike solutions given by (\ref{u0}) are shown in the unit ball in 2-D in Figure \ref{figure2}.  Moreover, one can find from (\ref{u0}) that as $\chi\rightarrow \infty$, the ``inner regions" will shrink to zero exponentially.  It is natural to further study the large time behavior of the nontrivial steady states established by Theorem \ref{thm11}.  Our next results are devoted to investigate the linearized eigenvalue problem of (\ref{origin}) around the spiky solutions in (\ref{u0}), which are summarized as:
\begin{theorem}\label{thm12}
Under the conditions and conclusions shown in Theorem \ref{thm11}, we have when $\chi>\chi_0$ and $\theta\in(0,\theta_0)$, the following alternatives hold:
\begin{itemize}
\item[(i).] if $c_{m,\chi}=c_{01}+O\big(\frac{1}{\sqrt{\chi}}\big)$ for all $m\in \mathcal I_{s}$, the steady state $u_s(x;\chi)$ is linearly stable;
\item[(ii).] if $c_{m,\chi}=c_{02}+O\big(\frac{1}{\sqrt{\chi}}\big)$ for some $m$, the steady state $u_s(x;\chi)$ is unstable.
\end{itemize}
\end{theorem}
 
 Theorem \ref{thm12} shows that single and multiple interior spikes to (\ref{ss}) can be stable in some cases.  This is an intriguing phenomenon that is not discovered while studying the reaction-diffusion-advection systems without the source term such as for the minimal Keller--Segel model.  It is well known that the classical minimal chemotaxis model only admits the linearly stable single boundary spike solution, and whatever multiple boundary spikes or interior spikes are always unstable \cite{KW2021}.  The main reason why our results are different is that the strong Allee effect term in our models can help stabilize interior spikes in some cases.  It is worthwhile to point out that the Allee effect term is not necessary for the stabilization of interior spikes, other source terms such as logistic growth can also stabilize the interior spikes. 
 
Besides investigating the single equation, Cosner and Rodriguez \cite{Nancy} also discussed the dynamics of some coupled systems.  We are further motivated by their work and analyze the effect of strategies on the survival of interacting species.  It is worthy mentioning that there exists some interesting concept called ideal free distribution (IFD) in ecology.  It was introduced by Fretwell \cite{fretwell1969} in 1969 to describe how one species distribute individuals to minimize competition and maximize fitness.  If fitness is exactly determined by the amount of given resources, the theory states that under the following assumptions:
\begin{itemize}
    \item[(i).] Individuals in the species are homogeneous and equally able to access resources;
    \item[(ii).] Individuals are free to move in the environment;
    \item[(iii).] Organisms understand how to acquire the largest amount of resources and maximize fitness,
\end{itemize}
 the arrangement of individuals exactly matches the distribution of resources in the environment.  In general, the external resources are supposed to be located at several sites and form various aggregates, then homogeneous individuals will move towards these sites and distribute themselves among these patches of resources.  A direct consequence of this theory is that the fitness of all individuals in each aggregates should be the same at equilibrium since the individuals maximized their fitness and are not able to increase it by moving.  There are plenty of movement strategies that can lead to the ideal free distribution.  One of available strategies is that individuals aggregate in each patch with the number of population being proportional to the amount of available resources.

The IFD strategy can be modelled by the following equation with the no-flux boundary condition:
\begin{align*}
\left\{\begin{array}{ll}
v_t=\nabla\cdot (\nabla v- v\nabla \ln r)+v(r-v)(v-\theta), &x\in\Omega, t>0,\\
(\nabla v- v\nabla \ln r)\cdot\textbf{n}=0,&x\in\partial \Omega,t>0,\\
v(x,0)=v_0(x),&x\in\Omega,
\end{array}
\right.
 \end{align*}
where the external resources, modelled by $r$, are fixed.  In this equation, one finds that $v=r$ is an equilibrium, which implies the distribution of the species is the same as that of resources.  The other strategy what we are interested in is named as aggressive strategy, which is species combine unbiased and strong biased dispersal with high speed.  By studying the influence of IFD strategy and aggressive strategy on the persistence of interacting species, we obtain the following results summarized in Theorem \ref{thm13}:
 \begin{theorem}\label{thm13}
Assume all conditions in Theorem \ref{thm11} hold.  Focusing on system (\ref{14}), let $A(x)=\ln (r(x))$, $g(x,u+v):=(r(x)-u-v)(u+v-\theta)$, if $\chi_1=\chi$, $\chi_2=1$ and the boundary condition is no-flux, then when $\chi\gg 1$, for $\theta$ sufficiently small, we have there only exist two types of spiky steady states to (\ref{14}):
\[(u_s,v_s)=(0,r(x));~~(u_s,v_s)=(u^*,0),\]
where $u^*$ is given by 
\begin{align}\label{ustarthm13}
 u^*(x;\chi)=\sum_{m=1}^k\bar c_{m,\chi} e^{-\frac{1}{2}\sum\limits_{i=1}^nh_m^{(i)}\chi(x-x^{(i)}_m)^2}+O\bigg(\frac{1}{\sqrt{\chi}}\bigg).
 \end{align}
In particular, if $m\in \mathcal I_b$, $\bar c_{m,\chi}=\bar c_{01}+O(\frac{1}{\sqrt{\chi}})$; if $m\in \mathcal I_s\backslash \mathcal I_b$, $\bar c_{m,\chi}=\bar c_{02}+O(\frac{1}{\sqrt{\chi}})$; if $m\not\in \mathcal I_s$, $\bar c_{m,\chi}=O(\frac{1}{\sqrt{\chi}})$, where $\bar c_{01}$ and $\bar c_{02}$ are defined in (\ref{barc0}).

Concerning the stability property of $(u_s,v_s)$, we have when $\theta\in\big(0,\theta^{**}\big)$ with $\theta^{**}:=\min\Big\{\frac{\int_{\Omega}r^3 dx}{\int_{\Omega}r^2dx},\theta_1\Big\}$ and $\theta_1$ being defined in (\ref{theta1}), $(0,r)$ is linearly stable.  Moreover, if $\theta$ is independent of $\chi,$ the following alternatives hold for $(u^*,0)$:
\begin{itemize}
\item[(i).] if $\bar c_{m,\chi}=\bar c_{01}+O\big(\frac{1}{\sqrt{\chi}}\big)$ for all $m\in \mathcal I_{s}$, the solution will be linearly stable;
\item[(ii).] if $\bar c_{m,\chi}=\bar c_{02}+O\big(\frac{1}{\sqrt{\chi}}\big)$ for some $m$, the solution will be unstable;
\end{itemize}
however, if $\theta$ depends on $\chi$ and satisfies $\theta\in (0,\varepsilon^*_1/\chi^{\frac{n}{2}})$ with $\varepsilon_1^*$ defined by (\ref{varepsilonstar}), $(u^*,0)$ will be always unstable.  
\end{theorem}
We further investigate the case that two species both follow the aggressive strategy but one of them has higher speed, and then we have 
\begin{theorem}\label{thm14}
Assume $g(x,u+v):=(1-u-v)(u+v-\theta)$, $\chi_1=\chi$, $\chi_2=c\chi$ with $c>1$ and the boundary condition is no-flux in system (\ref{14}).  Moreover, suppose all conditions of $A(x)$ shown in Theorem \ref{thm11} hold.  Let $k$ be any positive but fixed integer, $\mathcal I_{s1}, \mathcal I_{s2}\subseteq \{1,\cdots,k\}$ be any subset and $\mathcal I_{b1}$, $\mathcal I_{b2}$ be any subset of $\mathcal I_{s1}$ and $\mathcal I_{s2}$, respectively.  Suppose at least one of subsets $\mathcal I_{s1}$ and $\mathcal I_{s2}$ is nonempty, then if $\mathcal I_{s1}\cap\mathcal I_{s2}=\emptyset$, we have when $\chi\gg 1$ and $\theta\in(0,\theta_1)$ with $\theta_1$ being defined in (\ref{theta1}), there exist the steady states $(\bar u^*,\bar v^*)$ defined as
\begin{align}\label{firstcoexistthm13}
\left\{\begin{array}{ll}
\bar u^*(x;\chi)=\sum\limits_{m=1}^k S^{(1)}_{m,\chi} e^{-\frac{1}{2}\sum\limits_{i=1}^nh_m^{(i)}\chi(x-x^{(i)}_m)^2}+o(1);\\
 \bar v^*(x;\chi)=\sum\limits_{m=1}^k S^{(2)}_{m,\chi} e^{-\frac{c}{2}\sum\limits_{i=1}^nh_m^{(i)}\chi(x-x^{(i)}_m)^2}+o(1),
 \end{array}
 \right.
 \end{align}
 where $o(1)\rightarrow 0$ uniformly as $\chi \rightarrow \infty$.  In particular, if $m\in\mathcal I_{bi},$ $S^{(i)}_{m,\chi}=c_{01}+O(\frac{1}{\sqrt{\chi}})$; if $m\in \mathcal I_{si}\backslash\mathcal I_{bi}$, $S^{(i)}_{m,\chi}=c_{02}+O(\frac{1}{\sqrt{\chi}})$; if $m\not\in\mathcal I_{si},$ $S^{(i)}_{m,\chi}=O(\frac{1}{\sqrt{\chi}})$, where $i=1, 2$ and $c_{01},$ $c_{02}$ are defined in (\ref{c0}).  Moreover, if $\mathcal I_{s1}\cap \mathcal I_{s2}\not=\emptyset$, (\ref{14}) admits many possible non-constant spiky solution $(\hat u^*,\hat v^*)$, which are given by 
 \begin{align}\label{secondcoexistthm13}
 \left\{\begin{array}{ll}
 \hat u^*(x;\chi)=\sum\limits_{m=1}^k \bar S^{(1)}_{m,\chi} e^{-\frac{1}{2}\sum\limits_{i=1}^nh_m^{(i)}\chi(x-x^{(i)}_m)^2}+o(1);\\
 \hat v^*(x;\chi)=\sum\limits_{m=1}^k \bar S^{(2)}_{m,\chi} e^{-\frac{c}{2}\sum\limits_{i=1}^nh_m^{(i)}\chi(x-x^{(i)}_m)^2}+o(1),
 \end{array}
 \right.
 \end{align}
 where $o(1)\rightarrow 0$ uniformly as $\chi \rightarrow \infty$.  In particular, if $m\in \mathcal I_{s1}\cap \mathcal I_{s2},$ $\bar S^{(i)}_{m,\chi}= S_i^*+O(\frac{1}{\sqrt{\chi}})$; if $m\in \mathcal I_{si}\backslash\{\mathcal I_{s1}\cap \mathcal I_{s2}\},$ $\bar S^{(i)}_{m,\chi}=c_{0i}+O(\frac{1}{\sqrt{\chi}})$; if $m\not\in\mathcal I_{si},$ $\bar S^{(i)}_{m,\chi}=O(\frac{1}{\sqrt{\chi}})$, where $i=1, 2$ and $S_{1}^*,$ $S_{2}^*$ are defined as the solution to
 \begin{align*}
\left\{\begin{array}{ll}
I_1(S_1,S_2)=0,\\
I_2(S_1,S_2)=0,
\end{array}
\right.
\end{align*}
where $I_1$ and $I_2$ are given by
$$I_1:=-3^{-\frac{n}{2}}S_1^2-2(c+2)^{-\frac{n}{2}}S_1S_2-(2c+1)^{-\frac{n}{2}}S_2^2+2^{-\frac{n}{2}}(1+\theta)S_1+(c+1)^{-\frac{n}{2}}(1+\theta)S_2-\theta$$
and
$$I_2:=-(c+2)^{-\frac{n}{2}}S_1^2-2(2c+1)^{-\frac{n}{2}}S_1S_2-(3c)^{-\frac{n}{2}}S_2^2+(c+1)^{-\frac{n}{2}}(1+\theta)S_1+(2c)^{-\frac{n}{2}}(1+\theta)S_2-c^{-\frac{n}{2}}\theta.$$

Regarding the stability of the steady states $(\bar u^*,\bar v^*)$, we have that if there exists $m$ such that $S^{(i)}_{m,\chi}=c_{02}+O(\frac{1}{\sqrt{\chi}})$, $i=1$ or $2$, the spiky solutions will be unstable; otherwise they will be linear stable.  However, there exist $c^*$ and $\tilde \theta^*$ such that for $c\in(1,c^*)$ and $\theta\in(0,\tilde \theta^*),$ the possible coexistence steady states $(\hat u^*,\hat v^*)$ are always unstable. 
\end{theorem}
 
Before clarifying the results obtained in Theorem \ref{thm13} and Theorem \ref{thm14}, we introduce the concept of neighborhood invader strategies.  In ecology, if one population taking some strategy is able to invade any other established population whose strategies are sufficiently similar to it, we define this strategy as a neighborhood invader strategy \cite{Nancy}.  The stability results in Theorem \ref{thm13} imply when Allee threshold $\theta$ depending on $\chi$ is small, the IFD strategy is a neighborhood invader strategy and the high speed does not always help the species to persist   since $(u^*,0)$ is unstable with $u^*$ being given by (\ref{ustarthm13}).  In addition, Theorem \ref{thm13} and Theorem \ref{thm14} demonstrate that interacting species who are concentrated at some patch of resources tend to occupy all resources rather than share with each other.  It is necessary to emphasize that the non-coexistence of competing species is nuanced in Theorem \ref{thm13} and Theorem \ref{thm14}.  Theorem \ref{thm13} shows that for Allee threshold sufficiently small, there is not any coexistence equilibrium when the competition is between the species who follows the aggressive strategy and the one who follows the IFD strategy.  Whereas Theorem \ref{thm14} illustrates that when the interacting species both use aggressive strategies, although there possibly exist the coexistence equilibria $(\hat u^*,\hat v^*)$ defined by (\ref{secondcoexistthm13}), which possess some local peaks occupied by both two species, they are always unstable.  Theorem \ref{thm14} also shows that there exist the spiky steady states $(\bar u^*,\bar v^*)$ given by (\ref{firstcoexistthm13}) and each peak of them is occupied by only one of interacting species.  The stability of these steady states still depends on the heights of local peaks. 
 
The remaining part of this paper is organized as follows.  In section \ref{sec2}, we utilize the finite dimensional Lyapunov-Schmidt reduction method to derive the stationary solutions (\ref{u0}) rigorously.  Our argument is divided into four steps.  The first one is to construct a good ansatz of $u$.  Next, we study the linearized problem to establish a priori estimate of the error term.  The last two steps concern the nonlinear projected problem and the reduced problem.  In section \ref{sec3}, we perform a theoretical analysis of linearized eigenvalue problem around (\ref{u0}).  We next qualitatively study the interaction of competing species modelled by the coupled systems in Section \ref{sec4}.  Section \ref{sec5} is devoted to some numerical simulations of the pattern formation within (\ref{origin}) and (\ref{14}), where the dynamics of single and some multiple interior spikes are presented.  
 \section{Construction of Interior Spikes}\label{sec2}
In section \ref{sec2}, we proceed to prove the existence of spiky solutions to system (\ref{ss}).  The study of boundary-layer and interior spikes arising from reaction-diffusion systems can be traced back to \cite{LNT,NT1,NT2}.  These pioneering researches have attracted intensive studies in the last two decades, see the references \cite{Gui1999,Gui2000} and therein.  For the Lyapunov-Schmidt reduction method, we refer to \cite{Wei19981,Wei1998} and the lecture notes \cite{delpinowei}.  There are rich applications of this reduction method in various reaction-diffusion system (see the book \cite{WeiWinterBook}).   To help readers understand our argument comprehensively, we shall present our main idea for the selection of the ansatz in the following subsection.
\subsection{Approximate Solution of $u$}
Recall the asymptotic form of $A$:
 \begin{align}\label{recall}
A=A_m-\frac{1}{2}\sum\limits_{i=1}^n h_m^{(i)}\big( x^{(i)}- x_m^{(i)}\big)^2+o\big(\vert x-x_m\vert^2\big),
\end{align}
where $x_m$, $m=1,2,\cdots,k$ are  non-degenerate local maximum points and $h_m^{(i)}>0$ for any $i=1,\cdots,n$.  In addition, when $d_1=1$, the stationary problem (\ref{ss}) can be rewritten as
\begin{align}\label{ss1}
\left\{\begin{array}{ll}
\nabla\cdot(e^{\chi(A-A_m)}\nabla(e^{-\chi(A-A_m)}u))+\mu u(1-u)(u-\theta)=0,&x\in\Omega,\\
\nabla(e^{-\chi(A-A_m)}u)\cdot \textbf{n}=0,&x\in\partial\Omega.
\end{array}
\right.
\end{align}
Letting $\chi:=\frac{1}{\epsilon^2}$, $y:=\frac{x}{\epsilon}$ and $u(x):=U(y)$, the $u$-equation in (\ref{ss1}) becomes
\begin{align}\label{21}
\nabla_y\cdot\bigg(e^{\frac{A-A_m}{\epsilon^2}}\nabla_y\Big(e^{-\frac{A-A_m}{\epsilon^2}}U\Big)\bigg)+\mu\epsilon^2 f(U)=0,\quad y\in \Omega_{\epsilon},
\end{align}
where $\Omega_{\epsilon}:=\{y\in \mathbb R^n|\epsilon y\in \Omega\}.$  In light of (\ref{recall}), $e^{\chi(A-A_m)}$ has the following form near $x_m$:
\[e^{\chi(A-A_m)}= e^{-\frac{1}{2}\sum\limits_{i=1}^nh_m^{(i)}\chi(x^{(i)}-x_m^{(i)})^2+o(\chi\vert x-x_m\vert^2)}.\]
Let the $m$-th centre of the interior spike be $P_m:=\frac{x_m}{\epsilon}$ and $a_m(y):=-\frac{1}{2}\sum\limits_{i=1}^nh_m^{(i)}\frac{(x^{(i)}-x_m^{(i)})^2}{\epsilon^2}=-\frac{1}{2}\sum\limits_{i=1}^nh_m^{(i)}(y^{(i)}-P_m^{(i)})^2$. One can rewrite $e^{\chi(A-A_m)}$ to obtain
\begin{align}\label{22}
e^{\chi(A-A_m)}=e^{a_m(y)}+o(\vert y\vert^2e^{a_m(y)}).
\end{align}
Substituting (\ref{22}) into (\ref{21}), we find
\begin{align}\label{23}
\nabla_y\cdot\bigg[\Big(e^{a_m(y)}+o\Big(\vert y\vert^2e^{a_m(y)}\Big)\Big)\nabla_y\Big[\Big(e^{-a_m(y)}+o\Big(\vert y\vert^2e^{a_m(y)}\Big)\Big)U\Big]\bigg]+\mu\epsilon^2 f(U)=0,\quad y\in \Omega_{\epsilon},
\end{align}
The limiting equation of (\ref{23}) becomes the following equation
\begin{align}\label{limiting}
\nabla_y\cdot(e^{a_m(y)}\nabla_y(e^{-a_m(y)}U))=0,\quad y\in \mathbb R^n.
\end{align}
We can see that there exists a non-constant solution $U_m(y)=c_me^{a_m(y)}$ to system (\ref{limiting}), where $c_m>0$ is constant.  Note that  our argument holds for every non-degenerate local maximum point $x_m$. So it is natural to consider the linear combination of $U_m$ as the approximate solution to (\ref{ss}).  To achieve our idea, we shall extract all local regions that only includes one of $x_m$ from $\Omega$, then truncate $U_m$ in every region and linearly combine them eventually.    

Under the hypothesis (H1), we first have the fact that there exist exactly $k$ mutually disjoint $\Omega_m:=\{x\in \Omega\vert \Delta A< \delta\}$ such that $x_m\in\Omega_m$.  Next, the following cut-off functions $\eta_m\in C^\infty_c(\Omega)$ are introduced:
\begin{align*}
\eta_m:=\left\{\begin{array}{ll}
1,&x\in\Omega_m,\\
0,&x\in \Omega\backslash B_{R_m}(0),
\end{array}
\right.
\end{align*}
where $\Omega_m\subset B_{R_m}(0)$ and $B_{R_m}(0)$ are mutually disjoint.  Now, we give the ansatz of $u$ which has the following form:
\begin{align}\label{ansatz}
u_0=\sum_{m=1}^k c_m e^{\chi(A-A_m)}\eta_m,
\end{align}
where $c_m>0$ need to be determined.  Define $u_0(x):=U_0(y)$ and $\phi(y):=U(y)-U_0(y)$. Next we discuss the property of the error term $\phi(y)$ in the following subsection. 
\subsection{Linearized Projected Problem}
After deriving the approximate solution of $u$, we are going to establish and analyze the linearized problem satisfied by $\phi$ in order to prove its existence rigorously.  It is necessary to introduce the following Hilbert space:
\[H_{N_*}^2(\Omega_{\epsilon})=\bigg\{U\in H^2(\Omega_{\epsilon})\Big\vert \bigg(\nabla_y U-\frac{1}{\epsilon^2} U\nabla_y A\bigg)\cdot \textbf{n}=0\text{~on~}\partial\Omega_{\epsilon}\bigg\}.\]
Then, define
\begin{align}\label{operator}
S_{\epsilon}(U)=\nabla_y\cdot\Big(\nabla_y U-\frac{1}{\epsilon^2} U\nabla_y A\Big)+\epsilon^2 f(U)
\end{align}
for $U\in H^2_{N_*}(\Omega_{\epsilon}).$  We have the fact that solving (\ref{ss}) is equivalent to find the solution of
$$S_{\epsilon}(U)=0\text{~~for some}~~ U\in H^2_{N_*}(\Omega_{\epsilon}).$$
The linearized operator of (\ref{operator}) is defined by
\begin{align}\label{29}
L_{\epsilon}(\phi):=\nabla_y\cdot\Big(\nabla_y \phi-\frac{1}{\epsilon^2}\phi\nabla_yA\Big).
\end{align}
We denote $Z_m$ for $m=1,\cdots,k$ as
\begin{align}\label{kernel}
Z_m=e^{-\frac{1}{2}\sum\limits_{i=1}^n h_m^{(i)}(y^{(i)}-P_m^{(i)})^2}=e^{a_m(y)}.
\end{align}
Before stating that $Z_m$ are the approximate kernels of $L_{\epsilon}$, one needs to analyze the property of the operator $L_{\epsilon}(\phi)$ near every $x_m$ for $\epsilon$ small.  In fact, define $y_m:=\frac{ x-x_m}{\epsilon}$ and
\[L_{0}^m(\phi):=\nabla_{y_m}\cdot\Big(e^{a_m(y_m)}\nabla_{y_m}(e^{-a_m(y_m)}\phi)\Big),\]
then we have the following classification result on the kernel of $L_0^m(\phi)$:
\begin{proposition}\label{proposition21}
For each $m=1,2,\cdots,k$, the bounded solution to the following problem
\begin{align}\label{limitingoperator}
\left\{\begin{array}{ll}
L_0^m(\phi)=0,&y_m\in \mathbb R^n,\\
\phi\in H^2_0(\mathbb R^n), &\vert \phi\vert=O(1)e^{-\sigma\vert y_m\vert^2}
\end{array}
\right.
\end{align}
is unique and given by (\ref{kernel}), where $\sigma>0$ is a small constant.
\end{proposition}
\begin{proof}
Define $\hat \phi(y_m):=e^{-a_m(y_m)}\phi$.  The $\phi-$equation becomes
\begin{align}\label{211}
\nabla\cdot\Big(e^{a_m(y_m)}\nabla \hat\phi\Big)=0,~~~y_m\in \mathbb R^n.
\end{align}
Define $\hat\phi_{+}$ and $\hat\eta$ as
\begin{align*}
\hat\phi_{+}=\left\{\begin{array}{ll}
\hat\phi,  \hat\phi>0,\\
0, \hat\phi\leq 0,
\end{array}
\right.\text{~~~and~~~}\hat \eta=\left\{\begin{array}{ll}
1,&y_m\in B_{R}(0),\\
0,&y_m\in \mathbb R^n\backslash B_{2 R}(0),
\end{array}
\right.
\end{align*}
where $R>0$ is a constant.  Multiply (\ref{211}) by $ (\hat\phi_+)^{N}\hat\eta^2$ with $N$ being determined later on, then integrate it over $\mathbb R^n$ to get
\begin{align*}
0=&\int_{\mathbb R^n}\nabla\cdot(e^{a_m}\nabla \hat \phi)(\hat\phi_+)^{N}\hat\eta^2dy_m=-\int_{\mathbb R^n}(e^{a_m}\nabla \hat \phi)\cdot\nabla((\hat\phi_+)^{N}\hat\eta^2)dy_m,
\end{align*}
which implies
\begin{align*}
0=&-\int_{\mathbb R^n}e^{a_m}\nabla \hat \phi \cdot \nabla ((\hat \phi_{+})^N)\hat\eta^2dy_m-\int_{\mathbb R^n}(e^{a_m}\nabla \hat \phi)\cdot(\hat\phi_+)^{N}\nabla(\hat\eta^2)dy_m;
\end{align*}
moreover, since the support of $\hat\eta$ is $B_{2R}(0)$, one has
\begin{align*}
\int_{B_{2R}(0)}e^{a_m}\nabla \hat \phi \cdot \nabla ((\hat \phi_{+})^N)\hat\eta^2dy_m=-\int_{\mathbb R^n}(e^{a_m}\nabla \hat \phi)\cdot(\hat\phi_+)^{N}\nabla(\hat\eta^2)dy_m.
\end{align*}
By using integration by parts, we can obtain
\begin{align*}
\frac{4N}{(N+1)^2}\int_{B_{2R}(0)}e^{a_m}\big\vert\nabla \big((\hat\phi_+)^{\frac{N+1}{2}}\big)\big\vert^2dy_m=\frac{1}{N+1}\int_{B_{2R}(0)}(\hat\phi_+)^{N+1}\nabla\cdot(e^{a_m}\nabla\hat\eta^2)dy_m.
\end{align*}
Therefore,
\begin{align}\label{212}
\int_{B_{2R}(0)}e^{a_m}\big\vert\nabla \big((\hat\phi_+)^{\frac{N+1}{2}}\big)\big\vert^2dy_m\leq C\int_{B_{2R}(0)}\vert\hat\phi_{+}\vert^{N+1}\cdot \vert\nabla\cdot(e^{a_m}\nabla\hat\eta^2)\vert dy_m,
\end{align}
where $C>0$ is some large constant.  In light of $\vert \hat\phi\vert\leq C_1 e^{-a_m(y_m)} e^{-\sigma\vert y_m\vert^2}$ for small $\sigma>0$ and some constant $C_1>0$, one finds from (\ref{212}) that there exists some constant $C_2>0$ such that
\begin{align}\label{referee21}
\int_{B_{2R}(0)}e^{a_m}\big\vert\nabla_{y_m} \big((\hat\phi^+_m)^{\frac{N+1}{2}}\big)\big\vert^2dy_m\leq C_2(2R)^n\Big(e^{-Na_m(y_m)}\vert_{\vert y_m\vert=2R}\Big)e^{-\sigma(N+1)( 2R)^2}.
\end{align}
Recall that $a_m(y_m):=-\frac{1}{2}\sum\limits_{i=1}^nh_m^{(i)}[y^{(i)}_m]^2,$ then we have 
\begin{align*}
\Big(e^{-Na_m(y_m)}\vert_{y_m=2R}\Big)\leq e^{C_3N\vert y_m\vert^2}\vert_{\vert y_m\vert=2R}\leq e^{C_3N(2R)^2},
\end{align*}
where $C_3>0$ is some constant only depending on the maximum of $h_m^{(i)},$ $i=1,\cdots,n$.  We choose $N=\frac{\sigma}{C_3}$ then obtain from \eqref{referee21} that
\begin{align*}
\int_{B_{2R}(0)}e^{a_m}\big\vert\nabla_{y_m} \big((\hat\phi^+_m)^{\frac{N+1}{2}}\big)\big\vert^2dy_m\leq C_2(2R)^ne^{-\sigma N( 2R)^2}=C_2(2R)^ne^{-\frac{\sigma^2}{C_3} ( 2R)^2}\rightarrow 0\text{~as~}R\rightarrow\infty.
\end{align*}
where $\sigma$ can be chosen small enough.  Thus, we have $\hat \phi_+\equiv C_4$ for $y_m\in\mathbb R^n$, where $C_4>0$ is a constant. It is similar to multiply (\ref{211}) by $\hat \phi_{-}\hat\eta^2$ to obtain $\hat\phi_{-}\equiv C_5$ for some constant $C_5>0$.  Therefore, we prove there exists a unique bounded solution to (\ref{limitingoperator}) for each $m$.
\end{proof}
Proposition \ref{proposition21} implies that when $\epsilon$ is small, the null space of operator $L_{\epsilon}(\phi)$ can be expressed as $\text{span}\{Z_1,\cdots,Z_k\}$, where $Z_m$, $m=1,\cdots,k$ are given by (\ref{kernel}).  We would like to mention that Lam \cite{lam2012} obtained the weaker version of Proposition \ref{proposition21} in Appendix.  Now, we establish the linearized projected problem satisfied by $\phi$:
\begin{align}\label{project}
\left\{\begin{array}{ll}
L_{\epsilon}(\phi)=\nabla_y\cdot\big(\nabla_y \phi-\frac{1}{\epsilon^2}\phi\nabla_yA\big)=g-\sum\limits_{m=1}^kd_mZ_m,&y\in\Omega_{\epsilon},\\
\int_{\Omega_{\epsilon}}\phi Z_mdy=0,&m=1,2,\cdots,k,\\
\phi\in H^2_{N_*}(\Omega_{\epsilon}).
\end{array}
\right.
\end{align}
Given $\sigma_1\in(0,1)$, consider the $*$-norm
\[\Vert g\Vert_{*}=\sup_{y\in\Omega_{\epsilon}}\bigg\vert\bigg(\sum\limits_{m=1}^ke^{\sigma_1 a_m(y)}\bigg)^{-1} g(y)\bigg\vert,\]
then we shall estimate $\phi$ and prove its existence, and the result is summarized as the following proposition:  

\begin{proposition}\label{proposition22}
Assume $g\in L^2(\Omega_{\epsilon})$ and $\Vert g\Vert_{*}$ is bounded.  Then there exist positive constant $\sigma_1\in (0,1)$, $\epsilon_0$ and $C$ such that for any $\epsilon\in(0,\epsilon_0]$, (\ref{project}) admits the unique solution $(\phi,\{d_m\})$.  Moreover, $\phi$ satisfies the following a priori estimate: 
\begin{align}\label{phiestimate}
\Vert \phi\Vert_{*}\leq C\Vert g\Vert_{*}.
\end{align}
\end{proposition}

\begin{proof}
We divide our proof into two steps.  First of all, one estimates the error term $\phi$ with the assumption that $(\phi,\{d_m\})$ exists.  Then we show the well-posedness of system (\ref{project}) including existence and uniqueness by invoking Fredholm alternative theorem.\\

\textbf{Step A-a priori estimates:}\\
To this end, we define $\Omega_{m_\epsilon}:=\{y\in \Omega_{\epsilon}\vert \Delta_y A<\epsilon^2\delta\}$ and $\hat \eta_m$ which has the following form:
\begin{align*}
\hat\eta_m=\left\{\begin{array}{ll}
1,&y\in \Omega_{m_\epsilon},\\
0,&y\in\Omega_{\epsilon}\backslash B_{R_{m,\epsilon}}(P_m),
\end{array}
\right.
\end{align*}
where $\Omega_{m_{\epsilon}}\subset B_{R_{m,\epsilon}}(P_m)$ for some constant $R_{m,\epsilon}>0$ depending on $\epsilon.$  Then, multiply the $\phi$-equation in (\ref{project}) by $\hat \eta_m$ and integrate it over $\Omega_{\epsilon}$ to find
\begin{align}\label{dm1}
\int_{\Omega_{\epsilon}} (L_{\epsilon}(\phi)-g)\hat\eta_mdy=-d_m\int_{\Omega_{\epsilon}} Z_m\hat \eta_mdy-\sum_{j\not= m}d_{j}\int_{\Omega_{\epsilon}}Z_j\hat\eta_mdy.
\end{align}
We simplify the left hand side to obtain
\begin{align*}
&\int_{\Omega_{\epsilon}}\Big[\nabla\cdot\Big(\nabla\phi-\frac{1}{\epsilon^2}\phi\nabla A\Big)\Big]\hat \eta_mdy-\int_{\Omega_{\epsilon}}g\hat\eta_mdy=\int_{\Omega_{\epsilon}}\nabla\cdot\Big(e^{a_m(y)}\nabla \Big(e^{-a_m(y)}\phi\Big)\Big)\hat\eta_mdy-\int_{\Omega_{\epsilon}}g\hat\eta_mdy,
\end{align*}
which, after we apply the integration by parts, becomes
\begin{align*}
\int_{\Omega_{\epsilon}}\Big[\nabla\cdot\Big(\nabla\phi-\frac{1}{\epsilon^2}\phi\nabla A\Big)\Big]\hat \eta_mdy-\int_{\Omega_{\epsilon}}g\hat\eta_mdy=&-\int_{\Omega_{\epsilon}}(e^{a_m}\nabla (e^{-a_m}\phi))\cdot\nabla\hat\eta_mdy-\int_{\Omega_{\epsilon}}g\hat\eta_mdy,\\
=&\int_{{B_{R_{m,\epsilon}}(P_m)}}\phi\Delta\hat\eta_m+\phi\nabla a_m\cdot\nabla \hat\eta_m dy-\int_{\Omega_{\epsilon}}g\hat\eta_mdy,
\end{align*}
where the last equality holds since the support of $\hat\eta_m$ is $B_{R_{m,\epsilon}}(P_m)$.  Hence, we have from (\ref{dm1}) that
\begin{align}\label{dmexpress}
d_m=\frac{\int_{B_{R_{m,\epsilon}}(P_m)}gdy-\int_{B_{R_{m,\epsilon}}(P_m)\backslash \Omega_{m_\epsilon}}(\phi\Delta\hat\eta_m+\phi\nabla a_m\cdot\nabla\hat\eta_m )dy+\sum\limits_{j\not= m}d_{j}\int_{B_{R_{m,\epsilon}}(P_m)}Z_jdy}{\int_{B_{R_{m,\epsilon}}(P_m)} Z_mdy}.
\end{align}
In addition, when $j\not=m$, we find that
\begin{align}\label{217}
\int_{B_{R_{m,\epsilon}(P_m)}}Z_jdy=\int_{B_{R_{m,\epsilon}(P_m)}}e^{-\frac{1}{2}\sum\limits_{i=1}^n h_j^{(i)}(y^{(i)}-P_j^{(i)})^2}dy\leq C e^{-\frac{1}{2}\tilde \sigma\vert P_m-P_j \vert^2},
\end{align}
where $\tilde\sigma>0$ is small and $C>0$ is a constant.  Since $\vert P_m-P_j\vert=O\big(\frac{1}{\epsilon}\big)$ for $m\not=j$, one can conclude from (\ref{dmexpress}) and (\ref{217}) that
\begin{align}\label{dmestimate}
\vert d_m\vert\leq C_1\Vert g\Vert_{*}+\epsilon^nC_1\Vert \phi\Vert_{*}+C_1\sum_{j\not=m}\vert d_j\vert e^{-\frac{1}{2}\tilde \sigma\vert P_m-P_j \vert^2}.
\end{align}
With the help of (\ref{dmestimate}), we now discuss the estimates on the ``inner" and ``outer" region of every $P_m$, $m=1,2,\cdots,k$.

First of all, we study the property of $\phi$ in the outer region $\Omega_{\epsilon}\backslash \cup_{m=1}^k B_{R}(P_m)$ where $R$ is  a large constant independent of $\epsilon$. To this end, we construct the barrier function $w=\hat\mu e^{\sigma_2\frac{\bar A}{\epsilon^2}}$, where $\hat\mu>0$ will be chosen later, $\sigma_2>0$ is a constant and
\[\bar A=\sum\limits_{m=1}^k(A-A_m)\hat\eta_m-C(1-\hat\eta_1)\cdots(1-\hat\eta_k).\]
It is straightforward to verify that when $\epsilon$ is small,
\begin{align*}
L_{\epsilon}w=&\nabla_y\cdot\bigg(\nabla_y w-\frac{1}{\epsilon^2}w\nabla_yA\bigg)\\
=&w\bigg[\frac{\sigma^2_2}{\epsilon^2}\vert\nabla_y\bar A\vert^2-\frac{\sigma_2}{\epsilon^2}\nabla_y A\cdot\nabla_y\bar A+\sigma_2\Delta_y \bar A-\Delta_y A\bigg]+O(\epsilon^2)w.
\end{align*}
Now, we claim that there exists a constant $C>0$ such that $L_{\epsilon} w\leq -Cw$ holds in $\Omega_{\epsilon}\backslash\cup_{m}B_{R}(P_m)$ for small $\epsilon$.  Indeed, if $y\in\Omega_{m_\epsilon}$ for some $m$, then there exists some constant $C_2>0$ such that $\vert y-P_m\vert>C_2,$ which implies
\[\nabla_y\cdot\bigg(\nabla_y w-\frac{1}{\epsilon^2}w\nabla_yA\bigg)\\
=w\bigg[\frac{\sigma_2}{\epsilon^2}(\sigma_2-1)\vert\nabla_y A\vert^2+(\sigma_2-1)\Delta_y A\bigg]\leq -C_3w,\]
where $C_3>0$ is a constant; if $y\not\in\Omega_{m_\epsilon}$ for all $m=1,2,\cdots,k$, then we have $\Delta A>\delta>0$. It follows that
\begin{align*}
w\bigg[\frac{\sigma^2_2}{\epsilon^2}\vert\nabla_y\bar A\vert^2-\frac{\sigma_2}{\epsilon^2}\nabla_y A\cdot\nabla_y\bar A+\sigma_2\Delta_y \bar A-\Delta_y A\bigg]\leq -C_4w,
\end{align*}
where $C_4>0$ is a constant.  We collect the above arguments and hence finish the proof of our claim.  We further define $\phi_1:=\phi-\hat Cw$ for a large constant $\hat C>0$ and obtain from the claim that
\begin{align*}
L_{\epsilon}(\phi_1)=\nabla_y\cdot \Bigg(\nabla_y\phi_1-\frac{1}{\epsilon^2}\phi_1\nabla_yA\Bigg)>0,
\end{align*}
where $y\in\Omega_{\epsilon}\backslash\cup_{m=1}^kB_{R}(P_m).$  
In addition, due to the hypothesis (H2), we have that the boundary condition of $\phi_1$ satisfies
\[\partial_{\textbf{n}}\phi_1-\frac{\phi_1}{\epsilon^2}\partial_{\textbf{n}} A=\tilde C\frac{w}{\epsilon^2}\nabla_y A\cdot\textbf{n}<0\text{~on~}\partial\Omega_{\epsilon},\]
where $\tilde C>0$ is a constant.  Moreover, one lets $\hat\mu=\sup_{m=1,\cdots,k}\Vert \phi\Vert_{L^\infty(B_{R}(P_m)\cap\Omega_{\epsilon})}+\Vert g\Vert_{*}+\sum_{m=1}^k\vert d_m\vert\cdot\Vert Z_m\Vert_{*}$ and applies the maximum principle to $\frac{\phi_1}{w}$ to get $\phi_1\leq 0$ pointwisely on $\Omega_{\epsilon}\backslash\cup_{m=1}^kB_{R}(P_m)$.  As a consequence, $\phi\leq \hat Cw$ on $\Omega_{\epsilon}\backslash\cup_{m=1}^kB_{R}(P_m)$.  Similarly, we can show $-\phi\leq \hat Cw$.  Combining the above two estimates, we have $\vert \phi\vert\leq \hat Cw$ for some constant $\hat C>0$ on $\Omega_{\epsilon}\backslash\cup_{m=1}^kB_{R}(P_m)$.


We are going to prove (\ref{phiestimate}) via the contradiction argument.  Assume that there exist a sequence $\epsilon_n\rightarrow 0$ such that the corresponding sequence $\phi_n$ and $g_n$ satisfy
\[\Vert g_n\Vert_{*}\rightarrow 0,\text{~~}\Vert\phi_n\Vert_{*}=1.\]
Moreover, one obtains from (\ref{dmestimate}) that $d_m^n\rightarrow 0$ for any $m=1,2,\cdots,k$ as $n\rightarrow \infty$.  Noting that $\Vert\phi_n\Vert_{*}$ and $\Vert g\Vert_{*}$ are bounded, we conclude $\phi_n\in C^{1,\gamma}(B_R(P_m))$ thanks to the standard elliptic estimate.  On the other hand, note that $\phi_n$ satisfies the following equation:
$$\Delta_y \phi_n-\frac{1}{\epsilon^2}\nabla_y\phi_n\cdot\nabla_y A-\frac{1}{\epsilon^2}\phi_n\Delta_y A=g_n-\sum\limits_{m=1}^kd_m^nZ_m.$$
Hence, $\vert \Delta \phi_n\vert \leq C_7$ for some constant $C_7>0$.  It follows that we can extract a sub-sequence $\phi_{n_j}\rightarrow \phi_{\infty}$ as $j\rightarrow \infty$ such that the limit $\phi_{\infty}$ is a solution of the following equation:
\begin{align*}
\left\{\begin{array}{ll}
\nabla_{y_m}\cdot\Big(e^{a_m(y_m)}\nabla_{y_m}(e^{-a_m(y_m)}\phi_{\infty})\Big)=0,&y_m\in\mathbb R^n,\\
\int_{\mathbb R^n}\phi_{\infty} Z_mdy=0,&m=1,2,\cdots,k,\\
\end{array}
\right.
\end{align*}
where $y_m=\frac{x-x_m}{\epsilon}$.  Moreover, by invoking Proposition \ref{proposition21}, we can show $\phi_{\infty}\in \text{span}\{Z_1,Z_2,\cdots,Z_k\}$ and satisfies the orthogonality condition, which imply $\phi_{\infty}\equiv 0$ in $\cup_{m=1}^k\bar B_{R}(P_m)$.  

We next estimate $\phi_\infty$ in the outer region $\Omega_{\epsilon}\backslash\cup_{m=1}^k\bar B_{R}(P_m)$.
To this end, we first find $\hat\mu_n\rightarrow 0$ as $n\rightarrow \infty$ since $\Vert g_n\Vert_{*}$, $d_m^n$ and $\Vert \phi_n\Vert_{L^\infty(B_{R}(P_m)\cap\Omega_{\epsilon})}$ vanish for any $m=1,2,\cdots,k$.  Moreover, noting that there exists $C_8>0$ such that $\vert\phi_n\vert\leq C_8 w_n$ in $\Omega_{\epsilon}\backslash\cup_{m=1}^k\bar B_{R}(P_m)$, we conclude from $\hat\mu_n\rightarrow 0$ that $\phi_n\rightarrow 0$ in $\Omega_{\epsilon}\backslash \cup_{m=1}^k B_{R}(P_m)$. 

In summary, $\Vert\phi_n \Vert_{*}\rightarrow 0$ in $\Omega_{\epsilon}$ as $n\rightarrow \infty$, which reaches the contradiction with $\Vert \phi_n\Vert_{*}\equiv 1,$ $\forall n=1,2,\cdots$. This completes the proof of (\ref{phiestimate}).\\

\textbf{Step B-Existence of $\phi$:}\\
Let
\[X=\Big\{\phi\in H_{N_{*}}^2(\Omega_{\epsilon}):\int_{\Omega_{\epsilon}}\phi Z_mdy=0,\text{~~}m=1,2,\cdots,k\Big\}.\]
Then the $\phi$-equation in (\ref{project}) can be rewritten as
\begin{align}\label{compact}
\phi+B(\phi)=\tilde g\text{~~in~~} X,
\end{align}
where $\tilde g$ is defined by duality and $B: X\rightarrow X$ is a linear compact operator.  With the aid of Fredholm alternative theorem, it suffices to show that (\ref{compact}) admits the unique solution for $\tilde g=0$ so as to obtain the existence.  This statement follows from the discussion in Step A.  Therefore, we conclude Proposition \ref{proposition22} holds.
\end{proof}
Proposition \ref{proposition22} implies that there exists an invertible operator $\mathcal A$ such that $\phi=\mathcal A(g)$ and 
$$\Vert \mathcal A(g)\Vert_{*}\leq \bar C\Vert g\Vert_{*},$$
where $\bar C$ is a positive constant.  Now, we are well-prepared to solve the nonlinear equation satisfied by $\phi.$
\subsection{Nonlinear Projected Problem}
In this subsection, the contraction mapping theorem will be employed to show the existence of the solution $\phi$ to the following nonlinear projected problem:
\begin{align}\label{nonlinear}
\left\{\begin{array}{ll}
L_{\epsilon}(\phi)+E+N(\phi)=-\sum\limits_{m=1}^kd_mZ_m,\\
\int_{\Omega_\epsilon}\phi Z_mdy=0,\text{~for~}m=1,\cdots,k,\\
\phi\in H_{N_*}^2(\Omega_\epsilon),
\end{array}
\right.
\end{align}
where $E$ is the error of the approximate solution $U_0$ and $N(\phi)$ is the nonlinear term, which are defined by
\begin{align}\label{errore}
E=\nabla_{y} \cdot\bigg(\nabla_{y} U_0-\frac{1}{\epsilon^2} U_0\nabla_{y} A\bigg)\text{~~~and~~~} N(\phi)=\mu\epsilon^2f(U_0+\phi).
\end{align}
In light of Proposition \ref{proposition22}, the $\phi$-equation in (\ref{nonlinear}) can be rewritten as
\[\phi=T(\phi):=-\mathcal A(E+N(\phi)),\]
where $\mathcal A$ is an invertible operator.  To prove the existence, it is equivalent to showing that  $\phi$ is a fixed point of the operator $T$.  We restrict $\phi$ into the following Hilbert space:
\[\mathcal H:=\Big\{\phi\in H_{N_*}^2(\Omega_\epsilon):\Vert \phi\Vert_{*}\leq \bar C_1\epsilon,\text{~}\int_{\Omega_\epsilon}\phi Z_mdy=0\Big\},\]
where constant $\bar C_1>0$ will be determined later on.  We use the contraction mapping theorem to derive the results which are summarized as the following proposition.
\begin{proposition}\label{proposition23}
There exist $\sigma_3\in(0,1)$, $\epsilon_0>0$ and large constants $\bar C_2, \bar C_3>0$ such that for all $\epsilon\in(0,\epsilon_0]$, the following estimates hold:
\begin{align}\label{estar}
\Vert E\Vert_{*}\leq \bar C_2\epsilon\text{~~~and~~~}\Vert N(\phi)\Vert_{*}\leq \bar C_3\epsilon^2,
\end{align}
where $\phi\in\mathcal H.$  Moreover, (\ref{nonlinear}) admits a unique solution $\phi$ which satisfies
\begin{align}\label{contractionestimate}
\Vert\phi\Vert_{*}\leq \bar C_1\epsilon.
\end{align}
\end{proposition}
\begin{proof}
We first compute the approximate error $E$ defined by (\ref{errore}).  Since the ansatz of $u$ satisfies (\ref{ansatz}), we have
\begin{align}\label{225}
E=&\sum_{m=1}^k\nabla_{y_m} \cdot\bigg(\nabla_{y_m} U_0-\frac{1}{\epsilon^2} U_0\nabla_{y_m} A\bigg)=\sum_{m=1}^k c_m\nabla_{y_m}\cdot(e^{a_m(y_m)}\nabla_{y_m}\hat\eta_m)\\
=&\sum_{m=1}^kc_m[e^{a_m(y_m)}\epsilon\nabla_{x} A\cdot\nabla_{y_m} \hat\eta_m+ e^{a_m(y_m)}\Delta_{y_m}\hat\eta_m]\nonumber=O(1){\epsilon}\sum_{m=1}^ke^{-h_m\vert y_m\vert^2},
\end{align}
where $h_m>0$ are constants.  It follows that there exists $\bar C_2>0$ such that $\Vert E\Vert_{*}\leq \bar C_2\epsilon.$  Next, we estimate $N(\phi)=\mu\epsilon^2f(U_0+\phi)$, where $\phi\in \mathcal H.$  Note that
\begin{align*}
\vert N(\phi)-\mu\epsilon^2f(U_0)\vert\leq \bar C_4\epsilon^2\vert\phi\vert,
\end{align*}
where $\bar C_4>0$ is some constant, then one has
\begin{align*}
\vert N(\phi)\vert \leq &\vert N(\phi)-\mu\epsilon^2f(U_0)\vert+\mu\epsilon^2\vert f(U_0)\vert\\
\leq &\bar C_4\epsilon^2\vert\phi\vert+\mu\epsilon^2\vert f(U_0)\vert\\
\leq &\bar C_4\epsilon^2\vert\phi\vert+\mu\epsilon^2 \vert U_0(1-U_0)(U_0-\theta)\vert \\
\leq &\bar C_4\epsilon^2\vert\phi\vert+\mu\epsilon^2 (U^2_0+U^3_0+\theta U_0+\theta U^2_0)\\
\leq &\bar C_4\epsilon^2\vert\phi\vert+O(1)\epsilon^2\sum_{m=1}^k e^{- h_m\vert y_m\vert^2},
\end{align*}
which implies 
\begin{align}\label{NNN}
\Vert N(\phi)\Vert_{*} \leq \bar C_5\epsilon^2\Vert\phi\Vert_{*}+\bar C_5\epsilon^2,
\end{align}
where $\bar C_5>0$ is some constant.

According to Proposition \ref{proposition22}, we have there exists some constant $\bar C>0$ such that
\begin{align*}
\Vert T(\phi)\Vert_{*}\leq& \bar C\Vert E+N(\phi)\Vert_{*}\\
\leq&\bar C(\Vert E\Vert_{*}+\Vert N(\phi) \Vert_{*}).
\end{align*} 
We shall show the operator $T$ is a mapping from $\mathcal H\rightarrow \mathcal H.$  By choosing $\bar C_1:=4\bar C\bar C_2$, we invoke \eqref{225} and \eqref{NNN} to obtain for any $\phi\in \mathcal H$,
\begin{align*}
\Vert T(\phi)\Vert_{*}\leq&\bar C(\Vert E\Vert_{*}+\Vert N(\phi) \Vert_{*})\\
\leq &\bar C(\bar C_2\epsilon+\bar C_5\epsilon^2\Vert\phi\Vert_{*}+\bar C_5\epsilon^2)\leq 2\bar C\bar C_2\epsilon+\bar C\bar C_5\epsilon^2\Vert\phi\Vert_{*}\\
\leq &2\bar C\bar C_2\epsilon+4\bar C^2\bar C_2\bar C_5\epsilon^3\leq 4\bar C\bar C_2\epsilon=\bar C_1\epsilon.
\end{align*}
As a consequence, $T(\phi)$ maps into itself.  Moreover, we further analyze $N(\phi)$ to obtain
\begin{align}\label{230}
\Vert N(\phi)\Vert_{*}=\mu\epsilon^2\Vert f(U_0+\phi)\Vert_{*}\leq & \bar C_5\epsilon^2\Vert\phi\Vert_{*}+\bar C_5\epsilon^2\nonumber\\
=&\bar C_1\bar C_5\epsilon^3+\bar C_5\epsilon^2\nonumber\\
=&\bar C_3\epsilon ^2,
\end{align}
where $\bar C_3>0$ is a constant.
We can see that (\ref{estar}) follows from (\ref{225}) and (\ref{230}).  

It is left to prove that $T$ satisfies the contraction property.  In fact, 
\begin{align*}
\Vert T(\phi_1)-T(\phi_2)\Vert_{*}\leq & \bar C\Vert N(\phi_1)-N(\phi_2)\Vert_{*}\\
\leq &\bar C\Vert\mu\epsilon^2f'(U_0)(\phi_1-\phi_2)\Vert_{*}+\bar C_6\Vert\epsilon^3(\phi_1-\phi_2)\Vert_{*}\\
\leq & \bar C_7\epsilon^2\Vert\phi_1-\phi_2\Vert_{*}=o(1)\Vert\phi_1-\phi_2\Vert_{*},
\end{align*}
where $\bar C_6,\bar C_7>0$ are large constants.  It follows that $T$ is a contraction mapping from $\mathcal H$ into $\mathcal H$. By the contraction mapping theorem we infer that there exists a fixed point $\phi$ which is a solution to (\ref{nonlinear}) in $\mathcal H$ and satisfies (\ref{contractionestimate}).
\end{proof}
\begin{remark}
With the help of the fixed point characterization of $\phi$ and the fact that the error $E$ depends continuously on the parameters $c_m$, $m=1,\cdots,k$, we obtain the map $(c_1,\cdots,c_k)\rightarrow \phi $ into the space $C(\bar \Omega_{\epsilon})$ is continuous in the $*$-norm. 
\end{remark}
\subsection{Reduced Problem}
After proving the existence of $\phi$, we solve the reduced problem $d_m=0$ for any $m=1,2,\cdots,k$ by adjusting the coefficients $c_{m,\chi}$.  Recall that $U$ is the solution of the following equation:
\begin{align}\label{231}
\nabla_y\cdot\bigg(\nabla_y U-\frac{1}{\epsilon^2} U\nabla_y A\bigg)+\mu f(U)=-\sum\limits_{m=1}^kd_mZ_m,
\end{align}
where $y\in\Omega_{\epsilon}$.  Testing (\ref{231}) against $\hat\eta_m$ and integrating it over $\Omega_{\epsilon}$, then we find
\begin{align}\label{232}
\int_{\Omega_{\epsilon}} \big[L_{\epsilon}(U_0+\phi)+E+N(U_0+\phi)\big]\hat\eta_mdy=-d_m\int_{\Omega_{\epsilon}} Z_m\hat \eta_mdy.
\end{align}
On the one hand, since $E$ satisfies (\ref{estar}), we obtain $\int_{\Omega_\epsilon}E\hat\eta_mdy=O(1)\epsilon$; on the other hand, $c_0$ is determined by the integral constraint (\ref{17}).  Therefore,
we have from (\ref{232}) that $d_m$ satisfies
\[d_m=\hat c_{m,\chi}+O(1)\epsilon,\]
where $\hat c_{m,\chi}:=c_{m,\chi}-c_0$.  It follows that $\hat c_{m,\chi}=O(\epsilon)=O(\frac{1}{\sqrt{\chi}})$ for $m=1,2,\cdots,k$.

\textbf{Proof of Theorem \ref{thm11}:} 
\begin{proof}
We let $U(y)=U_0(y)+\phi(y)$ be the solution to \eqref{ss}, where $$U_0:=\sum_{m=1}^kc_{m,\chi} e^{-\frac{1}{2}\sum\limits_{i=1}^nh_m^{(i)}\chi(x-x^{(i)}_m)^2}\eta_m.$$  
Here $c_{m,\chi}=c_0+\hat c_{m,\chi},$ and $\hat c_{m,\chi}$, $m=1,\cdots,k$ will be chosen later on.  Then, one gets $\phi$ satisfies (\ref{nonlinear}) with $E$ and $N(\phi)$ are given by $(\ref{errore}).$  To prove Theorem \ref{thm11}, we only need to show there exists some function $\phi$ to (\ref{nonlinear}) with $d_m=0$, $\forall m=1,2,\cdots,k.$  Proposition \ref{proposition23} implies that there exists a solution $\phi\in\mathcal H$ to (\ref{nonlinear}).  Since $\phi\in\mathcal H,$ we further have $\Vert\phi\Vert_{*}\leq \bar C_1\epsilon$ for some constant $\bar C_1>0.$  It is left to adjust $\hat c_{m,\chi}$ such that $d_m=0$.  Note that we have the asymptotic formula of $d_m$ is $d_m=\hat c_{m,\chi}+O(\epsilon),$ then one can obtain there exist $\hat c_{m,\chi}$ such that $d_m=0$ for any $m=1,\cdots,k.$  Moreover, $\hat c_{m,\chi}=O(\epsilon).$
In addition, since $U_0$ possesses the exponential decay property, we find 
$$\Big\vert U_0-\sum_{m=1}^kc_{m,\chi} e^{-\frac{1}{2}\sum\limits_{i=1}^nh_m^{(i)}\chi(x-x^{(i)}_m)^2}\Big\vert=o(1),$$
where $o(1)\rightarrow 0$ as $\chi\rightarrow\infty.$  Finally, we conclude the proof of Theorem \ref{thm11}. 
\end{proof}
This theorem establishes  nontrivial localized steady states to (\ref{origin}).  To further investigate their qualitative properties, we shall focus on the associated eigenvalue problem of (\ref{origin}).   
 \section{Eigenvalue Estimates and Stability Analysis}\label{sec3}
 In this section, we will investigate the local linear stability of the interior spike $u_s$ defined in Theorem \ref{thm11}.  To this end, we linearize (\ref{origin}) around $u_s$ and choose the solution $u$ in the following form: 
 $$u(x,t)=u_s+\varepsilon e^{\lambda t}\psi(x),$$
 where $\psi$ is a small perturbation and $\varepsilon\approx 0.$ The linearized problem becomes
 \begin{align}\label{ep}
\left\{\begin{array}{ll}
\mathcal L_{\chi}(\psi):=\nabla\cdot(d_1\nabla\psi-\chi\nabla A\psi)+\mu f'(u_s)\psi=\lambda\psi, &x\in\Omega,\\
(d_1\nabla\psi-\chi\nabla A\psi)\cdot\textbf{n}=0,&x\in\partial \Omega.
\end{array}
\right.
\end{align}
Since we concentrate on the effect of $\chi$ on the dynamics, in the sequel, other parameters $d_1$ and $\mu$ can be set as $\mu=d_1=1.$  As is shown in \cite{Nancy}, the function $\hat\psi$ defined by $\hat\psi:=e^{-\chi A}\psi$ plays  an important role in our analysis.  It is straightforward to show that $\hat\psi$ satisfies
 \begin{align}\label{epnew}
\left\{\begin{array}{ll}
{\hat{\mathcal L}}_{\chi}(\hat\psi):=\nabla\cdot\big(e^{\chi A}\nabla\hat\psi\big)+ f'(u_s)e^{\chi A}\hat\psi=\lambda e^{\chi A}\hat\psi, &x\in\Omega,\\
\frac{\partial \hat\psi}{\partial \textbf{n}}=0,&x\in\partial \Omega.
\end{array}
\right.
\end{align}
We consider $\hat\psi\in \mathcal Y$ where $\mathcal Y$ is given by
\[\mathcal Y=\big\{\hat\psi\in H^1(\Omega):\partial\hat\psi/{\partial \textbf{n}}=0\text{~on~}\partial \Omega\big\}.\]
In addition, the associated weighted inner product of $\mathcal Y$ is defined as
$$<\hat\phi,\hat\psi>_{\mathcal Y}:=\int_{\Omega}e^{\chi A}\big[\nabla \hat\psi\cdot\nabla \hat\phi+\hat\psi\hat\phi\big]dx,$$
and the norm is defined by
$$\Vert \hat\psi \Vert_{\mathcal Y}:=\sqrt{\int_{\Omega}e^{\chi A}\big[\vert\nabla \hat\psi\vert^2+\vert\hat\psi\vert^2\big]dx}.$$

We plan to perform a priori estimate of $\lambda$ so as to prove $\lambda$ is bounded.  To achieve our goal, we multiply the $\hat\psi$-equation in (\ref{epnew}) by $\hat\psi$ and integrate it over $\Omega$ to find 
\begin{align*}
-\int_{\Omega} e^{\chi A}\vert \nabla\hat\psi\vert^2 dx+\int_{\Omega}f'(u_s)e^{\chi A}\vert\hat\psi\vert^2dx=\lambda\int_{\Omega}e^{\chi A}\vert\hat\psi\vert^2 dx,
\end{align*}
which implies
\begin{align}\label{33}
\vert \lambda\vert\leq C_1\int_{\Omega} e^{\chi A}\vert \nabla\hat\psi\vert^2 dx+C_1\int_{\Omega} e^{\chi A}\vert f'(u_s)\vert\vert\hat\psi\vert^2dx,
\end{align}
where $C_1>0$ is a constant.  Since $f(u)\in C^\infty$ with respect to $u$, one has from (\ref{33}) that there exist constants $C_2,C_3>0$ such that
$$\vert\lambda\vert\leq C_2\Vert \hat\psi\Vert_{\mathcal Y}\leq C_3,$$
where $\hat \psi\in \mathcal Y.$  With the boundedness of $\lambda$, we are able to focus on the formulation of the eigenvalue problem satisfied by $<\lambda,\psi>$. 
\subsection{Formulation of the Eigenvalue Problem }
 Problem (\ref{ep}) can be simplified through a change of variable and the stretched variable can be defined as $y=\frac{x}{\epsilon}$.  To be more general, we assume that in the sequel, $u_s$ is a generic function satisfying the following transform:
$$u_s(x):=U_s(y),\quad U_s(y):=U_0(y)+\phi(y),$$
where $U_0(y)$ and $\phi(y)$ have the following properties:
\[\quad U_0(y)=\sum_{m=1}^kc_{m,\epsilon} e^{-\frac{1}{2}\sum\limits_{i=1}^nh_m^{(i)}(y-P^{(i)}_m)^2},\quad \vert \phi(y)\vert\leq M \epsilon e^{-\bar\sigma y},\]
where constant $M>0$ and $\bar\sigma\in(0,1)$ independent of $\epsilon$ and $y.$  Moreover, the old steady states and eigen-pairs $(x,u_s(x),\lambda,\psi(x))$ are equivalent to the new ones given by $(y,U_s(y),\mu,\Psi(y))$ via the following transforms:
 $$x=\epsilon y,\text{~~}\forall y\in\Omega_{\epsilon}\Leftrightarrow y=\frac{x}{\epsilon},\text{~~}\forall x\in\Omega,\quad u_s(x)=U_s\bigg(\frac{x}{\epsilon}\bigg)\Leftrightarrow U_s(y)=u_s(\epsilon y);$$
 and
 $$\lambda=\frac{\mu}{\epsilon^2}\Leftrightarrow \mu=\epsilon^2\lambda,\quad \psi(x)=\Psi\Bigg(\frac{x}{\epsilon}\Bigg)\Leftrightarrow \Psi(y)=\psi(\epsilon y).$$ 
 Therefore, (\ref{ep}) is equivalent to the following scaled eigenvalue problem:
 \begin{align}\label{psieq}
\left\{\begin{array}{ll}
\mathscr L_{\epsilon}(\Psi):=\nabla_y\cdot(\nabla_y\Psi-\frac{1}{\epsilon^2}\nabla_y A\Psi)+\epsilon^2f'(U_s)\Psi=\mu\Psi, &y\in\Omega_{\epsilon},\\
\Big(\nabla_y\Psi-\frac{1}{\epsilon^2}\nabla_y A\Psi\Big)\cdot\textbf{n}=0,&y\in\partial \Omega_{\epsilon}.
\end{array}
\right.
\end{align}
Similarly, it is useful to analyze the property of $\hat\Psi$ defined by $\hat\Psi:=e^{-\frac{A}{\epsilon^2}}\Psi$.  The equation for $\hat\Psi$ is shown as
 \begin{align*}
{\hat{\mathscr L}}_{\epsilon}(\hat\Psi):=\nabla_y\cdot\Big(e^{\frac{A}{\epsilon^2}}\nabla_y\hat\Psi\Big)+\epsilon^2f'(U_s)e^{\frac{A}{\epsilon^2}}\hat\Psi=\mu e^{\frac{A}{\epsilon^2}}\hat\Psi, \quad y\in\Omega_{\epsilon}.
\end{align*}
Let $\hat\Psi\in Y$ with $Y$ being given by 
\[Y=\Big\{\hat\Psi\in H^1(\Omega_{\epsilon}):\partial\hat\Psi/{\partial \textbf{n}}=0\text{~on~}\partial \Omega_{\epsilon}\Big\}.\]
 Define the weighted inner product and norm as
$$<\hat\Phi,\hat\Psi>_{Y}:=\int_{\Omega_{\epsilon}}e^{\frac{A}{\epsilon^2}}\big[\nabla_y \hat\Psi\cdot\nabla_y \hat\Phi+\hat\Psi\hat\Phi\big]dy$$
and
$$\Vert \hat\Psi \Vert_{Y}:=\sqrt{\int_{\Omega_{\epsilon}}e^{\frac{A}{\epsilon^2}}\big[\vert\nabla_y \hat\Psi\vert^2+\vert\hat\Psi\vert^2\big]dy}.$$
Then we have the following properties of the operator $\hat {\mathscr L}_{\epsilon}$:
\begin{lemma}
Assume that $U_s$ is generic. Then the operator $\hat{\mathscr L}_{\epsilon}$ in $Y$ is self-adjoint.  And the following conclusions hold for the eigen-pairs $\big<\mu,\hat\Psi(y)\big>$:
\begin{itemize}
    \item[(i).] all eigenvalues are real and eigenvectors corresponding to different eigenvalues are perpendicular to each other;
    \item[(ii).] $\{\big<\mu_i,\hat\Psi_i\big>\}_{i=1}^\infty$ consist of a complete set of eigen-pairs, and $\mu_1\geq \mu_2\geq\cdots$, which satisfy
    \begin{align*}
    \mu_i:=\max_{Z\subset Y,\dim Z=i}\min_{\hat\Psi\in Z}\Bigg\{\frac{-\int_{\Omega_{\epsilon}}\Big(e^{\frac{A}{\epsilon^2}}\vert\nabla_y\hat\Psi\vert^2\Big)dy+\epsilon^2\int_{\Omega_{\epsilon}}f'(U_s)e^{\frac{A}{\epsilon^2}}\vert\hat\Psi\vert^2dy}{\int_{\Omega_{\epsilon}}e^{\frac{A}{\epsilon^2}}\vert\hat\Psi\vert^2dy}\Bigg\}.
\end{align*} 
\end{itemize}
\end{lemma}
\begin{proof}
By using the integration by parts, one has for any $\hat \Phi$ and $\hat \Psi$ in $Y$,
\begin{align*}
<\hat {\mathscr L}_{\epsilon}\hat\Phi,\hat\Psi>_{Y}=-\int_{\Omega_{\epsilon}}\Big(e^{\frac{A}{\epsilon^2}}\nabla_y\hat\Psi\cdot \nabla_y\hat\Phi \Big)dy+\epsilon^2\int_{\Omega_{\epsilon}}f'(U_s)e^{\frac{A}{\epsilon^2}}\hat\Psi \hat\Phi dy=<\hat{\mathscr L}_{\epsilon}\hat\Psi,\hat\Phi>_{Y},
\end{align*}
which implies $\hat {\mathscr L}_{\epsilon}$ is self-adjoint.  The conclusions (i) and (ii) follow from the standard results of self-adjoint operators.
\end{proof}
\subsection{Spectrum Analysis}
After formulating the eigenvalue problems and stating the characterization of eigen-pairs $<\mu,\hat\Psi(y)>$, our focus is then on the study of the spectrum of the operator $ {\mathscr L}_{\epsilon}$ defined by (\ref{psieq}).  To this end, some results obtained from the previous Lyapunov-Schmidt reduction procedure need to be stated.  First of all, define the approximate kernel set as
\[\mathcal K_{\epsilon}:=\text{span}\{Z_m^{\epsilon}\vert m=1,\cdots,k\}\subset H^2_{N^*}(\Omega_{\epsilon}),\]
and $\mathcal C_{\epsilon}$ as the approximate cokernel set.  Then $\mathcal K^{\perp}_{\epsilon}$ and $\mathcal C^{\perp}_{\epsilon}$ are denoted as the orthogonal sets of $\mathcal K_{\epsilon}$ and $\mathcal C_{\epsilon}.$  It is necessary to next set $\pi_{\epsilon,P}$ and $\pi^{\perp}_{\epsilon,P}$ to be the projection of $L^2(\Omega_{\epsilon})$ onto $\mathcal C_{\epsilon}$ and $\mathcal C_{\epsilon}^{\perp}$, respectively.  Our argument in Section \ref{sec2} can be summarized as that there exists a unique solution $\phi$ such that
$$\pi_{\epsilon}^{\perp}\circ S_{\epsilon}(U_0+\phi)=0,$$
where $\phi\in \mathcal K^{\perp}_{\epsilon}$ and $S_{\epsilon}$ is defined by (\ref{operator}).
The results involving with a priori estimate of $\phi$ in Proposition \ref{proposition22} are stated as:
\begin{proposition}\label{prop31}
Let $\tilde L_{\epsilon}:=\pi_{\epsilon}^{\perp}\circ L_{\epsilon}$, where $L_{\epsilon}$ is given by (\ref{29}), then there exist positive constants $\epsilon_0$ and $C$ such that for all $\epsilon\in(0,\epsilon_0]$, the solution $\phi$ to (\ref{project}) satisfies
\[\Vert \phi\Vert_{*}\leq C\Vert \tilde L_{\epsilon}\phi\Vert_{*},\]
where $\phi\in \mathcal K_{\epsilon}^{\perp}$.
\end{proposition}
Furthermore, the results for the nonlinear projected problem of $\phi$ stated in Proposition \ref{proposition23} are rewritten and shown in the following proposition:
\begin{proposition}\label{prop2}
There exists a positive constant $\epsilon_0$ such that for all $\epsilon\in(0,\epsilon_0]$, (\ref{nonlinear}) admits the unique solution $\phi\in\mathcal K_{\epsilon}^{\perp}$ satisfying $S_{\epsilon}(U_{0}+\phi)\in \mathcal C_{\epsilon}$ such that
\[\Vert \phi\Vert_{*}\leq C\epsilon.\]
\end{proposition}
The next step is to give the asymptotics satisfied by eigen-pairs $<\mu,\Psi>$ via Proposition \ref{prop31} and \ref{prop2}.  Our results are summarized as follows:
\begin{proposition}\label{prop3}
Let $<\mu,\Psi>$ be the solutions to (\ref{psieq}), then there exists $\epsilon_0>0$ such that for all $\epsilon\in(0,\epsilon_0]$, $i=1,2,\cdots,k$,
\begin{align}\label{muasy}
\mu_i =\epsilon^2\alpha_0h'(c_{i,0})+ O(1)\epsilon^3,
\end{align}
where $\alpha_0:=\frac{1}{\int_{\mathbb R^n}e^{-\vert y\vert^2}dy}$ and $c_{i,0}$ is either $c_0$ defined by (\ref{c0}) or $0$.  Furthermore,
$$\Psi_i=\sum_{m=1}^k\big[e^{(i)}_{m,0}+o(1)\big]Z_m+o(1),$$
where $e^{(i)}_{m,0}$ are non-negative constants but not identically zero.  As a consequence,
$$\lambda_i =\alpha_0h'(c_{i,0})+ O(1)\epsilon.$$
\end{proposition}
\begin{proof}
Note that $\lambda$ is bounded, then we can choose sub-sequences $\epsilon_n$, $\hat \Psi_n:=e^{-\frac{A}{\epsilon^2_n}}\Psi_n$ such that $\epsilon_n\rightarrow 0$,  $\Psi_n\rightarrow \Psi_0$ and $\hat \Psi_n\rightarrow \hat\Psi_0$.  Moreover, $\hat\Psi_0$ near every centre of the interior spike defined by $P_m$ satisfies
\begin{align}\label{36}
\left\{\begin{array}{ll}
\hat{\mathscr L}_0^m(\hat\Psi_0^m):=\nabla_{y_m}\cdot\Big(e^{a_m(y_m)}\nabla_{y_m}\hat\Psi_0^m\Big)=0,&y_m\in\mathbb R^n,\\
\hat \Psi_0^m\in H^1(\mathbb R^n),\quad \vert \hat\Psi_0^m\vert=O(1)e^{-\tilde\sigma_1\vert y_m\vert^2}
\end{array}
\right.
\end{align}
where $y_m:=\frac{ x-x_m}{\epsilon}$ and $\tilde\sigma_1\in (0,1).$  According to Proposition \ref{proposition21}, system (\ref{36}) admits the unique solution $\hat\Psi_0^m\equiv C,$ where $C>0$ is a constant, which implies $\Psi_0^m=Ce^{a_m(y_m)}.$ As a consequence, $\Psi_0$ has the following form:
\begin{align*}
\Psi_0=\sum_{m=1}^ke_{m}\Psi_0^m=\sum_{m=1}^k e_{m}Z_m,
\end{align*}
where $e_{m}>0$ are constants for $m=1,2,\cdots,k$ and $Z_m$ are approximate kernels of $L_{\epsilon}$ with $Z_m=e^{a_m(y_m)}.$  We would like to point out that $Z_m$ is an approximation of $Z_m^{\epsilon}$ defined by $e^{\frac{1}{\epsilon^2}(A-A_m)}$, which is the solution of
\[\nabla_y\cdot\bigg(\nabla_yZ_m^{\epsilon}-\frac{1}{\epsilon^2}\nabla_y AZ_m^{\epsilon}\bigg)=0,\quad y_m\in \mathbb R^n.\]
  
Now, we take the ansatz of $\Psi$, the eigen-funcion of the operator $\mathscr L_{\epsilon}$, as 
$$\Psi_0^{\epsilon}=\sum_{m=1}^k e_{m,\epsilon}Z_m^{\epsilon},$$
where $e_{m,\epsilon}>0$ are constants.  Thus $\Psi$ can be decomposed as
\[\Psi=\Psi_0^{\epsilon}+\Psi^{\perp}_{\epsilon},\]
where $\Psi_{\epsilon}^{\perp}\perp \mathcal K_{\epsilon}$ and $\Vert\Psi_{\epsilon}^{\perp}\Vert_{*}=o(1).$  Recall $\mathscr L_{\epsilon}$ is given in (\ref{psieq}), then by straightforward calculation, we have 
\begin{align*}
\mathscr L_{\epsilon}\Psi=&\sum_{m=1}^ke_{m,\epsilon}\bigg[\nabla_y\cdot\bigg(\nabla_yZ_m^{\epsilon}-\frac{1}{\epsilon^2}\nabla_yAZ_m^{\epsilon}\bigg)+\epsilon^2f'(U_s)Z_m^{\epsilon}\bigg]\\
&+\nabla_y\cdot\bigg(\nabla_y\Psi_{\epsilon}^{\perp}-\frac{1}{\epsilon^2}\nabla_yA\Psi_{\epsilon}^{\perp}\bigg)+\epsilon^2f'(U_s)\Psi_{\epsilon}^{\perp}\\
=&\mathscr L_{\epsilon}\Psi_{\epsilon}^{\perp}+\epsilon^2[f'(U_s)-f'(U_0)]\Psi_{\epsilon}^{\perp}+\sum_{m=1}^ke_{m,\epsilon}\epsilon^2f'(U_s)Z_m^{\epsilon}\\
=&\sum_{m=1}^k \epsilon^2\lambda e_{m,\epsilon}Z_m^\epsilon+\epsilon^2\lambda\Psi_{\epsilon}^{\perp}.
\end{align*}
In light of Proposition \ref{prop31} and the boundedness of $\lambda$, the following estimate holds:
\begin{align}\label{38}
\Vert \Psi_{\epsilon}^{\perp}\Vert_{*}\leq&C\Vert \epsilon^2[f'(U_s)-f'(U_0)]\Psi_{\epsilon}^{\perp}\Vert_{*}+C\bigg\Vert\sum_{m=1}^ke_{m,\epsilon}\epsilon^2f'(U_s)Z_m^{\epsilon}\bigg\Vert_{*}\nonumber\\
\leq &C_1\epsilon^2 \sum_{m=1}^k \vert e_{m,\epsilon}\vert,
\end{align}
where $C$ and $C_1$ are positive constants.  (\ref{38}) implies $\Psi_{\epsilon}^{\perp}$ can be neglected and $\Psi$ satisfies
\[\Psi=\sum_{m=1}^k e_{m,\epsilon}Z_m^\epsilon+O(1)\epsilon^2\sum_{m=1}^k\vert e_{m,\epsilon}\vert.\]
Since we have the fact that
\begin{align*}
Z_m^{\epsilon}=Z_m+O(1)\epsilon e^{-\hat\sigma_5\vert y_m\vert^2}\text{~for small~}\hat\sigma_5>0,
\end{align*}
$\Psi$ has the following form: 
\[\Psi=\sum_{m=1}^k e_{m,\epsilon}Z_m+O(1)\epsilon\sum_{m=1}^k\vert e_{m,\epsilon}\vert.\]
Now, the estimates of eigen-functions have been completed.  We obtain that there exists $k$ eigenpairs $<\mu_i,\Psi_i>$, $i=1,2,\cdots,k$ such that $\Psi_i$ satisfies
\begin{align}\label{Psi0}
\Psi_i=&\sum_{m=1}^k e^{(i)}_{m,\epsilon}Z_m^\epsilon+O(1)\epsilon^2\sum_{m=1}^k\vert e^{(i)}_{m,\epsilon}\vert\nonumber\\
=&\sum_{m=1}^k [e^{(i)}_{m,0}+o(1)]Z_m+O(1)\epsilon\sum_{m=1}^k\vert e^{(i)}_{m,\epsilon}\vert.
\end{align}
We next analyze the asymptotic behavior of the corresponding eigenvalues $\mu_i,$ then derive the property of $\lambda_i:=\epsilon^2\mu_i.$  To begin with, let $i=1,2,\cdots,k$ be arbitrary but fixed, then we test the $\Psi$-equation in (\ref{psieq}) against $\hat\eta_i$ and integrate it over $\Omega_{\epsilon}$ to obtain
\begin{align}\label{eigeneq}
\overbrace{\frac{1}{\epsilon^2}\int_{\Omega_{\epsilon}}\bigg(\nabla_y\Psi_{i}-\frac{1}{\epsilon^2}\nabla_y A\Psi_{i}\bigg)\nabla_y \hat \eta_i dy}^{I_1}+\int_{\Omega_{\epsilon}}f'(U_s)\Psi_{i}\hat\eta_i dy=\lambda_i\int_{\Omega_{\epsilon}}\Psi_i \hat\eta_i dy.
\end{align}
Our next aim is to estimate $I_1$, which we define $y_i=y-P_i$, becomes
\begin{align}\label{39}
I_1=\frac{1}{\epsilon^2}\int_{B_{R_{i,\epsilon}}(0)\backslash B_{R_{i,\epsilon}-\tilde\delta}(0)}\bigg(\nabla_y\Psi_{i}-\frac{1}{\epsilon^2}\nabla_y A\Psi_{i}\bigg)\nabla_y \hat \eta_i dy_i,
\end{align}
where $\tilde \delta>0$ is small and $R_{i,\epsilon}={R_i}/\epsilon$ with a large constant $R_i$.  Since $\Psi_i=O(1)\frac{1}{\epsilon^2}e^{-\hat\sigma_6\vert R_i/\epsilon\vert^2}$ for some small constant $\hat\sigma_6>0$ on $\partial B_{R_{i,\epsilon}}(0)$, one obtains from (\ref{39}) that there exists $\hat\sigma_7>0$ such that $I_1=O(1)e^{-\hat\sigma_7\vert R_i/\epsilon \vert^2}$.  Moreover,
\begin{align}\label{310}
\int_{\Omega_{\epsilon}}f'(U_s)\Psi_i\hat\eta_i dy=&e_{i,\epsilon}\int_{B_{R_{i,\epsilon}(P_i)}}f'(U_s)Z_i^{\epsilon}\hat\eta_idy+O(1)\epsilon^2\sum_{m=1}^k\vert e^{(i)}_{m,\epsilon}\vert\int_{\Omega_{\epsilon}}f'(U_s)\hat\eta_i dy\nonumber\\
&+\sum_{m\not=i}e_{m,\epsilon}\int_{B_{R_{i,\epsilon}}(P_i)}f'(U_s)Z_m^{\epsilon}\hat \eta_idy,
\end{align}
where the last term has the order satisfying $O\bigg(e^{-\hat\sigma_8\frac{\min_{m\not=i}\vert x_m-x_i\vert^2}{\epsilon^2}}\bigg)$ with small $\hat\sigma_8>0$.  Thus, by applying $f'(U_0+\phi)=f'(U_0)+f''(U_0)\phi$ to (\ref{310}), one has
\begin{align*}
\int_{\Omega_{\epsilon}}f'(U_s)\Psi_{i}\hat\eta_i dy=&e_{i,\epsilon}\int_{B_{R_{i,\epsilon}}(P_i)}f'(U_0)Z_i^{\epsilon}\hat \eta_idy+e_{i,\epsilon}\int_{B_{R_{i,\epsilon}}(P_i)}f''(U_0)\phi Z_i^{\epsilon}\hat \eta_i dy+O(1)\epsilon^2\sum_{m=1}^k\vert e^{(i)}_{m,\epsilon}\vert\\
= &e_{i,\epsilon}\int_{B_{R_{i,\epsilon}}(P_i)}f'(U_0)Z_i^{\epsilon}\hat \eta_i dy+ O(1)\epsilon\sum_{m=1}^k\vert e^{(i)}_{m,\epsilon}\vert,
\end{align*}
where we use the fact that $\Vert\phi\Vert_{*}=O(\epsilon)$ due to Proposition \ref{prop2}.  Similarly, the right hand side of (\ref{eigeneq}) can be written as
\begin{align*}
\lambda_i\int_{\Omega_{\epsilon}}\Psi_i\hat\eta_idy=\lambda_i e_{i,\epsilon}\int_{B_{R_{i,\epsilon}}(P_i)}Z_i^{\epsilon}\hat\eta_{i}dy +O(1)\epsilon^2\sum_{m=1}^k\vert e^{(i)}_{m,\epsilon}\vert,
\end{align*}
where $\int_{B_{R_{i,\epsilon}}(P_i)}Z_i^{\epsilon}\hat\eta_{i}dy$ is finite due to the boundedness of $\int_{\Omega_{\epsilon}}Z_i^{\epsilon} dy$.  In conclusion, (\ref{eigeneq}) can be simplified as 
\begin{align}\label{312}
e_{i,\epsilon}\int_{B_{R_{i,\epsilon}}(P_i)}f'(U_0)Z_i^{\epsilon}\hat \eta_i dy=\lambda_i e_{i,\epsilon}\int_{B_{R_{i,\epsilon}}(P_i)}Z_i^{\epsilon}\hat\eta_{i}dy+ O(1)\epsilon\sum_{m=1}^k\vert e^{(i)}_{m,\epsilon}\vert.
\end{align}
It is easy to see that $e_{i,\epsilon}\not=0$ since the definition of eigen-functions.  Now, it is left to discuss $I_2^{(i)}$ defined by
$$I_2^{(i)}=\int_{B_{R_{i,\epsilon}}(P_i)}f'(U_0)Z_i^{\epsilon}\hat \eta_i dy.$$
Recall that
\begin{align*}
Z_m^{\epsilon}:=e^{\frac{1}{\epsilon^2}(A-A_m)}=e^{-\sum\limits_{i=1}^n\frac{1}{2}h_m^{(i)}\big(y_m^{(i)}\big)^2+O(1)\epsilon\vert y_m\vert^3}=e^{a_m(y_m)}+O(1)\epsilon e^{-\hat\sigma_9\vert y_m\vert^2}
\end{align*}
for some small $\hat\sigma_9>0$, and 
\begin{align*}
 U_{0}=\sum_{m=1}^k c_{m,0} e^{a_m(y_m)}+O(\epsilon)\sum_{m=1}^k e^{a_m(y_m)},
 \end{align*}
 where $c_{m,0}:=c_0$ for $m\in \mathcal I_{s}$ and $c_{m,0}:=0$ for $m\not\in \mathcal I_s$, then one has
\begin{align*}
I_2^{(i)}=&\int_{B_{R_{i,\epsilon}}(P_i)}f'\big(c_{i,0}e^{a_i(y_i)}\big)e^{a_i(y_i)}\hat \eta_idy+O(1)e^{-\hat\sigma_8\frac{\min_{m\not=i}\vert x_m-x_i\vert^2}{\epsilon^2}}\sum_{m=1}^k\vert e^{(i)}_{m,\epsilon}\vert+O(1)\epsilon\\
=&\int_{\mathbb R^n}f'\big(c_{i,0}e^{a_i(y_i)}\big)e^{a_i(y_i)}dy+\int_{B_{R_{\epsilon}}(P_i)}f'\big(c_{i,0}e^{a_i(y_i)}\big)e^{a_i(y_i)}\hat \eta_idy-\int_{\mathbb R^n}f'\big(c_{i,0}e^{a_i(y_i)}\big)e^{a_i(y_i)}dy\\
&+O(1)e^{-\hat\sigma_8\frac{\min_{m\not=i}\vert x_m-x_i\vert^2}{\epsilon^2}}\sum_{m=1}^k\vert e^{(i)}_{m,\epsilon}\vert+O(1)\epsilon\\
=&\int_{\mathbb R^n}f'\big(c_{i,0}e^{a_i(y_i)}\big)e^{a_i(y_i)}dy+O(1)\epsilon\sum_{m=1}^k\vert e^{(i)}_{m,\epsilon}\vert.
\end{align*}
Denote $h(\xi)$ as
$$h{(\xi)}:=\xi[-2^{\frac{n}{2}}\xi^2+3^{\frac{n}{2}}(1+\theta)\xi-6^{\frac{n}{2}}\theta],$$
then we have $h(c_{i,0})=\int_{\mathbb R^n}f\big(c_{i,0} e^{a_i(y_i)}\big)dy.$  Furthermore, $I_2^{(i)}$ becomes
\[I_2^{(i)}=h'(c_{i,0})+O(1)\epsilon\sum_{m=1}^k\vert e^{(i)}_{m,\epsilon}\vert.\]
Now, (\ref{312}) can be rewritten as
\begin{align}\label{lambda0}
\lambda_i =\alpha_0h'(c_{i,0})+ O(1)\epsilon\sum_{m=1}^k\vert e^{(i)}_{m,\epsilon}\vert=\alpha_0h'(c_{i,0})+ O(1)\epsilon,
\end{align}
where $\alpha_0:=\frac{1}{\int_{\mathbb R^n}e^{-\vert y\vert^2}dy}=\pi^{-\frac{n}{2}}$ is a positive constant and $i=1,2,\cdots,k$.  Combine (\ref{Psi0}) and (\ref{lambda0}), then our proof is finished.
\end{proof}
We have established the bounded eigenvalue estimates of the operator $\mathscr L_{\epsilon}$ in Proposition \ref{prop3}.  Our next goal is to study the local linearized stability of interior spikes given by (\ref{u0}) via these estimates.
\subsection{Stability Analysis}
It is well known that the stability properties are determined by the signs of the principal eigenvalues.  On the other hand, the asymptotics (\ref{muasy}) indicate that the signs depend on the height of every bump defined by $c_0$.  Therefore, there are several cases for the stability properties in terms of $c_0$, which are summarized as:
\begin{proposition}\label{prop4}
Assume all conditions in Theorem \ref{thm11} hold, then the steady states given by (\ref{u0}) satisfy the following alternatives:
\begin{itemize}
    \item[(i).] when $c_{m,0}$ are either $c_{01}$ or $0$ for any $m=1,\cdots,k$, steady states (\ref{u0}) are linear stable;
    \item[(ii).]  when there exists some $m$ such that $c_{m,0}=c_{02}$, steady states (\ref{u0}) are unstable;
\end{itemize}
\end{proposition}
\begin{proof}
According to Proposition \ref{prop3}, we have for any $i=1,\cdots,k,$ $\lambda_i$ satisfies
\begin{align}\label{315}
\lambda_i =\alpha_0h'(c_{i,0})+ O(1)\epsilon,
\end{align}
where $\alpha_0=\pi^{-\frac{n}{2}}$ and $h$ is given by
$$h{(\xi)}:=\xi[-2^{\frac{n}{2}}\xi^2+3^{\frac{n}{2}}(1+\theta)\xi-6^{\frac{n}{2}}\theta].$$

Before proving this proposition, we state some properties satisfied by $h$. First of all, $h(\xi)$ admits three distinct zeros denoted by $0$, $c_{01}$ and $c_{02}$, since $\theta$ is assumed to satisfy  $\theta\in(0,\theta_1)$ where $\theta_1$ is given by (\ref{theta1}).  Next, we have the facts that $h'$ at $0$ and $c_{01}$ are negative but at $c_{02}$ are positive.  We shall prove the statements in case (i) and case (ii), respectively.    

In case (i), one supposes that $c_{i,0}=c_{01}$ or $c_{i,0}=0$ for any $i=1,\cdots,k$.  Noting that (\ref{315}) and the properties of $h$, one finds $\lambda_i<0$ for small $\epsilon$.  In other words, the principal eigenvalues of the operator $\mathscr L_{\epsilon}$ are negative.  Therefore, when $\epsilon\ll 1$, steady states (\ref{u0}) are linear stable.

In case (ii), we claim that there exists some eigen-pair $<\mu_*,\Psi_*>$ within (\ref{psieq}) such that $\mu_*>0$ for sufficiently small $\epsilon$.  Indeed, with the help of Proposition \ref{prop3}, one can choose $i=m$ such that
$$\Psi_m=\sum_{j=1}^k\big[e^{(m)}_{j,0}+o(1)\big]Z_j+o(1),$$
where $e^{(m)}_{m,0}=O(1)$ and $e^{(m)}_{j,0}=o(1)$ for $j\not=m.$  Moreover, the corresponding $\mu_m$ satisfies
$$\mu_m =\epsilon^2\alpha_0h'(c_{m,0})+ O(1)\epsilon^3.$$
According to the properties satisfied by $h$, we have $h'(c_{m,0})>0$.  This conclusion in conjunction with (\ref{muasy}) implies $\mu_m>0$ when $\epsilon$ is sufficiently small.  Now, we indeed show that when $\epsilon\ll 1,$ there exists the eigen-function $\Psi_*:=\Psi_m$ with some coefficient $e^{(m)}_{m,0}=O(1)$ such that the associated eigenvalue $\mu_{*}$ satisfies $\mu_*>0$, which is our claim. 

Due to the definition of $\lambda_*$, one further obtains that $\lambda_*>0$ now that $\epsilon$ is sufficiently small.  As a consequence, steady states (\ref{u0}) are unstable. 
\end{proof}
\begin{remark}
When $\epsilon$ is small, the $k$ linear independent eigen-functions $\Psi_i,$ $i=1,2,\cdots,k$ defined in Proposition \ref{prop3} are approximate orthogonal, and thereby the coefficient matrix $\Big(e^{(m)}_{j}\Big)_{1\leq m,j\leq k}$ is ``nearly diagonal".
\end{remark}
Proposition \ref{prop4} implies that when $\epsilon\ll 1$, (\ref{origin}) admits many stable and unstable interior spikes.  In particular, the existence of bumps with the small positive heights will cause steady states to become unstable.  We would like to mention that the alternatives in Proposition \ref{prop4} also hold for sufficiently large $\chi$ since $\epsilon:=\frac{1}{\sqrt{\chi}}$, which immediately finishes the proof of Theorem \ref{thm12}.  Figure \ref{singledynamics}, \ref{unstabledynamics}, \ref{multidynamic1} and \ref{multiunstable} support our theoretical results in Theorem \ref{thm12}.

One can observe from our analysis that the source $f(u)$ plays the key role in stabilizing interior spikes for (\ref{origin}).  Indeed, the lack of source term in other reaction-advection-diffusion systems such as the minimal Keller--Segel models typically leads to the instability of multi-spike solutions.  For the results of the minimal Keller--Segel models, we refer readers to \cite{KW2021}.  

The single equation served as a paradigm to model the evolution of the single species admits some interesting patterns.   However, more interesting phenomena are discovered while studying the interaction effects among multiple species.  

\section{Population Competition Models}\label{sec4}
To investigate the coexistence of two interacting species, we shall consider several specific forms of (\ref{14}).  First of all, we denote $u$ and $v$ as the densities of two competing species, then suppose $\chi_1=\chi$, $\chi_2=1$ and the boundary condition is no-flux in (\ref{14}) to obtain
\begin{align}\label{IFD}
\left\{\begin{array}{ll}
u_t=\nabla\cdot(\nabla u-\chi u\nabla A)+ug(x,u+v),&x\in\Omega,t>0,\\
v_t=\nabla\cdot(\nabla v-v\nabla A)+vg(x,u+v),&x\in\Omega,t>0,\\
\partial_{\textbf{n}} u-\chi u\partial_{\textbf{n}}A=\partial_{\textbf{n}}v-v\partial_{\textbf{n}}A=0,&x\in\partial \Omega,t>0,\\
u(x,0)=u_0(x)\geq 0, v(x,0)=v_0(x)\geq 0,&x\in\Omega,
\end{array}
\right.
  \end{align}
 where $v$-species follows the IFD strategy, whereas $u$-species takes the aggressive strategy with the higher speed.  We are interested at the existence of steady states and their local dynamics to (\ref{IFD}) suject to Allee effect.  In particular, we focus on the effect of IFD strategy on the stability of steady states.
 
 We next set $\chi_1=\chi$, $\chi_2=c\chi$ and the boundary condition is no-flux in (\ref{14}) to get
  \begin{align}\label{AA}
\left\{\begin{array}{ll}
u_t=\nabla\cdot(\nabla u-\chi u\nabla A)+ug(x,u+v),&x\in\Omega,t>0,\\
v_t=\nabla\cdot(\nabla v-c\chi v\nabla A)+vg(x,u+v),&x\in\Omega,t>0,\\
\partial_{\textbf{n}} u-\chi u\partial_{\textbf{n}}A=\partial_{\textbf{n}}v-c\chi v\partial_{\textbf{n}}A=0,&x\in\partial \Omega,t>0,\\
u(x,0)=u_0(x)\geq 0, v(x,0)=v_0(x)\geq 0,&x\in\Omega,
\end{array}
\right.
  \end{align}
  where $c>1$ is a constant.  The system (\ref{AA}) subject to Allee effect can be used to describe the interaction between two aggressive species and we will also focus on the concentration phenomena within (\ref{AA}) and its dynamics.
 

\subsection{IFD Strategy Versus Aggressive Strategy}
One of our central problems in Section \ref{sec4} is the effect of IFD on the survival of species.  To investigate it, we first consider the stationary problem of (\ref{IFD}) over the heterogeneous environment, which is
  \begin{align}\label{SSS}
\left\{\begin{array}{ll}
0=\nabla\cdot(\nabla u-\chi u\nabla A)+ug(x,u+v),&x\in\Omega,\\
0=\nabla\cdot(\nabla v-v\nabla A)+vg(x,u+v),&x\in\Omega,\\
\partial_{\textbf{n}} u-\chi u\partial_{\textbf{n}}A=\partial_{\textbf{n}}v-v\partial_{\textbf{n}}A=0,&x\in\partial \Omega,
\end{array}
\right.
  \end{align}
  where the spatial region $\Omega \subset \mathbb R^n$, $n\geq 1$ is a bounded domain with smooth boundary $\partial\Omega$.  Here $\chi$ measures the speed of the aggressive species and the signal $A$ reflects the variations of directed dispersal magnitude.  Recall the growth pattern $g$ is defined as Allee effect form, which is
  $$g(x,u+v)=(u+v-\theta)(r-u-v),$$ 
  where $r=r(x)$ is the fixed external resources.  To show the IFD strategy is a local invader strategy, the invasive species $v$ is supposed to take the IFD strategy, i. e. the $v$-equation in (\ref{SSS}) should admit the solution $v=r$ with $A$ being $A=\ln r$.  Therefore, we obtain $(0,r(x))$ is a non-constant steady state to (\ref{SSS}).  This equilibrium represents the invasive species eventually wins and the established species is extinct.  
\subsubsection{Existence of Non-constant Steady States}
In this section, we construct other non-trivial steady states and our results are summarized as follows:
\begin{proposition}\label{prop41}
Assume all conditions in Theorem \ref{thm11} hold.  Let $A=\ln r(x)$ and $g(x,u+v)=(u+v-\theta)(r-u-v)$ in (\ref{SSS}), then we have when $\chi\gg 1,$ for $\theta$ sufficiently small, (\ref{SSS}) only admits the following two types of spiky solutions:
\[(u_s,v_s)=(0,r(x));~~~(u_s,v_s)=(u^*,0).\]
Here $u^*$ is given by 
\begin{align}\label{ustarprop41}
 u^*(x;\chi)=\sum_{m=1}^k\bar c_{m,\chi} e^{-\frac{1}{2}\sum\limits_{i=1}^nh_m^{(i)}\chi(x-x^{(i)}_m)^2}+o(1),
 \end{align}
 where $o(1)\rightarrow 0$ uniformly as $\chi \rightarrow \infty.$  In particular, if $m\in\mathcal I_b$, $\bar c_{m,\chi}=\bar c_{01}+O(\frac{1}{\sqrt{\chi}})$; if $m\in \mathcal I_s\backslash\mathcal I_{b},$ $\bar c_{m,\chi}=\bar c_{02}+O(\frac{1}{\sqrt{\chi}})$; if $m\not\in\mathcal I_{s},$ $\bar c_{m,\chi}=O(\frac{1}{\sqrt{\chi}})$, where $\bar c_{01}$ and $\bar c_{02}$ are defined in (\ref{barc0}).
\end{proposition}
\begin{proof}
 We have the fact that $(0,r(x))$ is a non-trivial solution of (\ref{SSS}).  To find other steady states, we first let $v=0$ in (\ref{SSS}), then the system are simplified into the following form:
 \begin{align}\label{44}
 \left\{\begin{array}{ll}
 0=\nabla\cdot(\nabla u-\chi u\nabla A)+ug(u),&x\in\Omega,\\
 \partial_{\textbf{n}} u-\chi u\partial_{\textbf{n}}A=0,&x\in\partial\Omega,
 \end{array}
 \right.
 \end{align}
 where $g(u)=(u-\theta)(r-u).$  Noting that $A_m$ is the $m$-th non-degenerate maximum of $A$, we define $\bar a_m:=e^{A_m}$, $\delta_1:= 3^n \theta^2 + (2\cdot 3^n \bar a_m- 4\cdot 12^{\frac{n}{2}}  \bar a_m) \theta+3^n \bar a_m^2$ and
\begin{align*}
\bar\theta_1:=\bar a_m\cdot\frac{2^{n+1}- 3^{\frac{n}{2}}- 2\sqrt{4^{  n} - 2^{n}\cdot3^{\frac{ n}{2}}}}{3^\frac{n}{2}}\in (0,\bar a_m),
\end{align*}
then we obtain by Theorem \ref{thm11} that for $\theta\in(0,\bar\theta_1)$, there exists a non-constant solution $u^*$ to (\ref{44}) with the leading order term is
 \begin{align}\label{ansatz1}
u^*_0=\sum_{m=1}^k \bar c_{m,\chi} e^{\chi(A-A_m)}\eta_m,
\end{align}
where $\bar c_{m,\chi}:=\bar c_0+O\big(\frac{1}{\sqrt{\chi}}\big)$ and $\bar c_0$ is either $\bar c_{01}$ or $\bar c_{02}$ which are given by
 \begin{align}\label{barc0}
 \bar c_{01}:=\frac{3^{\frac{n}{2}}\big(\bar a_m+\theta\big)+ \sqrt{\delta_1}}{ 2^{\frac{n}{2}+1}},\quad
\bar c_{02}:=\frac{3^{\frac{n}{2}}\big(\bar a_m+\theta\big)-\sqrt{\delta_1}}{ 2^{\frac{n}{2}+1}}.
    \end{align}
 Therefore, with the help of Theorem \ref{thm11}, for each fixed integer $k\geq 1$, we have the solution $u^*$ satisfying the form of \eqref{ustarprop41}.  Similarly, noting that $\bar c_{01}$ and $\bar c_{02}$ satisfies (\ref{barc0}), one finds $\bar c_{m,\chi}=\bar c_{01}+O(\frac{1}{\sqrt{\chi}})$ for $m\in\mathcal I_{b}$; $\bar c_{m,\chi}=\bar c_{02}+O(\frac{1}{\sqrt{\chi}})$ for $m\in\mathcal I_{s}\backslash\mathcal I_b$ and $\bar c_{m,\chi}=O(\frac{1}{\sqrt{\chi}})$ for $m\not\in\mathcal I_{s}$.    
 
Next, we assume $u=0$ and let $v=Cr$ with $C>0,\not=1$ being some undetermined constant.  Then we substitute them into the $v$-equation of (\ref{SSS}) and integrate it to obtain that $C$ satisfies
 \begin{align*}
\int_{\Omega}v(v-\theta)(r-v)d x=C(1-C)\int_{\Omega}r^2(Cr-\theta)dx=0.
\end{align*}
By solving it, one has $C=\beta\theta$ with $\beta:=\frac{\int_{\Omega}r^2 dx}{\int_{\Omega}r^3dx}$ for $\theta\in\big(0,\frac{1}{\beta}\big)$.  Hence, $(0,\beta\theta r)$ is a potential solution to (\ref{SSS}).  However, while substituting $v=\beta \theta r$ into the $v$-equation again, we find 
$$r\equiv \frac{\theta}{C}=\frac{\int_{\Omega}r^3 dx}{\int_{\Omega}r^2dx}~~\text{in~~}\Omega,$$
which implies $r$ must be a constant.  It is contradicted to the fact that $r=e^A$ is non-constant since assumption (A1) states $A$ is spatially heterogeneous.  Therefore, when $u=0$, we obtain $v$ only has the form $(0,r(x))$.

 
 
 To study the coexistence of steady states, we proceed to further analyze the balancing conditions.  If $v=r$, we claim that for $\theta\in(0,\bar a_m)$, (\ref{SSS}) does not admit the $u$-solution with the height is $O(1)$.  To this end, we focus on the following integral constraint satisfied by $u$:
 \begin{align}\label{48}
 \int_{\Omega}u(r-u-v)(u+v-\theta)dx=-\int_{\Omega}u^2(r+u-\theta)dx=0.
\end{align}
 The ansatz of $u$ is similar as in (\ref{ansatz1}), which is  
  \begin{align}\label{ansatz2}
u_0=\sum_{m=1}^k \tilde c_{m,\chi} e^{\chi(A-A_m)}\eta_m,
\end{align}
where $\tilde c_{m,\chi}=\tilde c_0+O(\frac{1}{\sqrt{\chi}})$.  Upon Substituting (\ref{ansatz2}) into (\ref{48}), we find that $\tilde c_0$ satisfies
\begin{align}\label{410}
\tilde c_0\int_{\mathbb R^n}e^{-3\vert y\vert^3}dy+\bar a_m\int_{\mathbb R^n} e^{-2\vert y\vert^2}dy-\theta \int_{\mathbb R^n} e^{-2\vert y\vert^2}dy=0.
\end{align}
Since $\theta\in(0,\bar a_m)$, it follows that (\ref{410}) does not admit any positive root $\bar c_0$.  This finishes our claim.  Therefore, when $v=r$, there is the only one steady state $(0,r)$ to (\ref{IFD}) for $\theta\in(0,\bar a_m)$.

We further claim that there exists $\theta^*>0$ such that for any $\theta\in(0,\theta^*)$, (\ref{SSS}) does not admit the positive spiky solution $(u^*_s,v_s^*)$.  If not, we denote $(u_s^*,v_s^*)$ as the steady states to (\ref{IFD}), then we have the following balancing conditions:
\begin{align}\label{twobalance}
\left\{\begin{array}{ll}
\int_{\Omega} u_s^*(r-u_s^*-v_s^*)(u_s^*+v_s^*-\theta)dx=0,\\
\int_{\Omega} v_s^*(r-u_s^*-v_s^*)(u_s^*+v_s^*-\theta) dx=0.
\end{array}
\right.
\end{align}
Similarly, we consider the leading order term of $u_s^*$ is $u_0^*=\tilde c_{0}e^{-\frac{1}{2}\sum\limits_{i=1}^nh_m^{(i)}\chi(x-x^{(i)}_m)^2}$ and $v_s^*$ has the form of $v_s^*=\tilde Cr$, where $\tilde C$ is a constant needs to be determined.  We substitute them into (\ref{twobalance}) to obtain
\begin{align}\label{413}
&\frac{(1-\tilde C)\tilde c_0}{\sqrt{\Pi_{i=1}^n{h_i}}\chi^{\frac{n}{2}}}\int_{\Omega_{\chi}}e^{-y^2-\frac{2}{\chi}\vert y\vert^2}dy+\frac{1}{\bar a_m}\int_{\Omega}(\tilde C r-\theta)(1-\tilde C)r^2dx\nonumber\\
&-\frac{\tilde c_0^2}{\sqrt{\Pi_{i=1}^n{h_i}}\chi^{\frac{n}{2}}}\int_{\Omega_{\chi}}e^{-2\vert y\vert^2-\frac{1}{\chi}\vert y\vert^2}dy-\frac{\tilde C\tilde c_0 \bar a_m}{\sqrt{\Pi_{i=1}^n{h_i}}\chi^{\frac{n}{2}}} \int_{\Omega_{\chi}}e^{-\vert y\vert^2}dy+\frac{\theta\bar c_0}{\chi^{\frac{n}{2}}\sqrt{\Pi_{i=1}^n{h_i}}}\int_{\Omega_{\chi}}e^{-\vert y\vert^2}dy=0.
\end{align}
Letting $\chi\rightarrow \infty,$ one further has
$$\int_{\Omega}(\tilde C r-\theta)(1-\tilde C)r^2dx=0,$$
which implies $\tilde C=\beta\theta$ if $\tilde C\not=1.$  We substitute $\tilde C=\beta\theta$ into (\ref{413}) again to conclude that
$\tilde c_0=2^{\frac{n}{2}}[1-(\bar a_m+1)\tilde C+\theta],$ which indicates $\tilde c_0\rightarrow 2^{\frac{n}{2}}$ as $\theta\rightarrow 0^+$.  Similarly, we calculate from the second equation in (\ref{twobalance}) that
$$2^{\frac{n}{2}}\tilde c_0^2-3^{\frac{n}{2}}\tilde c_0[(1-2\beta\theta)\bar a_m+\theta]+\theta(1-\beta\theta)(1-\beta)\bar a_m6^{\frac{n}{2}}=0.$$
In this way, we find $\tilde c_0\rightarrow \big(\frac{3}{2}\big)^{\frac{n}{2}}\bar a_m$ as $\theta \rightarrow 0^+,$ which reaches the contradiction.  This shows that one can find $\theta^*>0$ such that for any $\theta\in(0,\theta^*),$ (\ref{twobalance}) does not admit any positive solution $(\tilde c_0,\tilde C)$.  This completes the proof of our claim.  Proposition \ref{prop41} follows from summarizing our above arguments.
\end{proof}
After showing the existence of steady states in Theorem \ref{thm13}, we shall discuss their local dynamics. 
 
\subsubsection{Linearized Eigenvalue Problem and Stability Analysis}
To analyze the stability of steady states given in Proposition \ref{prop41}, we assume
\[u(x,t)=u_s(x)+\varepsilon e^{\lambda t} \varphi(x),\quad v(x,t)=v_s(x)+\varepsilon e^{\lambda t}\psi(x),\]
where $(u,v)$ is defined as the solution of (\ref{IFD}).  Upon substituting this into (\ref{IFD}), we obtain the following linearized eigenvalue problem
  \begin{align}\label{EPS}
\left\{\begin{array}{ll}
\lambda \varphi=\nabla\cdot(\nabla \varphi-\chi \varphi\nabla A)+u_sg_u\vert_{(u_s+v_s)}\varphi+u_sg_v\vert_{(u_s+v_s)}\psi+g\vert_{(u_s+v_s)}\varphi,&x\in\Omega,\\
\lambda \psi=\nabla\cdot(\nabla \psi-\psi\nabla A)+v_sg_v\vert_{(u_s+v_s)}\psi+v_sg_u\vert_{(u_s+v_s)}\varphi+g\vert_{(u_s+v_s)}\psi,&x\in\Omega,\\
\partial_{\textbf{n}} \varphi-\chi \varphi\partial_{\textbf{n}}A=\partial_{\textbf{n}}\psi-\psi\partial_{\textbf{n}}A=0,&x\in\partial \Omega.
\end{array}
\right.
  \end{align}
  By analyzing (\ref{EPS}) around steady states defined in Theorem \ref{thm13}, we obtain the following proposition:
  \begin{proposition}\label{prop42}
Assume all conditions and conclusions in Proposition \ref{prop41} hold.  Then we have when $\theta\in\big(0,\theta^{**}\big)$ with $\theta^{**}:=\min\Big\{\frac{\int_{\Omega}r^3 dx}{\int_{\Omega}r^2dx},\theta_1\Big\}$ and $\theta_1$ being defined in \eqref{theta1}, $(0,r)$ is linearly stable.  Moreover, if $\theta$ is independent of $\chi,$ the following alternatives hold for $(u^*,0)$:
\begin{itemize}
\item[(i).] if $\bar c_{m,\chi}=\bar c_{01}+O\big(\frac{1}{\sqrt{\chi}}\big)$ for all $m\in \mathcal I_{s}$, the solution will be linearly stable;
\item[(ii).] if $\bar c_{m,\chi}=\bar c_{02}+O\big(\frac{1}{\sqrt{\chi}}\big)$ for some $m$, the solution will be unstable;
\end{itemize}
however, if $\theta$ depends on $\chi$ and satisfies $\theta\in (0,\varepsilon^*_1/\chi^{\frac{n}{2}})$ with $\varepsilon_1^*$ defined by (\ref{varepsilonstar}), $(u^*,0)$ will be always unstable.
  \end{proposition}
  \begin{proof}
After considering $(u_s,v_s)=(0,r(x))$, (\ref{EPS}) becomes
  \begin{align}\label{416}
\left\{\begin{array}{ll}
\lambda \varphi=\nabla\cdot(\nabla \varphi-\chi \varphi\nabla A),&x\in\Omega,\\
\lambda \psi=\nabla\cdot(\nabla \psi-\psi\nabla A)+r(\theta-r)(\varphi+\psi),&x\in\Omega,\\
\partial_{\textbf{n}} \varphi-\chi \varphi\partial_{\textbf{n}}A=\partial_{\textbf{n}}\psi-\psi\partial_{\textbf{n}}A=0,&x\in\partial \Omega.
\end{array}
\right.
  \end{align}
We integrate the $\varphi$-equation in (\ref{416}), then obtain from $\lambda\not=0$ that $\int_{\Omega}\varphi dx=0.$
Furthermore, it follows from the same argument in Section \ref{sec3} that $\lambda$ is bounded.  
Then, we define $y=\frac{x}{\epsilon}$, $r(x)=\tilde r(y)$, $\varphi(x)=\varPhi(y)$ and $\psi(x)=\varPsi(y)$ to get (\ref{416}) has the following form:
    \begin{align}\label{ep421}
\left\{\begin{array}{ll}
\epsilon^2\lambda \varPhi=\nabla_y\cdot\Big(\nabla_y \varPhi-\frac{1}{\epsilon^2} \varPhi\nabla_y A\Big),&y\in\Omega_{\epsilon},\\
\epsilon^2\lambda \varPsi=\nabla_y\cdot(\nabla_y \varPsi-\varPsi\nabla_y A)+\epsilon^2r(\theta-r)(\varPhi+\varPsi),&y\in\Omega_{\epsilon},\\
\partial_{\textbf{n}} \varphi-\frac{1}{\epsilon^2} \varphi\partial_{\textbf{n}}A=\partial_{\textbf{n}}\psi-\psi\partial_{\textbf{n}}A=0,&x\in\partial \Omega.
\end{array}
\right.
  \end{align}
  Since $\lambda$ is bounded, we have from Proposition \ref{prop3} that $\varPhi=C_2 e^{-\frac{1}{\epsilon^2} A}+o(1)$ and $\varPsi=C_3e^{-A}+o(1)$ for $\epsilon$ small, where $C_2,C_3>0$ are constants.  Thanks to $\int_{\Omega}\varphi dx=0,$ one finds from $\varPhi=C_2 e^{-\frac{1}{\epsilon^2} A}+o(1)$ that $C_2=0$, which implies $\varPhi=0$ is a solution to the $\varPhi$-equation in (\ref{ep421}).  Let $\varPhi=0$ in the $\varPsi$-equation of (\ref{ep421}) and integrate it to get
 \begin{align*}
\lambda\int_{\Omega_{\epsilon}}\varPsi dy=\int_{\Omega_{\epsilon}}\tilde r(\theta-\tilde r)\varPsi dy.
 \end{align*}
By noting that $\varPsi=C_3 \tilde r+o(1)$, we further have 
 $$\lambda=\frac{\theta\int_{\Omega}r^2dx-\int_{\Omega} r^3dx}{\int_{\Omega} rdx}+O(\epsilon),$$
 which indicates that $\lambda<0$ for $\epsilon$ small due to $\theta\in\Big(0,\frac{\int_{\Omega}r^3 dx}{\int_{\Omega}r^2dx}\Big).$  Thus, the steady state $(0,r(x))$ is linear stable. 
 
 
 We next discuss the stability of the steady state $(u^*,0)$ defined in Theorem \ref{thm13}.  Note that the associated linearized problem is
   \begin{align}\label{EPS1}
\left\{\begin{array}{ll}
\lambda \varphi=\nabla\cdot(\nabla \varphi-\chi \varphi\nabla A)+u_s^*g_u(x,u_s^*)\varphi+u_s^*g_v(x,u_s^*)\psi+g(x,u_s^*)\varphi,&x\in\Omega,\\
\lambda \psi=\nabla\cdot(\nabla \psi-\psi\nabla A)+g(x,u_s^*)\psi,&x\in\Omega,\\
\partial_{\textbf{n}} \varphi-\chi \varphi\partial_{\textbf{n}}A=\partial_{\textbf{n}}\psi-\psi\partial_{\textbf{n}}A=0,&x\in\partial \Omega.
\end{array}
\right.
  \end{align}
Then we have from Proposition \ref{prop3} that there exists $\psi_1=Cr+o(1)$ which satisfies the $\psi$-equation of (\ref{EPS1}), where $C\geq 0$ needs to be determined.  Substitute it into the $\psi$-equation and integrate over $\Omega$ to obtain
  \begin{align}\label{425}
  &\lambda_1 C\int_{\Omega}rdx+\theta C \int_{\Omega} r^2 dx+o(1)\nonumber\\
  =&\sum_{m=1}^k\frac{\epsilon^nC}{{\sqrt{\Pi_{i=1}^n{h_i^m}}}}\Bigg(-\bar c_{0,m}^2\bar a_m\int_{\Omega_{\epsilon}} e^{-2\vert y\vert^2}dy+\bar c_{0,m}\bar a_m^2\int_{\Omega_{\epsilon}}e^{-\vert y\vert^2}dy+\theta\bar c_{0,m}\bar a_m\int_{\Omega_{\epsilon}}e^{-\vert y\vert^2}dy\Bigg).
  \end{align}
When $\theta$ is fixed and independent of $\epsilon$, if $C\not=0$, one finds $\lambda_1\rightarrow -\theta \frac{\int_{\Omega} r^2 dx}{\int_{\Omega} r dx}<0$ as $\epsilon\rightarrow 0.$  As a consequence, $\lambda_1\approx -\theta \frac{\int_{\Omega} r^2 dx}{\int_{\Omega} r dx}$ for $\epsilon$ small.  Otherwise, we suppose $C=0$ which indicates that $\psi=0$ is a solution to the $\psi$-equation.  Moreover, the $\varphi$-equation in (\ref{EPS1}) becomes
$$\lambda \varphi=\nabla\cdot(\nabla \varphi-\chi \varphi\nabla A)+u_s^*g_u(x,u_s^*)\varphi+g(x,u_s^*)\varphi.$$
We analyze this single equation and conclude from Proposition \ref{prop3} that for $i=2,\cdots,k+1$,
$$\varphi_i=\sum_{m=1}^k\big[e^{(i)}_{m,0}+o(1)\big]z_m+o(1),$$
where $z_m(x):=Z_m(y)$ and $e^{(i)}_{m,0}$ are non-negative constants but not identically zero; furthermore,
$$\lambda_i =\pi^{-\frac{n}{2}}\bar h'(\bar c_{i,0})+ O(1)\epsilon,$$
where $\bar c_{i,0}$ is $0$, $\bar c_{01}$ or $\bar c_{02}$; while $\bar h$ is given by
$$\bar h{(\xi)}:=\xi[-2^{\frac{n}{2}}\bar c_0^2+3^{\frac{n}{2}}\big(\bar a_m+\theta\big)\bar c_0-6^{\frac{n}{2}}\theta \bar a_m].$$
According to Proposition \ref{prop4}, when for any $i=2,\cdots,k+1$, $\bar c_{i,0}$ are either $\bar c_{01}$ or $0$, all eigenvalues are strictly negative for $\epsilon$ small; if there exists some $i$ such that $\bar c_{i,0}=\bar c_{02}$, the corresponding eigenvalues satisfy $\lambda_i>0.$  Combining the fact that $\lambda_1<0$ for $\epsilon$ small, one can show the stability properties of $(u^*,0)$.

Focusing on the case that $\theta$ depends on $\chi$, we claim that there exists some $\varepsilon_1^*$ such that when $\theta\in(0,\varepsilon_1^*\epsilon^n)$, the steady state $(u^*,0)$ are unstable.  To prove our claim, we suppose $\theta=\varepsilon_1\epsilon^n$ and substitute it into (\ref{425}) to obtain
\begin{align}\label{426}
  \epsilon^{-n}\lambda_{10}\int_{\Omega}rdx=\sum_{m=1}^l\frac{\pi^{\frac{n}{2}}\bar a_m}{{\sqrt{\Pi_{i=1}^n{h_i^m}}}}\Bigg(-\bar c_{0,m}^2 2^{-\frac{n}{2}}+\bar c_{0,m} \bar a_m\Bigg)+\sum_{m=1}^l\frac{\pi^{\frac{n}{2}}}{{\sqrt{\Pi_{i=1}^n{h_i^m}}}}{\varepsilon_1\epsilon^n \bar c_0 \bar a_m}-\varepsilon_1  \int_{\Omega} r^2 dx,
  \end{align}
 where $\lambda_{10}$ is defined as the leading term of $\lambda_1$.  Our goal is to find all $\varepsilon_1$ such that $\lambda_{10}>0$ for $\epsilon$ small.  To achieve it, we conclude from (\ref{426}) that it is equivalent to verify
 \begin{align}\label{verify1}
  -\bar c_{0,\infty}^2 2^{-\frac{n}{2}}+\bar c_{0,\infty}-\varepsilon_1  \frac{1}{\alpha_1\pi^{\frac{n}{2}}}\int_{\Omega} r^2 dx>0,
  \end{align}
  where $\bar c_{0,\infty}=(\frac{3}{2})^{\frac{n}{2}}$
  and $\alpha_1=\sum_{m=1}^l\frac{{\bar a}^3_m}{{\sqrt{\Pi_{i=1}^n{h_i^m}}}}.$  We define $C_4:=\frac{1}{\alpha_1\pi^{\frac{n}{2}}}\int_{\Omega} r^2 dx$ and $\tilde h(\xi)=-\xi^2 2^{-\frac{n}{2}}+\xi- C_4\varepsilon_1$, then (\ref{verify1}) is equivalent to $\tilde  h(\bar c_{0,\infty})>0$.  Since we have the fact that when $\varepsilon_1\in(0,\frac{2^{n/2}}{4C_4})$, there exist two roots $\xi_1^*$ and $\xi_2^*$ of $\tilde h(\xi)=0$ given by
  \[\xi^*_1:=\frac{1-\sqrt{\delta_2}}{2^{1-\frac{n}{2}}},\quad \xi^*_2:=\frac{1+\sqrt{\delta_2}}{2^{1-\frac{n}{2}}},\]
  where $\delta_2:=1-2^{2-\frac{n}{2}} C_4\varepsilon_1,$
  our desired conclusion is equivalent to find all $\varepsilon_1$ such that $\xi_2^*>\bar c_{0,\infty}$, which is
  $$2\cdot 3^{n/2}<2^n+2^n\sqrt{1-\frac{4}{2^{n/2}}\varepsilon_1C_4}.$$
  By straightforward calculation, we obtain that when $\varepsilon_1\in(0,\varepsilon_1^*)$, $\lambda_{10}>0$ for $\epsilon$ small, where $\varepsilon^*_1$ is defined by
  \begin{align}\label{varepsilonstar}
  \varepsilon_1^*=\frac{2^{n/2}\cdot4^n-2^{n/2}\big(2\cdot 3^{n/2}-2^n\big)^2}{4\cdot 4^{n} C_4},
  \end{align}
  which completes the proof of our claim.  One can summarize the above arguments to show that Proposition \ref{prop42} holds.
\end{proof} 
Theorem \ref{thm13} follows from Proposition \ref{prop41} and Proposition \ref{prop42}.  On the one hand, we have from Theorem \ref{thm13} that when Allee threshold is sufficiently small, the interacting species who follow the aggressive strategy and the IFD strategy can not coexist.  On the other hand, Theorem \ref{thm13} illustrates that the species who follows the IFD strategy is likely to be better off since $(0,r)$ is linearly stable.  It is worthy mentioning that when Allee threshold is small but independent of $\chi$, the final pattern formations really depend on the height of every peak.  Whereas we can observe that in some situations, the aggressive species is extinct finally no matter what heights the peaks have.

 \subsection{Aggressive Strategies and Allee Effect}
 In this section, we focus on the existence and dynamics of spiky steady states to system (\ref{AA}).  From the viewpoint of biological background, (\ref{AA}) indicates that two species both take the aggressive strategy and one of them is the more aggressive.  Similarly, we first investigate the following stationary problem of (\ref{AA}):
  \begin{align}\label{SSfinal}
\left\{\begin{array}{ll}
0=\nabla\cdot(\nabla u-\chi u\nabla A)+ug(x,u+v),&x\in\Omega,\\
0=\nabla\cdot(\nabla v-c\chi v\nabla A)+vg(x,u+v),&x\in\Omega,\\
\partial_{\textbf{n}} u-\chi u\partial_{\textbf{n}}A=\partial_{\textbf{n}}v-c\chi v\partial_{\textbf{n}}A=0,&x\in\partial \Omega,
\end{array}
\right.
  \end{align}
  where $c>1$ is a constant, $\chi$ reflects the speed of the species and $A$ is the signal; while the growth pattern $g$ is given by
  $$g(x,u+v)=(u+v-\theta)(1-u-v).$$ 
To construct non-constant solutions to (\ref{AA}), we first consider the simpler cases, which are $u_s=0$ or $v_s=0$ in every local region, next analyze the coexistence of spiky steady states.
\subsubsection{Existence of Non-constant Steady States}
Our results are summarized as follows:
\begin{proposition}\label{prop43}
Assume all conditions of $A$ shown in Theorem \ref{thm11} hold. 
 Given $g(x,u+v):=(1-u-v)(u+v-\theta)$ in system (\ref{SSfinal}).  Let $k$ be any positive but fixed integer, $\mathcal I_{s1}, \mathcal I_{s2}\subseteq \{1,\cdots,k\}$ be any subset and $\mathcal I_{b1}$, $\mathcal I_{b2}$ be any subset of $\mathcal I_{s1}$ and $\mathcal I_{s2}$, respectively.  Suppose at least one of subsets $\mathcal I_{s1}$ and $\mathcal I_{s2}$ is nonempty, then if $\mathcal I_{s1}\cap\mathcal I_{s2}=\emptyset$, we have when $\chi\gg 1$ and $\theta\in(0,\theta_1)$, there exist the steady states $(\bar u^*,\bar v^*)$ defined as
\begin{align}\label{firstcoexist}
\left\{\begin{array}{ll}
\bar u^*(x;\chi)=\sum\limits_{m=1}^k S^{(1)}_{m,\chi} e^{-\frac{1}{2}\sum\limits_{i=1}^nh_m^{(i)}\chi(x-x^{(i)}_m)^2}+o(1);\\
 \bar v^*(x;\chi)=\sum\limits_{m=1}^k S^{(2)}_{m,\chi} e^{-\frac{c}{2}\sum\limits_{i=1}^nh_m^{(i)}\chi(x-x^{(i)}_m)^2}+o(1),
 \end{array}
 \right.
 \end{align}
 where $o(1)\rightarrow 0$ uniformly as $\chi \rightarrow \infty$.  In particular, if $m\in\mathcal I_{bi},$ $S^{(i)}_{m,\chi}=c_{01}+O(\frac{1}{\sqrt{\chi}})$; if $m\in \mathcal I_{si}\backslash\mathcal I_{bi}$, $S^{(i)}_{m,\chi}=c_{02}+O(\frac{1}{\sqrt{\chi}})$; if $m\not\in\mathcal I_{si},$ $S^{(i)}_{m,\chi}=O(\frac{1}{\sqrt{\chi}})$, where $i=1, 2$ and $c_{01},$ $c_{02}$ are defined in (\ref{c0}).

 Moreover, if $\mathcal I_{s1}\cap \mathcal I_{s2}\not=\emptyset$, there coexist many possible non-constant spiky solutions $(\hat u^*,\hat v^*)$ to (\ref{SSfinal}), in which $(\hat u^*,\hat v^*)$ have the same form as $(\bar u^*,\bar v^*)$ with $S_{m,\chi}^{(1)}$ and $S_{m,\chi}^{(2)}$ being replaced by $\bar S_{m,\chi}^{(1)}$ and $\bar S_{m,\chi}^{(2)}$, respectively.  In particular, if $m\in \mathcal I_{s1}\cap \mathcal I_{s2},$ $\bar S^{(i)}_{m,\chi}= S_i^*+O(\frac{1}{\sqrt{\chi}})$; if $m\in \mathcal I_{si}\backslash\{\mathcal I_{s1}\cap \mathcal I_{s2}\},$ $\bar S^{(i)}_{m,\chi}=c_{0i}+O(\frac{1}{\sqrt{\chi}})$; if $m\not\in\mathcal I_{si},$ $\bar S^{(i)}_{m,\chi}=O(\frac{1}{\sqrt{\chi}})$, where $i=1, 2$ and $S_1^*$ and $S_2^*$ are determined by (\ref{balancingsystem1}).

\end{proposition}
\begin{proof}
We firstly assume that either $S^{(1)}_{m,\chi}= O(\frac{1}{\sqrt{\chi}})$ or $S^{(2)}_{m,\chi}= O(\frac{1}{\sqrt{\chi}})$ in the neighborhood of every $m$-th non-degenerate maximum point. 
Thanks to Theorem \ref{thm11}, we have for $\chi$ large, there exists the non-constant solution $(\bar u^*,\bar v^*)$ given by (\ref{firstcoexist}) to (\ref{SSfinal}).

Next, we focus on the coexistence of steady states in the local bump.
It suffices to consider the following balancing conditions:
\begin{align}\label{balancingsystem}
\left\{\begin{array}{ll}
\int_{\mathbb R^n}e^{-\vert y\vert^2}(S_1 e^{-\vert y\vert^2}+S_2 e^{-c\vert y\vert^2}-\theta)(1-S_1 e^{-\vert y\vert^2}-S_2 e^{-c\vert y\vert^2})dy=0,\\
\int_{\mathbb R^n}e^{-c\vert y\vert^2}(S_1 e^{-\vert y\vert^2}+S_2 e^{-c\vert y\vert^2}-\theta)(1-S_1 e^{-\vert y\vert^2}-S_2 e^{-c\vert y\vert^2})dy=0.
\end{array}
\right.
\end{align}
By straightforward calculation, (\ref{balancingsystem}) can be simplified as:
\begin{align}\label{balancingsystem1}
\left\{\begin{array}{ll}
I_1(S_1,S_2)=0,\\
I_2(S_1,S_2)=0,
\end{array}
\right.
\end{align}
where $I_1$ and $I_2$ are defined by
$$I_1:=-3^{-\frac{n}{2}}S_1^2-2(c+2)^{-\frac{n}{2}}S_1S_2-(2c+1)^{-\frac{n}{2}}S_2^2+2^{-\frac{n}{2}}(1+\theta)S_1+(c+1)^{-\frac{n}{2}}(1+\theta)S_2-\theta$$
and
$$I_2:=-(c+2)^{-\frac{n}{2}}S_1^2-2(2c+1)^{-\frac{n}{2}}S_1S_2-(3c)^{-\frac{n}{2}}S_2^2+(c+1)^{-\frac{n}{2}}(1+\theta)S_1+(2c)^{-\frac{n}{2}}(1+\theta)S_2-c^{-\frac{n}{2}}\theta.$$
By using $S_2$ to express $S_1$, we have from $I_1(S_1,S_2)=0$   that $S_1=g_1(S_2)$ or $S_1=g_2(S_2)$, where 
\begin{align}\label{I1term}
g_1(S_2):=\frac{\big({\frac{3}{2}}\big)^{\frac{n}{2}}(1+\theta)-2\big(\frac{3}{c+2}\big)^{\frac{n}{2}}S_2+ 3^{\frac{n}{2}}\sqrt{\Delta_1}}{2};\nonumber\\ g_2(S_2):=\frac{\big({\frac{3}{2}}\big)^{\frac{n}{2}}(1+\theta)-2\big(\frac{3}{c+2}\big)^{\frac{n}{2}}S_2- 3^{\frac{n}{2}}\sqrt{\Delta_1}}{2},
\end{align}
and $\Delta_1$ is given by
\begin{align*}
\Delta_1(S_2):=&[2^{-\frac{n}{2}}(1+\theta)-2(c+2)^{-\frac{n}{2}}S_2]^2-4\cdot 3^{-\frac{n}{2}}[(2c+1)^{-\frac{n}{2}}S_2^2-(c+1)^{-\frac{n}{2}}(1+\theta)S_2+\theta ]\\
=&-4\bigg[\bigg(\frac{1}{6c+3}\bigg)^{\frac{n}{2}}-\frac{1}{(c+2)^n}\bigg]S_2^2+4(1+\theta)[(3c+3)^{-\frac{n}{2}}-(2c+4)^{-\frac{n}{2}}]S_2\\
&+2^{-n}\theta^2-(4\cdot 3^{-\frac{n}{2}}-2^{1-n})\theta+2^{-n}.
\end{align*}
Similarly, one finds from $I_2(S_1,S_2)=0$ that $S_1=g_3(S_2)$ or $S_1=g_4(S_2)$, where 
\begin{align}\label{I2term}
g_3(S_2)=\frac{\big(\frac{c+2}{c+1}\big)^{\frac{n}{2}}(1+\theta)-2\big(\frac{c+2}{2c+1}\big)^{\frac{n}{2}}S_2+(c+2)^{\frac{n}{2}}\sqrt{\Delta_2}}{2};\nonumber\\ g_4(S_2)=\frac{\big(\frac{c+2}{c+1}\big)^{\frac{n}{2}}(1+\theta)-2\big(\frac{c+2}{2c+1}\big)^{\frac{n}{2}}S_2-(c+2)^{\frac{n}{2}}\sqrt{\Delta_2}}{2},
\end{align}
and $\Delta_2$ is defined as
\begin{align*}
\Delta_2(S_2):=&[(c+1)^{-\frac{n}{2}}(1+\theta)-2(2c+1)^{-\frac{n}{2}}S_2]^2-4\cdot (c+2)^{-\frac{n}{2}}[(3c)^{-\frac{n}{2}}S_2^2-(2c)^{-\frac{n}{2}}(1+\theta)S_2+c^{-\frac{n}{2}}\theta ]\\
=&-4\bigg[\bigg(\frac{1}{3c^2+6c}\bigg)^{\frac{n}{2}}-\frac{1}{(2c+1)^n}\bigg]S_2^2-4(1+\theta)[(c+1)^{-\frac{n}{2}}(2c+1)^{-\frac{n}{2}}-(2c)^{-\frac{n}{2}}(c+2)^{-\frac{n}{2}}]S_2\\
&+(c+1)^{-n}(1+\theta)^2-4(c^2+2c)^{-\frac{n}{2}}\theta.
\end{align*}
To find the positive solution $(S_1^*,S_2^*)$ of the algebraic system (\ref{balancingsystem1}), it is equivalent to prove there exists $S^*_2$ such that one of the following statements holds:
\begin{align}\label{fourcase}
(i).\text{~~} g_1(S^*_2)=g_3(S^*_2);\quad (ii).~~g_1(S^*_2)=g_4(S^*_2);\nonumber\\ (iii).~~g_2(S^*_2)=g_3(S^*_2);\quad (iv).~~g_2(S^*_2)=g_4(S^*_2).
\end{align}
We only discuss case (i) and others can be similarly analyzed.  It is necessary to first study the properties of $g_1$ and $g_3$.  We claim $g_1(S_2)$ is a  decreasing function with respect to $S_2$.  Indeed, we have from (\ref{I1term}) that
\[\frac{d g_1}{dS_2}=-\Big(\frac{3}{c+2}\Big)^{\frac{n}{2}}+\frac{3^{\frac{n}{2}}\Delta'_1}{4\sqrt{\Delta_1}},\]
where
\[\Delta_1'=-8S_2[ (6c+3)^{-\frac{n}{2}}-(c+2)^{-n}]-4(1+\theta)[(2c+4)^{-\frac{n}{2}}-(3c+3)^{-\frac{n}{2}}]<0\]
due to $c>1$, and thereby $g'_1(S_2)<0,$ which proves our claim.  Similarly, we obtain from (\ref{I2term}) that $g_3$ is decreasing.  Note that 
\[g_1(0)=\Big(\frac{3}{2}\Big)^{\frac{n}{2}}(1+\theta)+3^{\frac{n}{2}}\sqrt{2^{-n}(1+\theta)^2-4\cdot 3^{-\frac{n}{2}}\theta },\]
and
$$g_3(0)=\Big(\frac{c+2}{c+1}\Big)^{\frac{n}{2}}(1+\theta)+(c+2)^{\frac{n}{2}}\sqrt{(c+1)^{-n}(1+\theta)^2-4\cdot (c^2+2c)^{-\frac{n}{2}}\theta},$$
then we have $g_1(0)\approx 2\Big(\frac{3}{2}\Big)^{\frac{n}{2}}> 2\Big(\frac{c+2}{c+1}\Big)^{\frac{n}{2}}\approx g_3(0)$ for $\theta$ small in light of $c>1$, which implies there exists $\bar\theta^*$ such that $g_1(0)>g_3(0)$ for all $\theta\in(0,\bar\theta^*)$.  Assume $c-1$ is small, then 
we expand $\Delta_1$ and $\Delta_2$ to obtain
\begin{align*}
\Delta_1=&-4(1+\theta)(c-1)2^{-\frac{n}{2}-2}\cdot 3^{-\frac{n}{2}-1}n S_2+2^{-n}\theta^2-(4\cdot 3^{-\frac{n}{2}}-2^{1-n})\theta+2^{-n}+O((c-1)^2),
\end{align*}
and
\begin{align*}
\Delta_2=&-4(1+\theta)(c-1)2^{-\frac{n}{2}-2}\cdot 3^{-\frac{n}{2}-1}n S_2+2^{-n}\theta^2-(4\cdot 3^{-\frac{n}{2}}-2^{1-n})\theta+2^{-n}\nonumber\\
&-(c-1)2^{-n-1}n(1+\theta)^2+8(c-1)3^{-\frac{n}{2}-1}n\theta+O((c-1)^2).
\end{align*}
We define 
$$\beta_2:=2^{-n}\theta^2-(4\cdot 3^{-\frac{n}{2}}-2^{1-n})\theta+2^{-n}$$
and further expand $\sqrt{\Delta_1}$ and $\sqrt{\Delta_2}$ to get
\begin{align}\label{delta1expand}
\sqrt{\Delta_1}=\sqrt{\beta_2}-\frac{1}{\sqrt{\beta_2}}(1+\theta)(c-1)2^{-\frac{n}{2}-1}\cdot 3^{-\frac{n}{2}-1}n S_2+O((c-1)^2),
\end{align}
and 
\begin{align}\label{delta2expand}
\sqrt{\Delta_2}=&\sqrt{\beta_2}-\frac{1}{\sqrt{\beta_2}}(1+\theta)(c-1)2^{-\frac{n}{2}-1}\cdot 3^{-\frac{n}{2}-1}n S_2\nonumber\\
&-\frac{1}{\sqrt{\beta_2}}(c-1)2^{-n-2}n(1+\theta)^2\frac{4}{\sqrt{\beta_2}}(c-1)3^{-\frac{n}{2}-1}n\theta+O((c-1)^2).
\end{align}
Substitute (\ref{delta1expand}) and (\ref{delta2expand}) into $g_1$ and $g_3$ to obtain
\begin{align}\label{g3ming1}
g_3(S_2)-g_1(S_2)= \frac{1}{2\sqrt{\beta_2}}(c-1)^33^{-\frac{n}{2}-3}n(n+2)S_2^2+\gamma(\theta)(c-1)+O((c-1)^4),
\end{align}
where
$$\gamma(\theta):=-2^{-\frac{n}{2}-3}3^{\frac{n}{2}-1}n(1+\theta)
-\frac{3^{\frac{n}{2}}}{2\sqrt{\beta_2}}2^{-n-2}n(1+\theta)^2+\frac{2}{3\sqrt{\beta_2}}n\theta\\
+\frac{\sqrt{\beta_2}n 3^{\frac{n}{2}-1}}{4}<0,$$
since $g_1(0)>g_3(0)$ for $\theta$ small.  
In light of $\gamma(\theta)\rightarrow 0$ as $\theta\rightarrow 0$, 
$\frac{1}{2\sqrt{\beta_2}}3^{-\frac{n}{2}-3}n(n+2)>0$ and $c>1$, we have it is possible to adjust $\theta$ in (\ref{g3ming1}) such that there exists $S_{2max}$ which satisfies $g_3(S_{2max})-g_1(S_{2max})>0$.  By the existence theorem of zeros, we obtain there exists $S_2^*$ such that $g_3=g_1.$
Therefore, we have there exists some $\bar c^*$ such that when $c\in (1,\bar c^*)$, (\ref{balancingsystem1}) admits the positive solution $(S_1^*,S_2^*)$ satisfying $g_3=g_1.$  This finishes the proof of the coexistence of non-constant solutions and Proposition \ref{prop43}.
\end{proof}
 \subsubsection{Study of Linearized Eigenvalue Problem}
We proceed to discuss the stability of steady states given in Proposition \ref{prop43}.  The linearized eigenvalue problem of (\ref{AA}) around the spiky solutions $(u_s,v_s)$ is:
  \begin{align}\label{EPSfinal}
\left\{\begin{array}{ll}
\lambda \bar \varphi=\nabla\cdot(\nabla \bar \varphi-\chi \bar \varphi\nabla A)+u_sg_u\vert_{(u_s+v_s)}\bar \varphi+u_sg_v\vert_{(u_s+v_s)}\bar \psi+g\vert_{(u_s+v_s)}\bar\varphi,&x\in\Omega,\\
\lambda \bar\psi=\nabla\cdot(\nabla \bar \psi- c\chi\bar \psi\nabla A)+v_sg_v\vert_{(u_s+v_s)}\bar\psi+v_sg_u\vert_{(u_s+v_s)}\bar \varphi+g\vert_{(u_s+v_s)}\bar\psi,&x\in\Omega,\\
\partial_{\textbf{n}} \bar \varphi-\chi \bar \varphi\partial_{\textbf{n}}A=\partial_{\textbf{n}}\bar \psi-c\chi \bar \psi\partial_{\textbf{n}}A=0,&x\in\partial \Omega.
\end{array}
\right.
  \end{align}
 By studying the signs of eigenvalues in (\ref{EPSfinal}), we establish the following proposition:
  \begin{proposition}\label{prop44}
Assume all conditions and conclusions in Proposition \ref{prop43} hold.  
If there exists $m$ such that $S^{(1)}_{m,\chi}=c_{02}+O(\frac{1}{\sqrt{\chi}})$ or $S^{(2)}_{m,\chi}=c_{02}+O(\frac{1}{\sqrt{\chi}})$, then we have the steady states $(\bar u^*,\bar v^*)$ are unstable; otherwise they are linearly stable.  Furthermore, there exist $c^*$ and $\tilde \theta^*$ such that for $c\in(1,c^*)$ and $\theta\in(0,\tilde \theta^*),$ the other steady states $(\hat u^*,\hat v^*)$ are always unstable. 
  \end{proposition}
\begin{proof}
To study the stability of $(\bar u^*,\bar v^*)$, without loss of generality, we assume $\bar v^*\sim 0$ in the cut-off region $\Omega_{\eta}$ containing $x_m$.  Since we have the fact that
\begin{align*}
g_u(x,u+v)=1+\theta-2(u+v),\quad g_u(x,u+v)=g_v(x,u+v),
\end{align*}
then the linearized eigenvalue problem of steady states $(\bar u^*,\bar v^*)$ becomes
  \begin{align}\label{EPSfinal1}
\left\{\begin{array}{ll}
\lambda \bar \varphi=\nabla\cdot(\nabla \bar \varphi-\chi \bar \varphi\nabla A)+\bar u^*(1+\theta-2\bar u^*)\bar \varphi+\bar u^*(1+\theta-2\bar u^*)\bar \psi+g(x,\bar u^*)\bar\varphi,&x\in\Omega_{\eta},\\
\lambda \bar\psi=\nabla\cdot(\nabla \bar \psi- c\chi\bar \psi\nabla A)+g(x,\bar  u^*)\bar\psi,&x\in\Omega_{\eta}.
\end{array}
\right.
  \end{align}
It is similar as the proof in Proposition \ref{prop42}, we have the eigenvector $\bar \psi=0$ in $\Omega_{\eta}$ for $\chi$ large.  
We next simplify the $\varphi$-equation in (\ref{EPSfinal1}) to get
\begin{align}\label{new1}
\lambda \bar \varphi=\nabla\cdot(\nabla \bar \varphi-\chi \bar \varphi\nabla A)+\bar u^*(1+\theta-2\bar u^*)\bar \varphi+g(x,\bar u^*)\bar\varphi,\quad x\in\Omega_{\eta}.
\end{align}
By applying the same argument in Section \ref{sec3} to (\ref{new1}), we can determine the signs of eigenvalues corresponding to the local cut-off regions, which imply the stability properties of $(\bar u^*,\bar v^*)$ stated in Proposition \ref{prop44}.  


The rest proof is devoted to the stability analysis of the steady states $(\hat u^*,\hat v^*)$.  Thanks to Proposition \ref{prop3}, we have for $\chi$ large, the eigen-vectors $(\bar\varphi,\bar\psi)$ satisfy $\bar \varphi\approx C_{2,m} e^{-\chi \sum_{i=1}^n h_m^{(i)}( x^{(i)}-x_m^{(i)})^2 }$ and $\bar \psi\approx C_{3,m} e^{-c\chi \sum_{i=1}^n h_m^{(i)}( x^{(i)}-x_m^{(i)})^2 }$ in the cut-off region containing $x_m$, where $C_{2,m}$ and $C_{3,m}$ are positive constants.  
We substitute them into (\ref{EPSfinal}) and integrate the $\bar\varphi$-equation and the $\bar\psi$-equation over the cut-off region $\Omega_{\eta}$ to obtain
\begin{align}\label{438}
\left\{\begin{array}{ll}
C_{2,m}\int_{\Omega_{\eta}} \hat u^* g_u\vert_{(\hat u^*+\hat v^*)}\hat u^*dx+C_{3,m}\int_{\Omega_{\eta}} \hat u^* g_u\vert_{(\hat u^*+\hat v^*)}\hat v^*dx=\lambda_m C_{2,m}\int_{\Omega_{\eta}} \hat u^*dx\\
C_{2,m}\int_{\Omega_{\eta}}\hat v^* g_u\vert_{(\hat u^*+\hat v^*)}\hat u^*dx+C_{3,m}\int_{\Omega_{\eta}}\hat v^* g_u\vert_{(\hat u^*+\hat v^*)}\hat v^*dx=\lambda_m C_{3,m}\int_{\Omega_{\eta}} \hat v^*dx.
\end{array}
\right.
\end{align}
We have from (\ref{438}) that the signs of eigenvalues $\lambda_m$ are determined by the following matrices:
\[B_m:=\left( \begin{array}{cc}
      \int_{\Omega_{\eta}} \hat u^* g_u\vert_{(\hat u^*+\hat v^*)}\hat u^*dx &\int_{\Omega_{\eta}} \hat u^* g_u\vert_{(\hat u^*+\hat v^*)}\hat v^*dx  \\
      \int_{\Omega_{\eta}} \hat v^* g_u\vert_{(\hat u^*+\hat v^*)}\hat u^*dx& \int_{\Omega_{\eta}} \hat v^* g_u\vert_{(\hat u^*+\hat v^*)}\hat v^*dx \\
    \end{array}
  \right).
\]
In particular, if the matrix $B_m$ has a positive eigenvalue, then we have the corresponding $\lambda_m$ satisfies $\lambda_m>0$, which implies the instability of $(\hat u^*,\hat v^*).$ We next analyze the properties satisfied by the matrix $B_m$.  On the one hand, we define $y=\frac{x-x_m}{\epsilon}$ to obtain that for $\chi$ large,
\begin{align}\label{440}
 (B_m)_{1,1}=\int_{\Omega_{\eta}} \hat u^* g_u\vert_{(\hat u^*+\hat v^*)}\hat u^*dx\approx D_1{S_1^*}^2 [(1+\theta)2^{-\frac{n}{2}}-2S^*_13^{-\frac{n}{2}}-2S^*_2(2+c)^{-\frac{n}{2}}],
\end{align}
\begin{align}\label{441}
 (B_m)_{1,2}=(B_m)_{2,1}&=\int_{\Omega_{\eta}} \hat v^* g_u\vert_{(\hat u^*+\hat v^*)}\hat u^*dx\nonumber\\
 &\approx S_1^*S_2^* D_2[(1+\theta)(c+1)^{-\frac{n}{2}}-2S^*_1(c+2)^{-\frac{n}{2}}-2S^*_2(2c+1)^{-\frac{n}{2}}],
\end{align}
and
\begin{align}\label{442}
(B_m)_{2,2}=\int_{\Omega_{\eta}} \hat v^* g_u\vert_{(\hat u^*+\hat v^*)}\hat v^*dx\approx 
{S_2^*}^2 D_3[(1+\theta)(2c)^{-\frac{n}{2}}-2S^*_1(2c+1)^{-\frac{n}{2}}-2S^*_2(3c)^{-\frac{n}{2}}],
\end{align}
where $D_1$, $D_2$, $D_3>0$ are constants.
On the other hand, one utilizes the same arguments in Section \ref{sec3} to obtain that 
\begin{align}\label{4411}
(B_m)_{1,1}=\frac{\partial I_1}{\partial S_1}\bigg\vert_{(S_1^*,S_2^*)},~ (B_m)_{1,2}=(B_m)_{2,1}=\frac{\partial I_1}{\partial S_2}\bigg\vert_{(S_1^*,S_2^*)}\text{~and~}(B_m)_{2,2}=\frac{\partial I_2}{\partial S_2}\bigg\vert_{(S_1^*,S_2^*)}.  
\end{align}
The signs of $\lambda_m$ can be determined by the traces and determinants of $B_m$, which are
\[ \text{Tr}(B_{m})= (B_{m})_{1,1}+(B_{m})_{2,2}.\quad \text{Det}(B_m)=(B_{m})_{1,1}\cdot (B_{m})_{2,2}-((B_{m})_{1,2})^2.\]
(\ref{4411}) implies $\text{Tr}(B_m)$ and $\text{Det}(B_m)$ can be rewritten as
$$\text{Tr}(B_m)=\frac{\partial I_1(S_1,S_2)}{\partial S_1}\bigg\vert_{(S_1^*,S_2^*)}+\frac{\partial I_2(S_1,S_2)}{\partial S_2}\bigg\vert_{(S_1^*,S_2^*)}$$
and 
$$\text{Det}(B_m)=\frac{\partial I_1(S_1,S_2)}{\partial S_1}\bigg\vert_{(S_1^*,S_2^*)}\cdot\frac{\partial I_2(S_1,S_2)}{\partial S_2}\bigg\vert_{(S_1^*,S_2^*)}-\Bigg(\frac{\partial I_1(S_1,S_2)}{\partial S_2}\bigg\vert_{(S_1^*,S_2^*)}\Bigg)^2.$$
Thanks to Proposition \ref{prop43}, we have the fact that for some $m$, $(S_1^*,S_2^*)$ are determined by one of four cases in (\ref{fourcase}).  When one of case (ii), case (iii) and case (iv) in (\ref{fourcase}) holds, we claim that the corresponding steady states $(\hat u^*,\hat v^*)$ are unstable.  We only exhibit the proof for case (ii) and the arguments are the same in other two cases.  
Since $(S_1^*,S_2^*)$ is determined by $g_1(S_2)=g_4(S_2)$, 
we have from the same arugment in Proposition \ref{prop42} that $\frac{\partial I_1}{\partial S_1}\big\vert_{(S_1^*,S_2^*)}<0$ and $\frac{\partial I_2}{\partial S_2}\big\vert_{(S_1^*,S_2^*)}>0,$ which implies $\text{Det}(B_m)<0$ and whereby we can show the matrix $B_m$ admits one positive eigenvalue.  This finishes the proof of our claim.

Now, we would like to analyze the stability of $(\hat u^*,\hat v^*)$ when case (i) holds.  First of all, it follows from $g_1(S^*_2)=g_3(S^*_2)$ that $\frac{\partial I_1}{\partial S_1}\big\vert_{(S_1^*,S_2^*)}$, $\frac{\partial I_2}{\partial S_2}\big\vert_{(S_1^*,S_2^*)}<0,$ which implies $\text{Tr}(B_m)<0$.  To investigate the sign of $\text{Det}(B_m)$, we calculate by using (\ref{440}), (\ref{441}) and (\ref{442}) to get
\begin{align}\label{detbm}
\text{Det}(B_m)=&4\Bigg(\frac{1}{(6c+3)^{\frac{n}{2}}}-\frac{1}{(c+2)^n}\Bigg){S_1^*}^2+4S_1^*S_2^*\bigg[\frac{1}{(9c)^{\frac{n}{2}}}-\frac{1}{[(2+c)(2c+1)]^{\frac{n}{2}}}\bigg]\nonumber\\
&+2\Bigg[-\frac{1}{ 6^{\frac{n}{2}}c^{\frac{n}{2}}}-\frac{1}{(4c+2)^{\frac{n}{2}}}+\frac{2}{(c+2)^{\frac{n}{2}}(c+1)^{\frac{n}{2}}}\Bigg](1+\theta)S_1^*\nonumber\\
&+4\Bigg[\frac{1}{(6c+3c^2)^{\frac{n}{2}}}-\frac{1}{(2c+1)^n}\Bigg]{S_2^*}^2+\Bigg(\frac{1}{(4c)^{\frac{n}{2}}}-\frac{1}{(c+1)^n}\Bigg)(1+\theta)^2\nonumber\\
&+2(1+\theta)S_2^*\Bigg[\frac{2}{(2c+1)^{\frac{n}{2}}(c+1)^{\frac{n}{2}}}-\frac{1}{c^{\frac{n}{2}}(4+2c)^{\frac{n}{2}}}-\frac{1}{6^{\frac{n}{2}}c^{\frac{n}{2}}}\Bigg]:=I_3(S_1^*,S_2^*).
\end{align}
Define 
\begin{align*}
\alpha_1&=4\Bigg(\frac{1}{(c+2)^n}-\frac{1}{(6c+3)^{\frac{n}{2}}}\Bigg),\quad  \alpha_2=4\bigg[\frac{1}{[(2+c)(2c+1)]^{\frac{n}{2}}}-\frac{1}{(9c)^{\frac{n}{2}}}\bigg],\\
\alpha_3&=2\Bigg[\frac{1}{ 6^{\frac{n}{2}}c^{\frac{n}{2}}}+\frac{1}{(4c+2)^{\frac{n}{2}}}-\frac{2}{(c+2)^{\frac{n}{2}}(c+1)^{\frac{n}{2}}}\Bigg](1+\theta)
\end{align*}
and
\begin{align*}
\alpha_4&=4\Bigg[\frac{1}{(2c+1)^n}-\frac{1}{(6c+3c^2)^{\frac{n}{2}}}\Bigg],\quad  \alpha_5=2(1+\theta)\Bigg[-\frac{2}{(2c+1)^{\frac{n}{2}}(c+1)^{\frac{n}{2}}}+\frac{1}{c^{\frac{n}{2}}(4+2c)^{\frac{n}{2}}}+\frac{1}{6^{\frac{n}{2}}c^{\frac{n}{2}}}\Bigg],\\
\alpha_6&=\Bigg(\frac{1}{(c+1)^n}-\frac{1}{(4c)^{\frac{n}{2}}}\Bigg)(1+\theta)^2,
\end{align*}
to rewrite (\ref{detbm}) as 
$$I_3(S_1^*,S_2^*)=-\alpha_1 {S_1^*}^2-\alpha_2S_1^*S_2^*-\alpha_3S_1^*-\alpha_4{S_2^*}^2-\alpha_5 S_2^*-\alpha_6.$$  We claim that $\text{Det}(B_m)=I_3(S_1^*,S_2^*)<0$.  To prove this, we solve $I_3(S_1,S_2)=0$ to obtain
\begin{align}\label{S11star}
S_{11}=\frac{-(\alpha_2 S_2+\alpha_3)-\sqrt{\Delta_3}}{2\alpha_1},\quad S_{12}=\frac{-(\alpha_2 S_2+\alpha_3)+\sqrt{\Delta_3}}{2\alpha_1},
\end{align}
where $\Delta_3$ is defined by
\begin{align*}
\Delta_3:=(\alpha^2_2-4\alpha_1\alpha_4)S_2^2+2(\alpha_2\alpha_3-2\alpha_1\alpha_5)S_2+\alpha_3^2-4\alpha_1\alpha_6.
\end{align*}
Moreover, we expand $\alpha_i$ at $c=1$ to obtain
\begin{align}\label{alpha123}
\alpha_1&=-\frac{2n}{3^{n+2}}(c-1)^2+O((c-1)^3),\quad  \alpha_2=-\frac{4n}{3^{n+2}}(c-1)^2+O((c-1)^3),\nonumber\\
\alpha_3&=\frac{n(n+26)}{2^{\frac{n}{2}+3}\cdot3^{\frac{n}{2}+2}}(c-1)^2(1+\theta)+O((c-1)^3);
\end{align}
and
\begin{align}\label{alpha456}
\alpha_4&=-\frac{2n}{3^{n+2}}(c-1)^2+O((c-1)^3),\quad  \alpha_5=(1+\theta)\frac{n(n+26)}{2^{\frac{n}{2}+3}\cdot 3^{\frac{n}{2}+2}}(1+\theta)(c-1)^2+O((c-1)^3),\nonumber\\
\alpha_6&=-\frac{n}{2^{n+3}}(c-1)^2(1+\theta)^2+O((c-1)^3).
\end{align}
Furthermore, $\sqrt{\Delta_3}$ can be expanded as
\begin{align}\label{delta3}
\sqrt{\Delta_3}=\sqrt{\alpha_3^2-4\alpha_1\alpha_6}+O((c-1)^3)=\frac{n\sqrt{(2 + n) (50 + n)}}{2^{\frac{n}{2}+3}\cdot3^{\frac{n}{2}+2}}(c-1)^2(1+\theta)+O((c-1)^3).
\end{align}
Substitute (\ref{alpha123}), (\ref{alpha456}) and (\ref{delta3}) into (\ref{S11star}) to obtain
\begin{align}\label{S11}
S_{11}=2^{-5 - \frac{n}{2}}\cdot 3^{\frac{n}{2}} (26 + n) (1 +\theta )-S_2^*+2^{-5 - \frac{n}{2}}\cdot 3^{\frac{n}{2}}\sqrt{(n+2)(50+n)}(1+\theta)+O(c-1),
\end{align}
Similarly, we expand $S_1^*$ at $c=1$ to obtain 
\begin{align*}
S_1^*=&2^{-\frac{n}{2}-1}\cdot 3^{\frac{n}{2}}(1+\theta)-S_2^*+2^{-1}\cdot3^{\frac{n}{2}}\sqrt{2^{-n}\theta^2-(4\cdot 3^{-\frac{n}{2}}-2^{1-n})\theta+2^{-n}}+O(c-1).
\end{align*}
Assume that $\theta$ is small, $S_1^*$ can be further expanded as
\begin{align}\label{S1star}
S_1^*=&2^{-\frac{n}{2}}\cdot 3^{\frac{n}{2}}(1+\theta)- 2^{\frac{n}{2}}\theta+O(\theta^2)+O(c-1).
\end{align}
Since for any $n\in \mathbb N^+$, 
$$2^{-\frac{n}{2}}\cdot 3^{\frac{n}{2}}<2^{-5 - \frac{n}{2}}\cdot 3^{\frac{n}{2}} (26 + n)+2^{-5 - \frac{n}{2}}\cdot 3^{\frac{n}{2}}\sqrt{(n+2)(50+n)},$$
we have from (\ref{S11}) and (\ref{S1star}) that there exist $c^*$ and $\tilde \theta^*$ such that for $c\in(1,c^*)$ and $\theta\in(0,\tilde \theta^*),$ $S_1^*<S_{11}$, which implies $I_3(S_1^*,S_2^*)<0$.  This finishes our claim and hence we obtain $(\hat  u^*,\hat  v^*)$ is unstable.  Now, we completes the proof of Proposition \ref{prop44}.


  \end{proof}
By combining Proposition \ref{prop43} and Proposition \ref{prop44}, we have Theorem \ref{thm14} holds.  This theorem indicates that there does not coexist any positive stable non-trivial pattern in every local region containing the non-degenerate maximum point of $A$, and the biological explanation is interacting aggressive species can not coexist in every local bump.
 \section{Numerical Studies and Discusion}\label{sec5}
In this section, several set of numerical simulations are presented to illustrate and highlight our theoretical analysis.  We apply the finite element method in FLEXPDE7 \cite{flex2021} to system (\ref{origin}) with the error is $10^{-4}$.  Besides supporting our theoretical results, our numerical simulations show that system (\ref{origin}) admits rich spatial-temporal dynamics. 
\begin{figure}[h!]
\centering
\begin{subfigure}[t]{0.5\textwidth}
    \includegraphics[width=\linewidth]{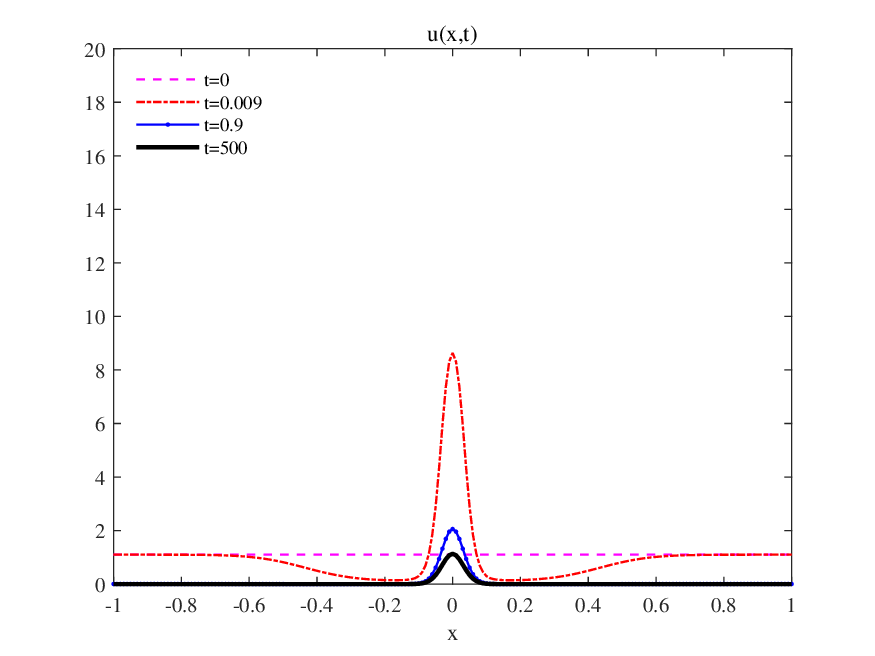}
    \caption*{$\chi=10$}
\end{subfigure}\hspace{-0.25in}
\begin{subfigure}[t]{0.5\textwidth}
  \includegraphics[width=\linewidth]{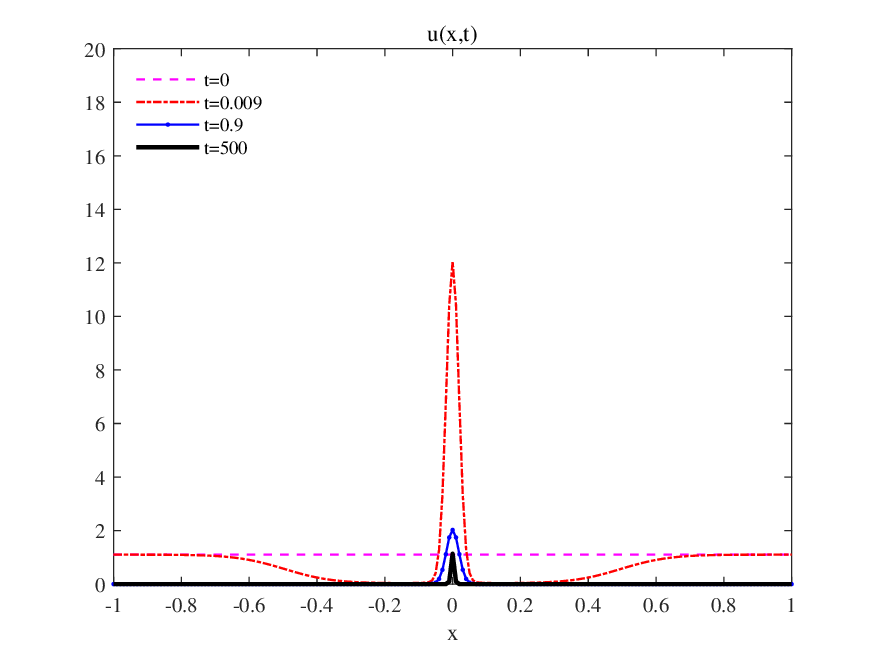}
  \caption*{$\chi=30$}
\end{subfigure}

\begin{subfigure}[t]{0.5\textwidth}
    \includegraphics[width=\linewidth]{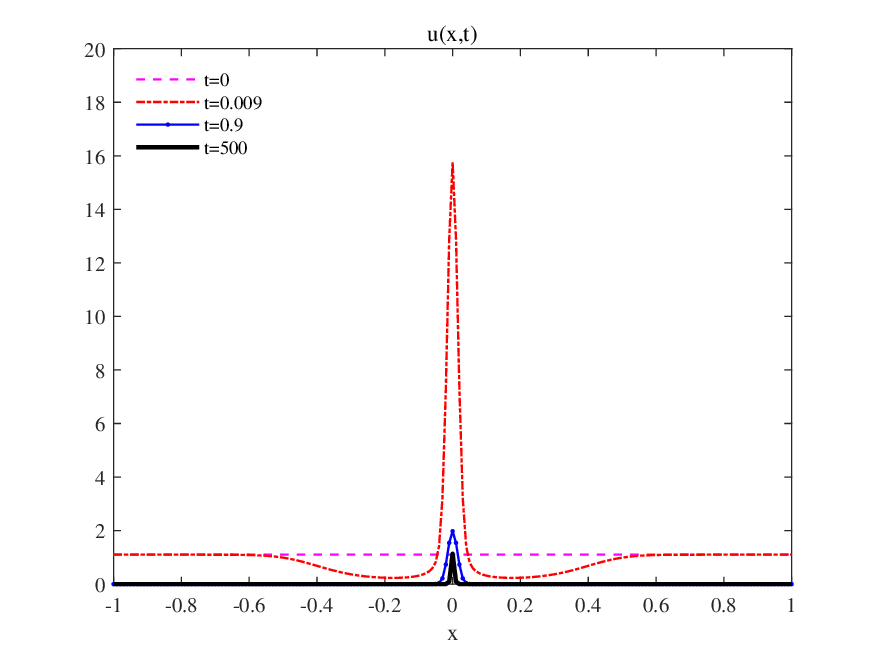}
    \caption*{$\chi=50$}
\end{subfigure}\hspace{-0.25in}
\begin{subfigure}[t]{0.5\textwidth}
  \includegraphics[width=\linewidth]{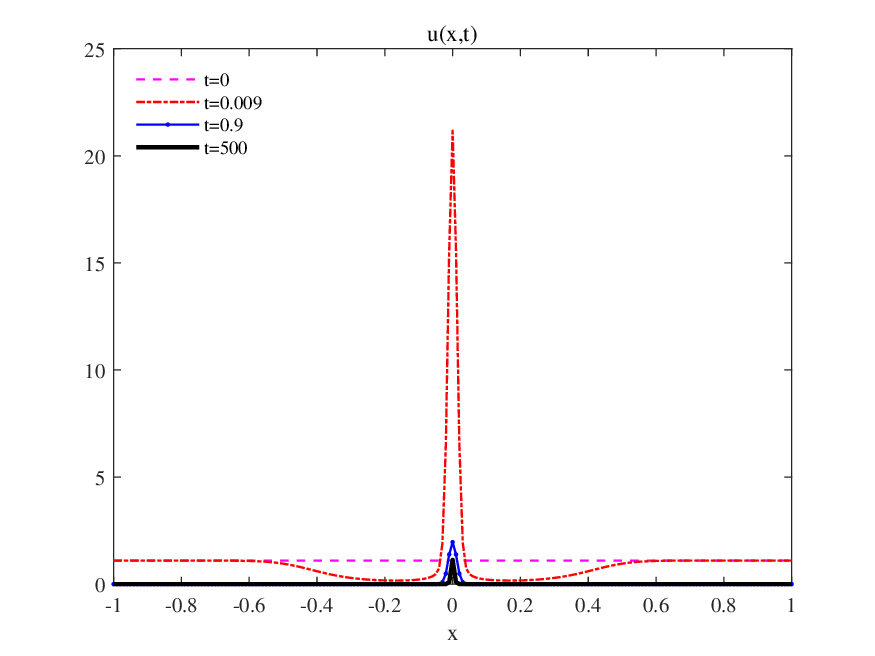}
  \caption*{$\chi=70$}
\end{subfigure}
\\ 
\caption{\textit{Dynamics of system (\ref{origin}) in a 1-D domain for different speeds $\chi$ with signal $A=\frac{5}{\sqrt{2\pi}}e^{-25x^2}$, Allee threshold $\theta=0.3$ and initial data $u_0(x)=1.1+0.001\cos(4\pi x)$, where the unit of time is second}.  We would like to mention that as $t$ increases, the solutions tend to be stable and finally converge to interior spikes and the solutions at $t=500$ can represent steady states of (\ref{origin}).  This figure also shows that the ``inner" region of steady states will shrink as $\chi\rightarrow \infty$ , but the height keeps a constant $1.1339$.}
\label{singledynamics}
\vspace{-0.0in}\end{figure}

\begin{figure}[h!]
\centering
\begin{subfigure}[t]{0.5\textwidth}
    \includegraphics[width=\linewidth]{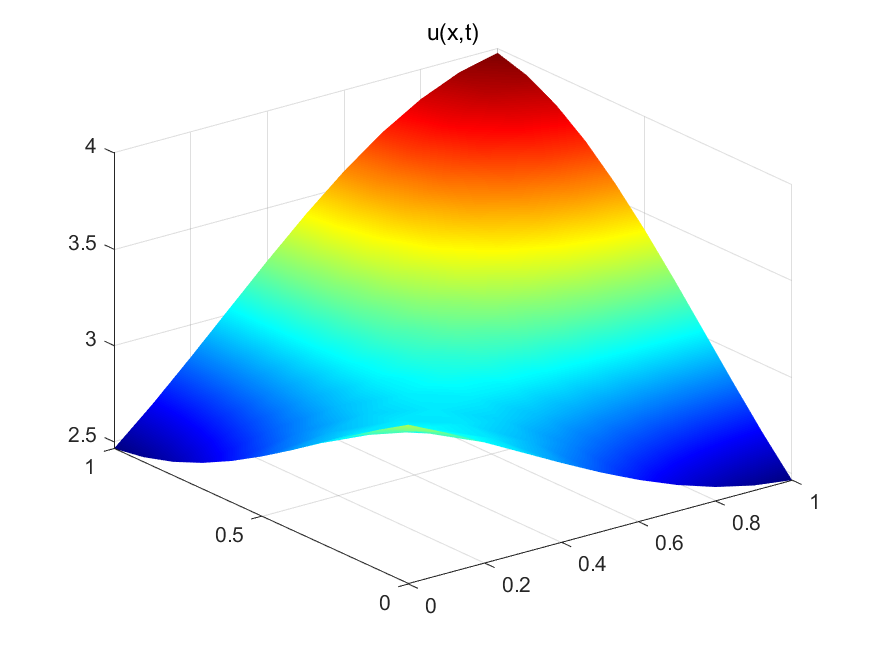}
      \caption*{t=0s}
\end{subfigure}\hspace{-0.25in}
\begin{subfigure}[t]{0.5\textwidth}
  \includegraphics[width=\linewidth]{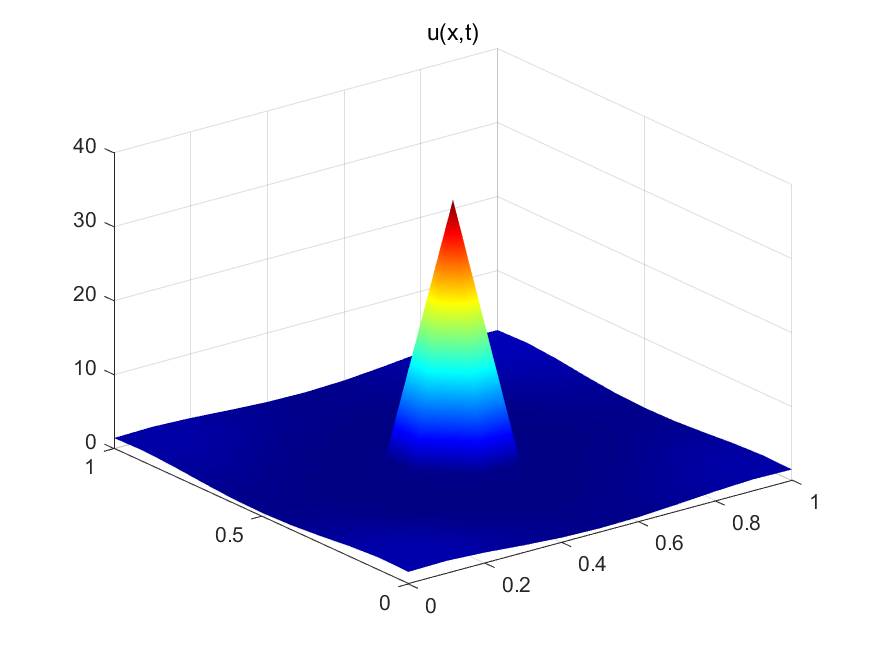}
        \caption*{t=0.01s}
\end{subfigure}

\begin{subfigure}[t]{0.5\textwidth}
    \includegraphics[width=\linewidth]{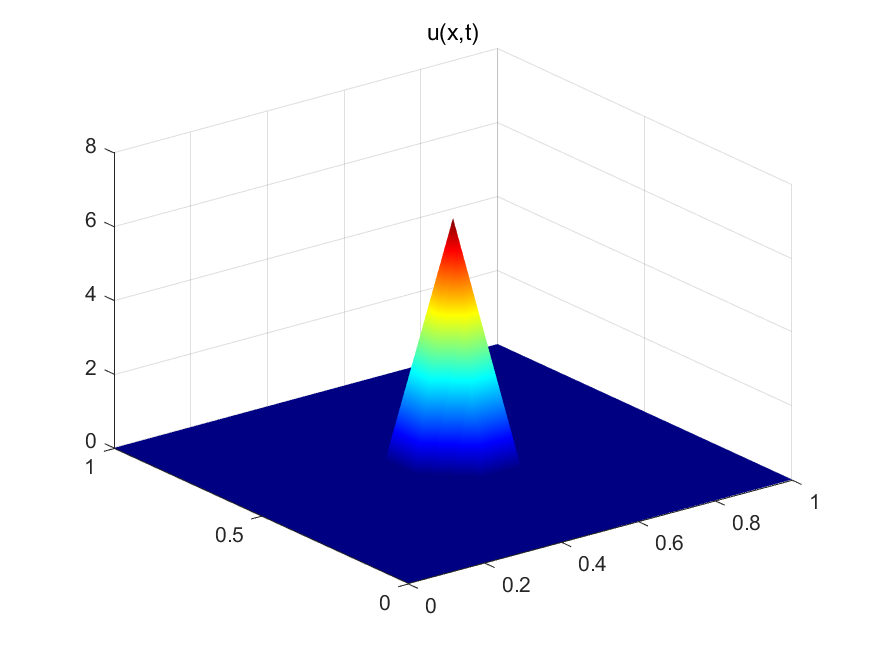}
          \caption*{t=0.1s}
\end{subfigure}\hspace{-0.25in}
\begin{subfigure}[t]{0.5\textwidth}
  \includegraphics[width=\linewidth]{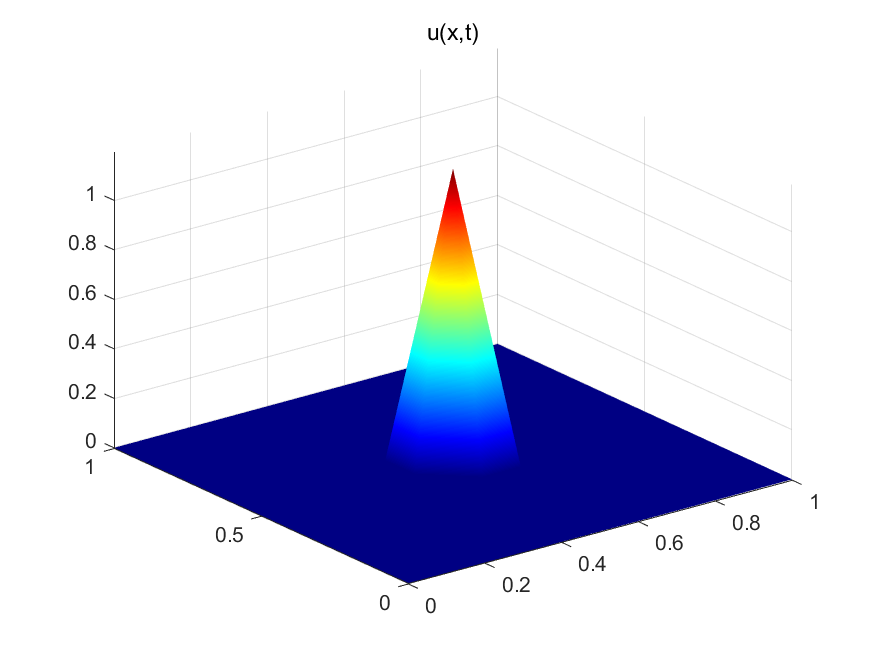}
        \caption*{t=1000s}
\end{subfigure}
\\ 
\caption{\textit{Dynamics of problem (\ref{origin}) in a 2-D rectangle with potential $A=\frac{5}{\sqrt{2\pi}}e^{-25[(x-{1}{2})^2+(y-\frac{1}{2}]^2}$, initial data $u_0=3+\cos(2x+1)\cos(2y+1)$, Allee threshold $\theta=0.3$ and conditional dispersal rate $\chi=20$.}  We can see that (\ref{origin}) admits the local stable interior spikes in the rectangular area.}
\label{2dsingle}
\end{figure}

Figure \ref{singledynamics} and Figure \ref{2dsingle} exhibit the pattern formation within system (\ref{origin}) when $A$ has only one local non-degenerate maximum point.  These figures illustrate that the single interior spike given by (\ref{u0}) with the height is $c_{01}$ is linearly stable.  Similarly, Figure \ref{unstabledynamics} shows that the single interior spike defined in (\ref{u0}) with the other positive height is unstable and some small perturbation will cause the time-dependent solutions to (\ref{origin}) move away from it.
\begin{figure}[h!]
\centering
\begin{subfigure}[t]{0.5\textwidth}
    \includegraphics[width=\linewidth]{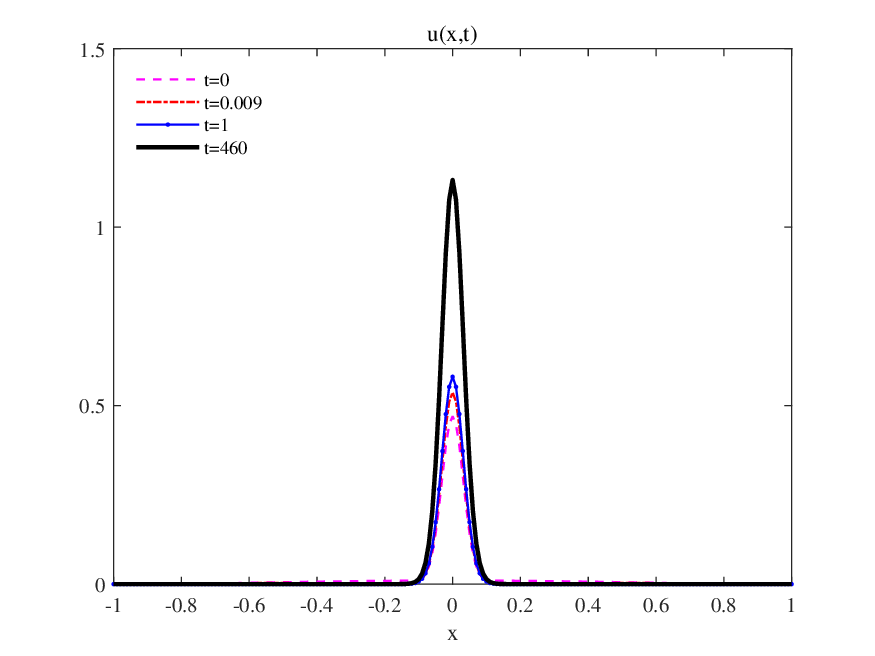}
\end{subfigure}\hspace{-0.25in}
\begin{subfigure}[t]{0.5\textwidth}
  \includegraphics[width=\linewidth]{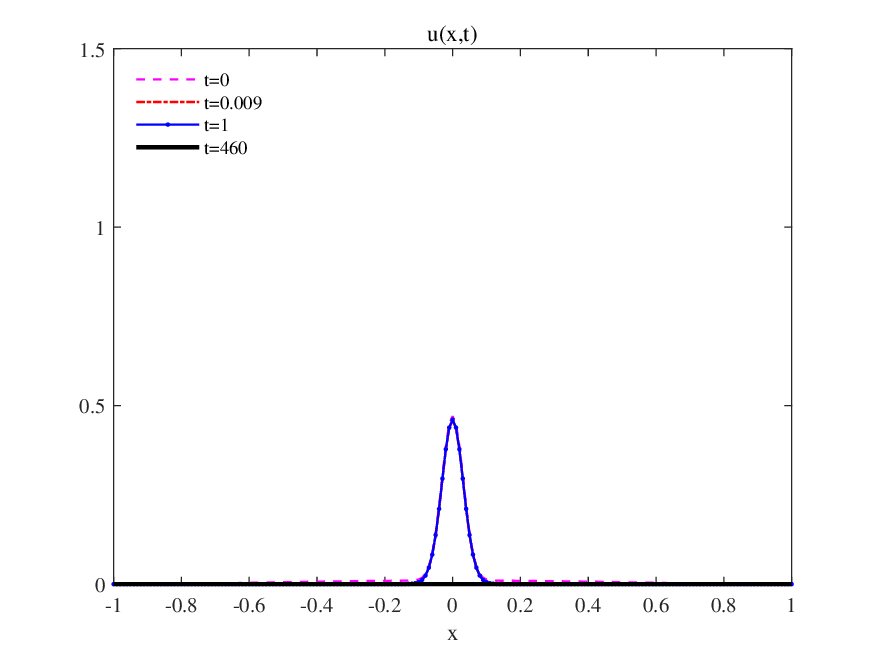}
\end{subfigure}
\\ \vspace{-0.15in}
\caption{\textit{$\theta=0.3$ and $\chi=10$.  Left: the solution of system (\ref{origin}) in 1-D at $t=0$, $0.009$, $1$ and $460 s$  with initial data $u_0(x)=0.46e^{-50\chi x^2}+0.01\cos(2 x)$; Right: the solution of system (\ref{origin}) in 1-D with initial data $u_0(x)=0.46e^{-50\chi x^2}+0.001\cos(2 x)$.}  We have $u(x,t)$ at $t=460$ can represent the steady state to (\ref{origin}) and it is shown that the single interior spike with the smaller height is unstable and converges to either the steady state with the larger height or zero in a long term.}
\label{unstabledynamics}
\end{figure}

We next present the stability of multi-interior spikes defined in (\ref{u0}) when signal $A$ admits two local non-degenerate maximum points.  Before that, the asymptotic profiles of them are shown in Figure \ref{multispike}.
\begin{figure}[h!]
\centering
\begin{subfigure}[t]{0.5\textwidth}
    \includegraphics[width=\linewidth]{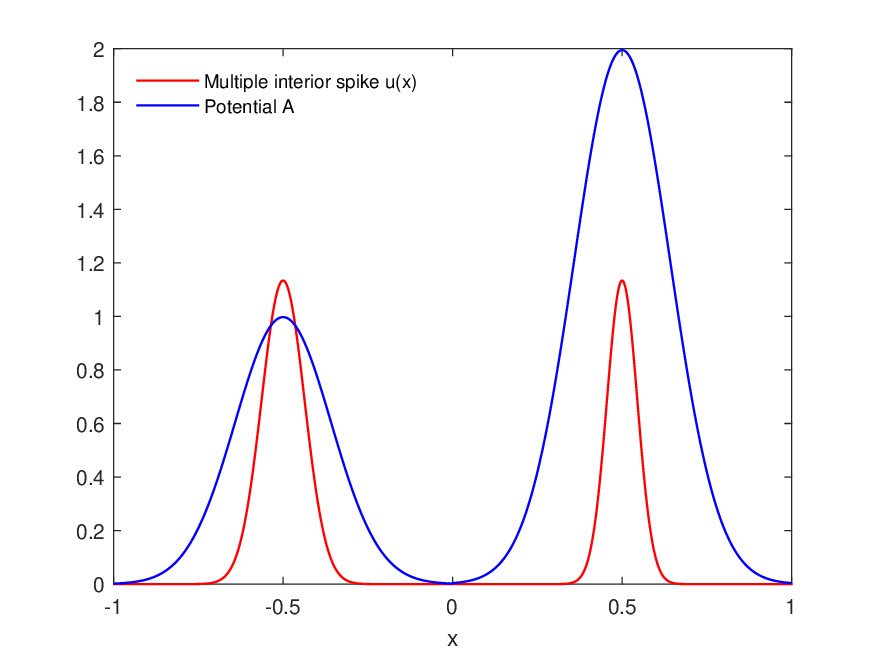}
\end{subfigure}\hspace{-0.25in}
\begin{subfigure}[t]{0.5\textwidth}
  \includegraphics[width=\linewidth]{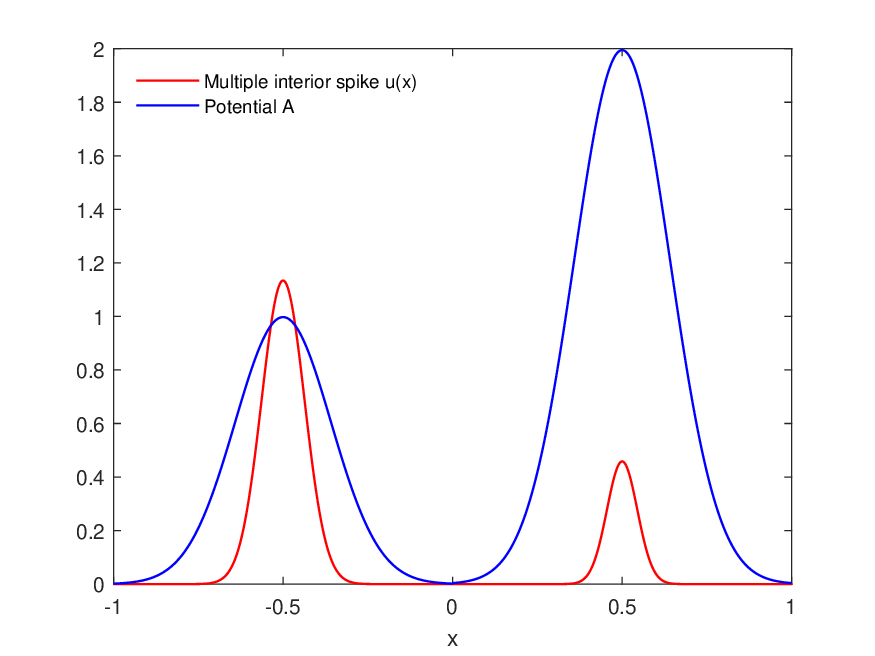}
\end{subfigure}

\begin{subfigure}[t]{0.5\textwidth}
    \includegraphics[width=\linewidth]{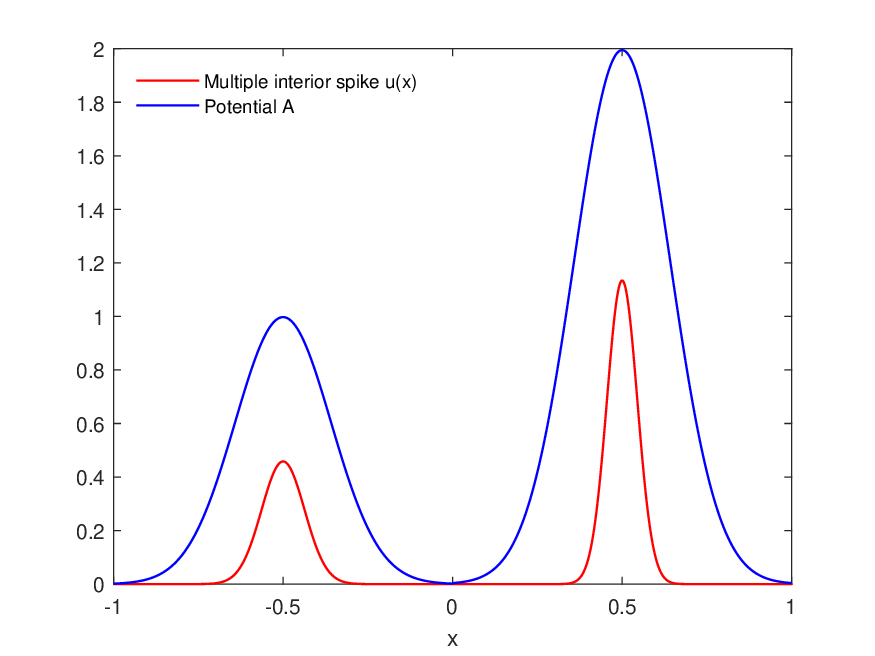}
\end{subfigure}\hspace{-0.25in}
\begin{subfigure}[t]{0.5\textwidth}
  \includegraphics[width=\linewidth]{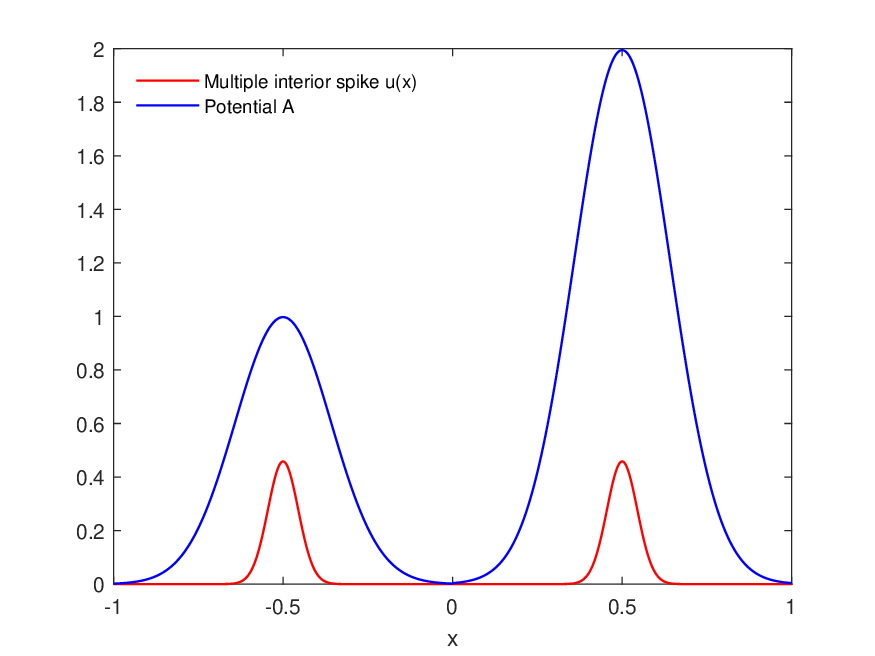}
\end{subfigure}
\\ \vspace{-0.15in}
\caption{\textit{For a 1-D domain, we have four types of interior spikes constructed in Theorem \ref{thm11} with the height of each bump is either $c_{01}=1.1339$ or $c_{02}=0.4582$ when potential $A=\frac{5}{\sqrt{2\pi}}e^{-25(x-0.5)^2}+\frac{5}{2\sqrt{2\pi}}e^{-25(x+0.5)^2}$, Allee threshold $\theta=0.3$ and conditional dispersal rate $\chi=10$.}  We find that (\ref{origin}) admits a variety of interior spikes when $A$ has more than $1$ local maximum points.}
\label{multispike}
\vspace{-0.0in}\end{figure}
Similar to the single interior spike, our numerical result shown in Figure \ref{multidynamic1} indicates that those multiple spiky solutions whose every bump has the larger height are local linearly stable.
\begin{figure}[h!]
\centering
\begin{subfigure}[t]{0.5\textwidth}
    \includegraphics[width=\linewidth]{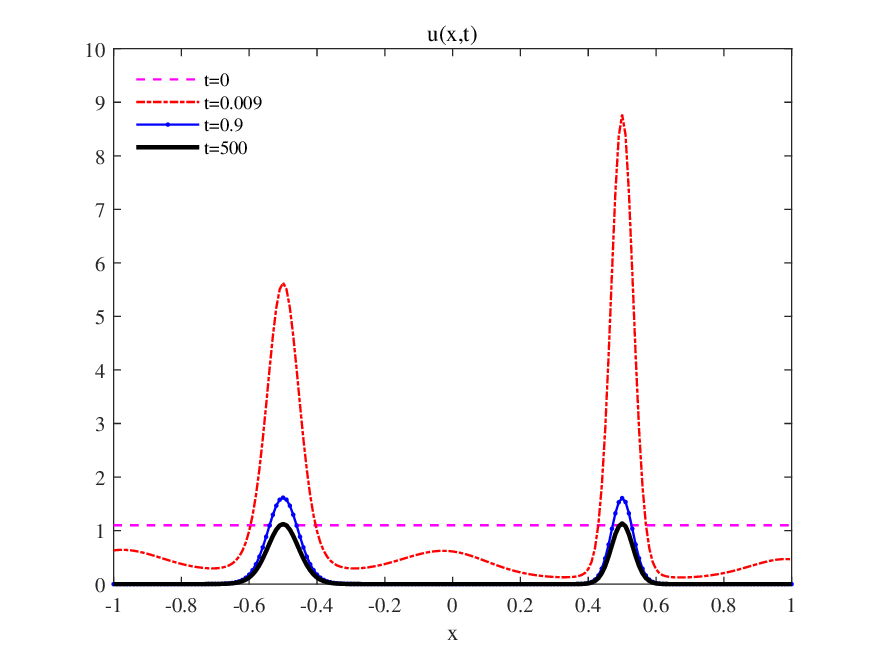}
      \caption*{Initial data $u_0=1.1$}
\end{subfigure}\hspace{-0.25in}
\begin{subfigure}[t]{0.5\textwidth}
  \includegraphics[width=\linewidth]{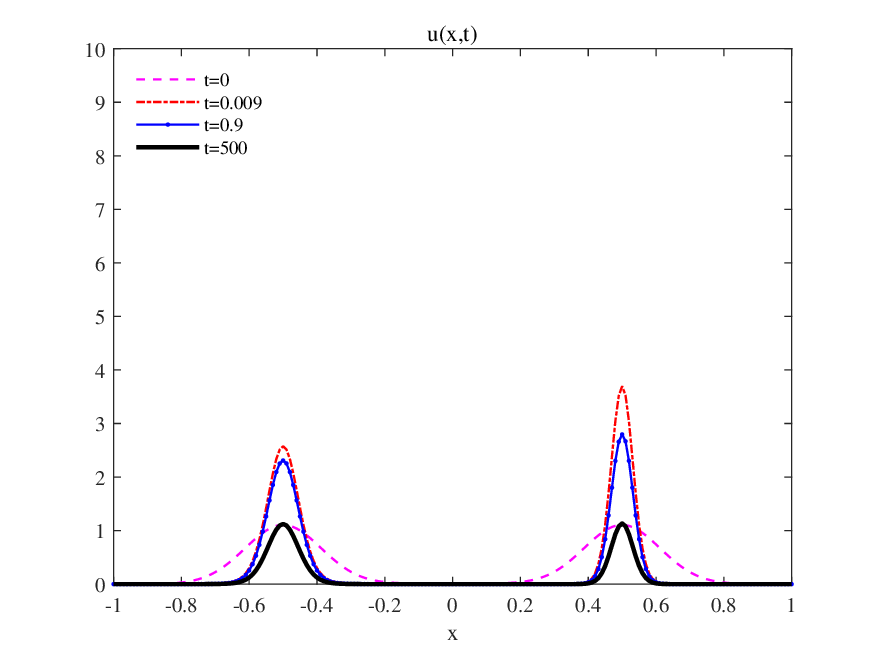}
        \caption*{Initial data $u_0=1.1e^{-40(x+1/2)^2}+1.1e^{-40(x+1/2)^2}+0.01\cos(2x)$}
\end{subfigure}

\begin{subfigure}[t]{0.5\textwidth}
    \includegraphics[width=\linewidth]{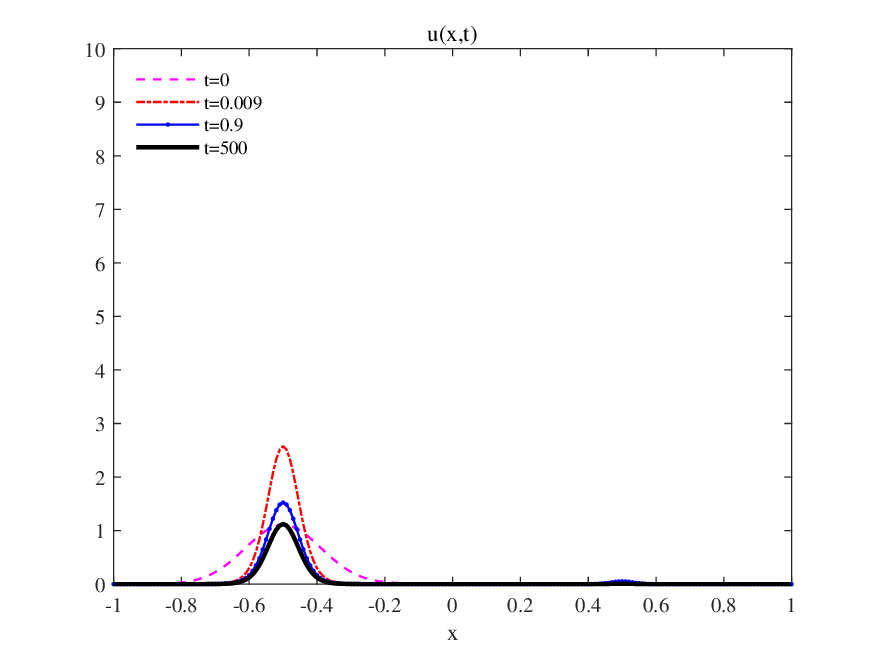}
          \caption*{Initial data $u_0=1.1e^{-40(x+1/2)^2}+0.01\cos(2x)$}
\end{subfigure}\hspace{-0.25in}
\begin{subfigure}[t]{0.5\textwidth}
  \includegraphics[width=\linewidth]{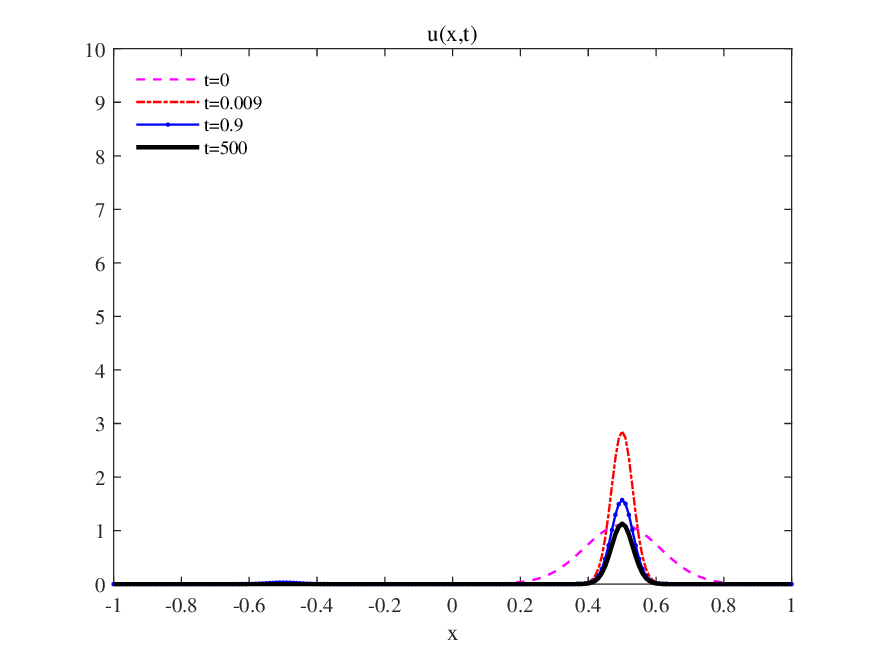}
        \caption*{Initial data $u_0=1.1e^{-40(x-1/2)^2}+0.01\cos(2x)$}
\end{subfigure}
\\ 
\caption{\textit{Dynamics of time-dependent problem (\ref{origin}) in 1-D with potential $A=\frac{5}{\sqrt{2\pi}}e^{-25(x-0.5)^2}+\frac{5}{2\sqrt{2\pi}}e^{-25(x+0.5)^2}$, Allee threshold $\theta=0.3$ and conditional dispersal rate $\chi=10$.}  The solutions to (\ref{origin}) at $t=500s$ can be regarded as the steady states and this figure demonstrates that steady states with the larger positive heights are locally linearized stable.}
\label{multidynamic1}
\end{figure}
In contrast, once one of their bumps admits the smaller height, the stationary solutions will become unstable, as shown in Figure \ref{multiunstable}.   
\begin{figure}[h!]
\centering
\vspace{-0.1in}
\begin{subfigure}[t]{0.3\textwidth}
    \includegraphics[width=\linewidth]{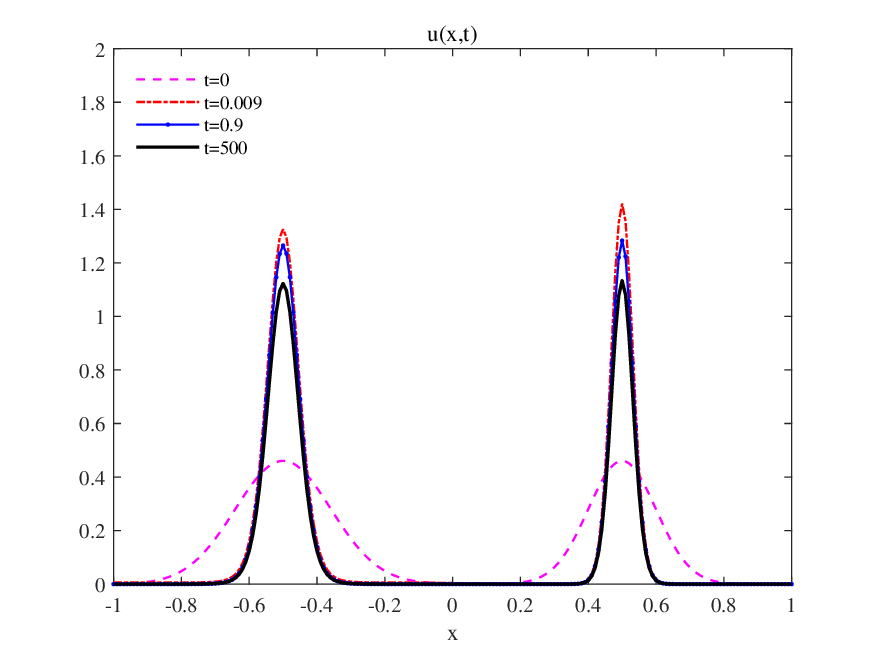}
\caption*{Initial data $u_0=0.46\Big[e^{-50(x+\frac{1}{2})^2}+e^{-50(x-\frac{1}{2})^2}\Big]+0.01\cos(x^2)$}
\end{subfigure}
\begin{subfigure}[t]{0.3\textwidth}
  \includegraphics[width=\linewidth]{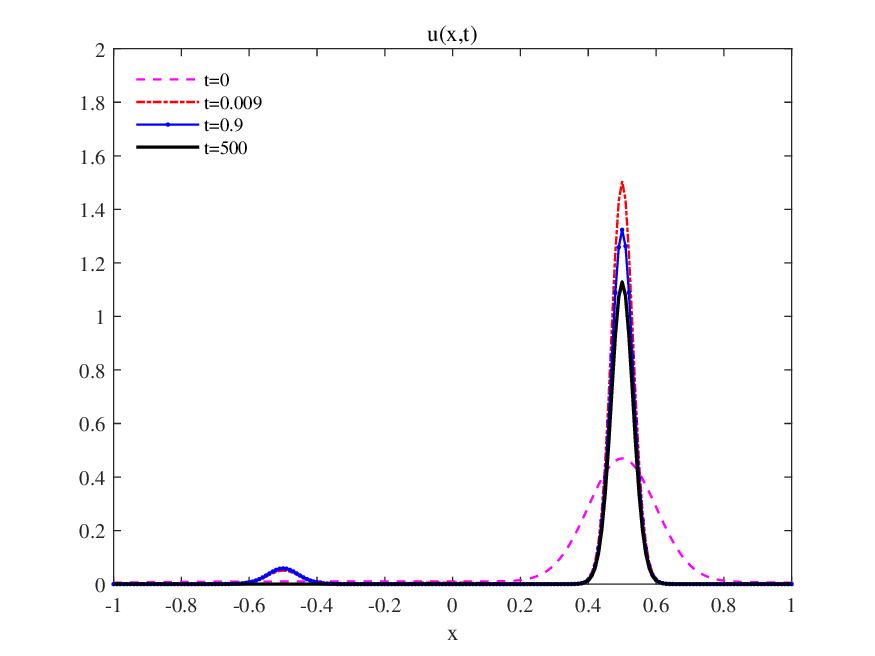}
\caption*{Initial data $u_0=0.46e^{-50(x-1/2)^2}+0.01\cos(x^2)$}
\end{subfigure}
\begin{subfigure}[t]{0.3\textwidth}
    \includegraphics[width=\linewidth]{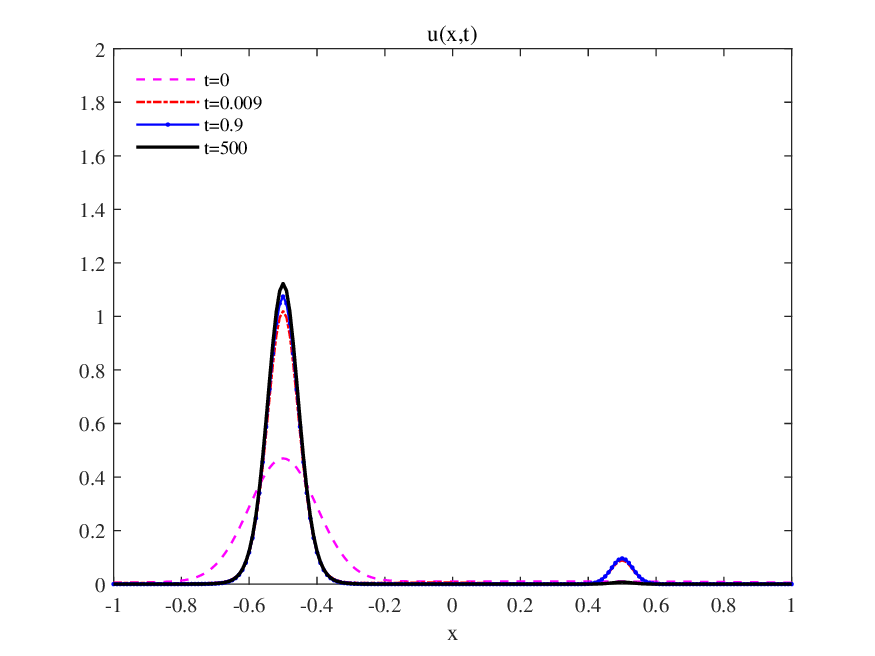}
\caption*{Initial data $0.46e^{-50(x+1/2)^2}+0.01\cos(x^2)$}
\end{subfigure}
\\ \vspace{-0.15in}
\caption{\textit{For a 1-D domain, given Allee threshold $\theta=0.3$, speed $\chi=10$ and signal $A=\frac{1}{\sqrt{2\pi}\sigma}e^{-\frac{(x-1/2)^2}{\sigma^2}}+\frac{1}{\sqrt{2\pi}\sigma}e^{-\frac{(x+1/2)^2}{\sigma^2}}$ with $\sigma=0.2$, we have the solutions to (\ref{origin}) with different initial data at $t=0$, $0.009$, $0.9$ and $500 s$.}  Similarly, the time-dependent solutions $u(x,t)$ at $t=500$ are steady states to (\ref{origin}), and the heights of every bumps in steady states are either $c_{01}$ or $O(\epsilon).$  According to the choice of initial data, we find the interior spikes with the smaller heights are unstable.}
\label{multiunstable}
\end{figure}

Figure \ref{system1} and Figure \ref{system2} exhibit the large time behavior of solutions to (\ref{IFD}) and (\ref{AA}), respectively.  From the viewpoint of population ecology, the numerical results shown in Figure \ref{system1} can be interpreted that the conservative species will be better off in the long run when the Allee threshold is small.  This phenomenon is counter-intuitive since one might believe that the higher speed benefits the persistence of a species, so that an aggressive species is more likely to survive.  Our result demonstrates that aggressive strategy is not always optimal and that an IFD strategy is preferable for species persistence in some cases.  A further qualitative result shown in Figure \ref{system2} is that competitive species does not like to coexist and, instead, prefer to occupy all resources by themselves.  Our interpretation of this result is that aggressive species do not want to share any resources with each other.  
\begin{figure}[h!]
\centering
\begin{subfigure}[t]{0.5\textwidth}
    \includegraphics[width=\linewidth]{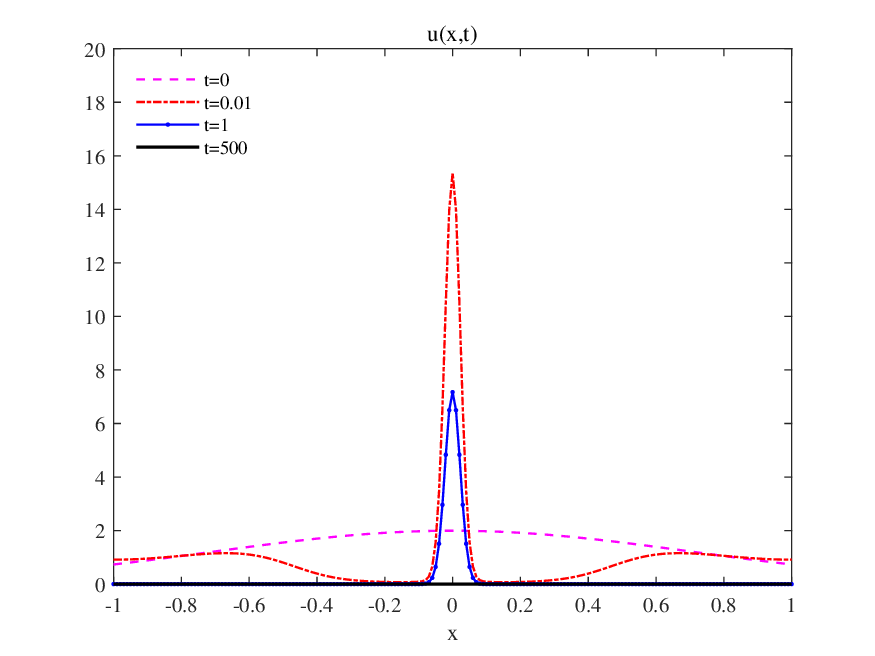}
    \caption*{$\theta=0.03$, $A=\frac{5}{\sqrt{2\pi}}e^{-25x^2}$\text{~and~} $v_0=0.1+0.01\cos x$ }
\end{subfigure}\hspace{-0.25in}
\begin{subfigure}[t]{0.5\textwidth}
  \includegraphics[width=\linewidth]{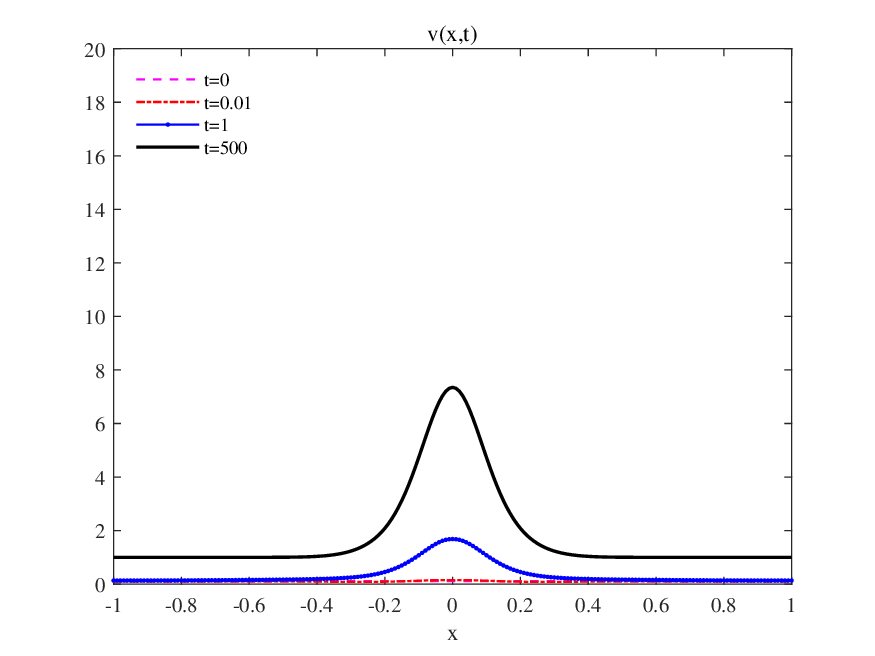}
  \caption*{$\theta=0.03$, $A=\frac{5}{\sqrt{2\pi}}e^{-25x^2}$\text{~and~}$u_0=\frac{5}{\sqrt{2\pi}}e^{-x^2}$ }
\end{subfigure}

\begin{subfigure}[t]{0.5\textwidth}
    \includegraphics[width=\linewidth]{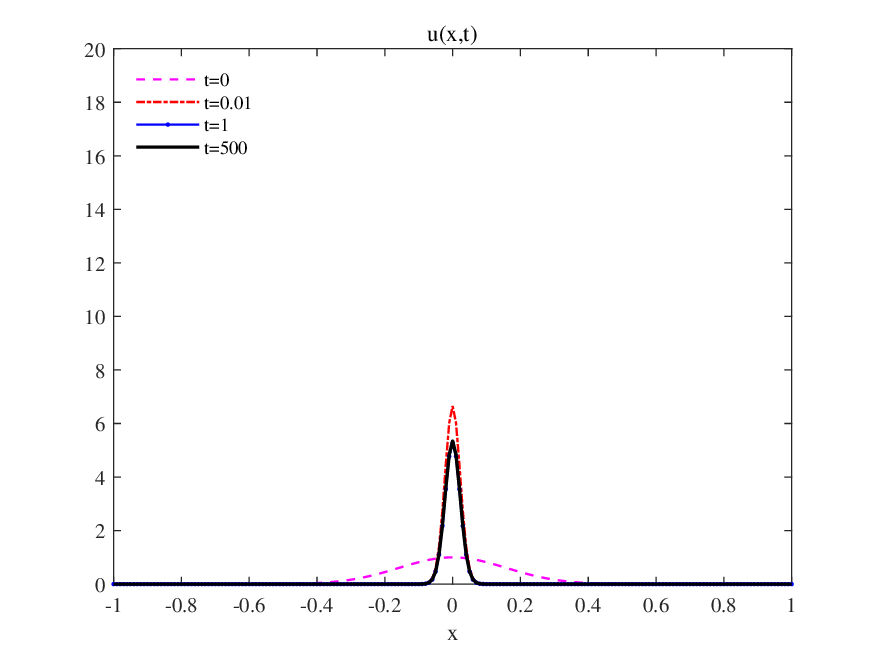}
    \caption*{$\theta=0.3$, $A=\frac{5}{\sqrt{2\pi}}e^{-25x^2}-0.5$\text~{and~}$v_0=e^{-\chi x^2}$}
\end{subfigure}\hspace{-0.25in}
\begin{subfigure}[t]{0.5\textwidth}
  \includegraphics[width=\linewidth]{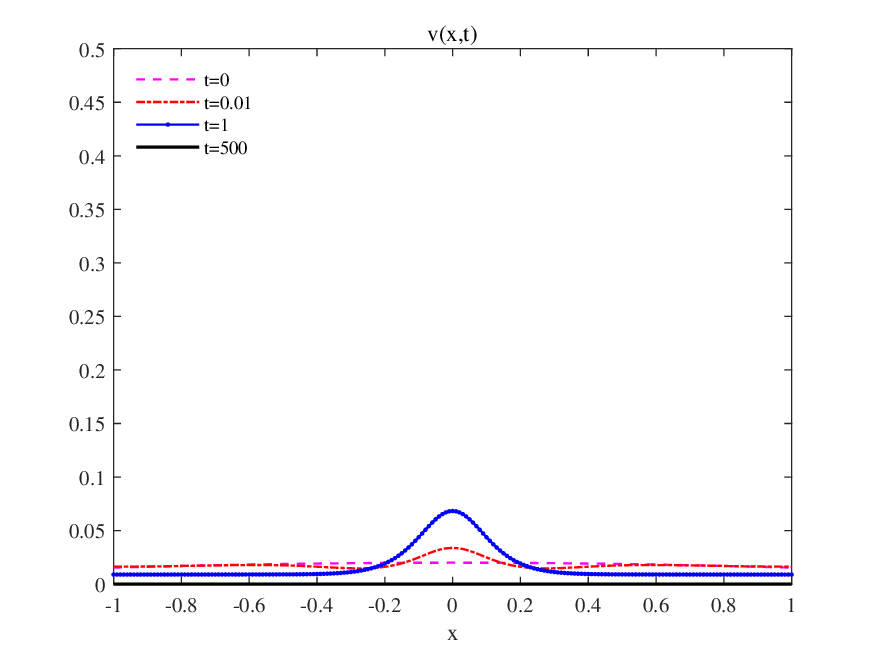}
  \caption*{$\theta=0.3$, $A=\frac{5}{\sqrt{2\pi}}e^{-25x^2}-0.5$\text~{and~}$u_0=0.01+0.01\cos x$}
\end{subfigure}

\begin{subfigure}[t]{0.5\textwidth}
    \includegraphics[width=\linewidth]{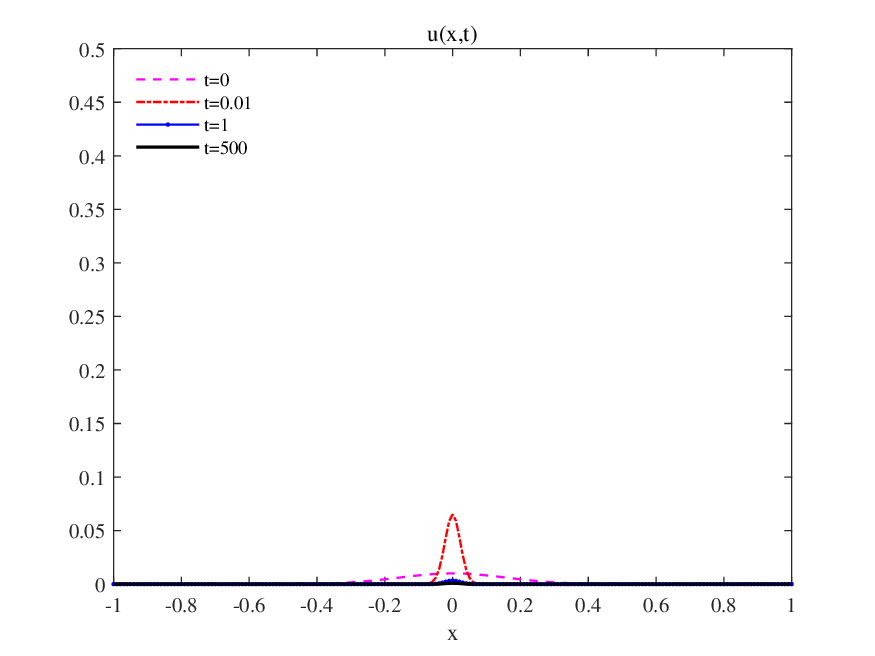}
    \caption*{$\theta=0.3$, $A=\frac{5}{\sqrt{2\pi}}e^{-25x^2}-0.5$ and $u_0=0.01 e^{-\chi x^2}$}
\end{subfigure}\hspace{-0.25in}
\begin{subfigure}[t]{0.5\textwidth}
  \includegraphics[width=\linewidth]{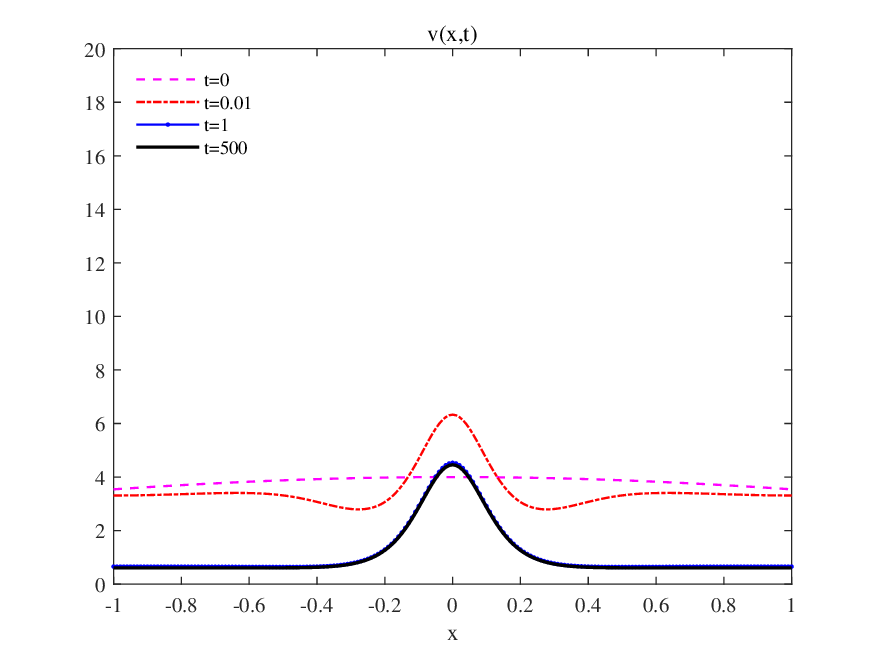}
  \caption*{$\theta=0.3$, $A=\frac{5}{\sqrt{2\pi}}e^{-25x^2}-0.5$ and $v_0=3+\cos x$}
\end{subfigure}
\\ 
\caption{\textit{Dynamics of system (\ref{IFD}) in 1-D for different $\theta$ with $\chi=20$, where the unit of time is second}.  It can be seen that when Allee threshold $\theta$ is small, the $u$-species persists while the $v$-species is extinct finally even though the amount of initial resources possessed by the $u$-species is small.}
\label{system1}
\end{figure}

\begin{figure}[h!]
\centering
\begin{subfigure}[t]{0.5\textwidth}
    \includegraphics[width=\linewidth]{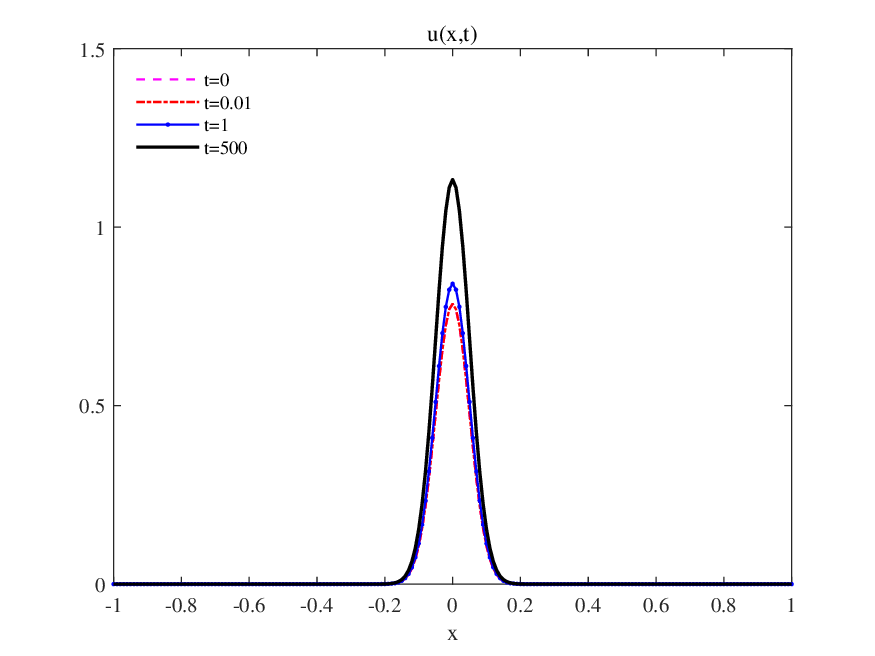}
      \caption*{Initial data $u_0=0.785 e^{-\chi x^2}$}
\end{subfigure}\hspace{-0.25in}
\begin{subfigure}[t]{0.5\textwidth}
  \includegraphics[width=\linewidth]{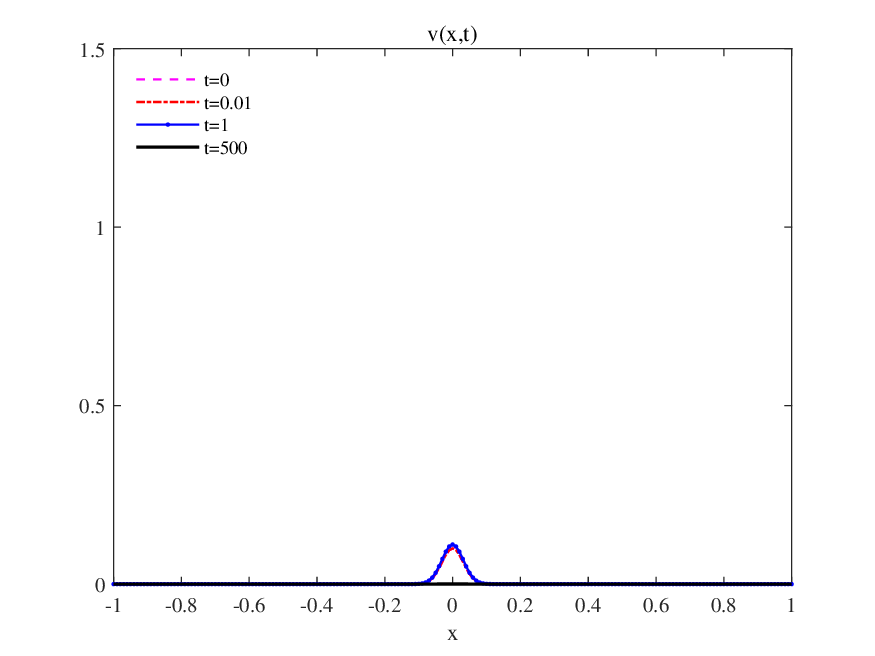}
        \caption*{Initial data $v_0=0.1e^{-\frac{5}{2}\chi x^2}$}
\end{subfigure}

\begin{subfigure}[t]{0.5\textwidth}
    \includegraphics[width=\linewidth]{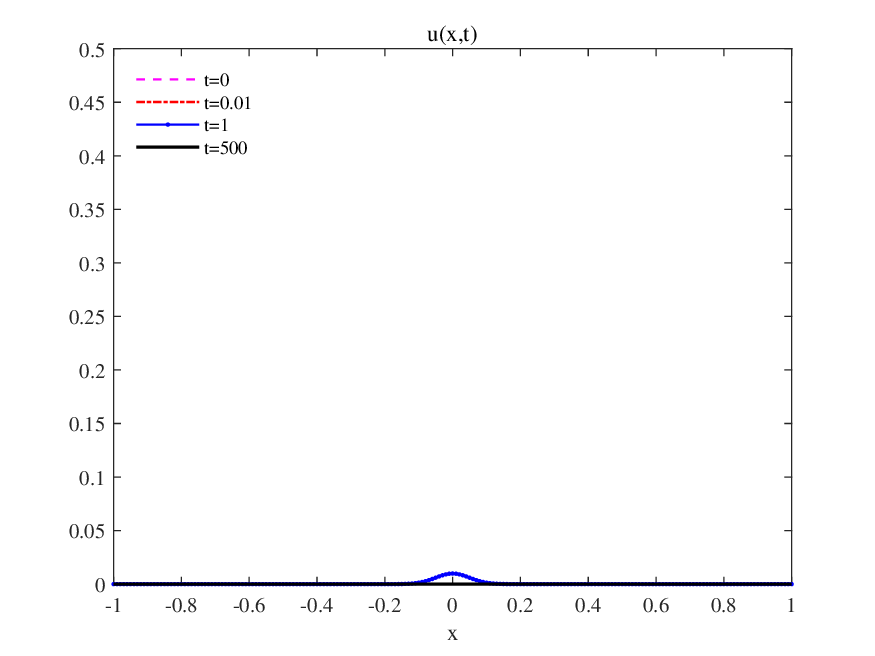}
          \caption*{Initial data $u_0=0.01e^{-\chi x^2}$}
\end{subfigure}\hspace{-0.25in}
\begin{subfigure}[t]{0.5\textwidth}
  \includegraphics[width=\linewidth]{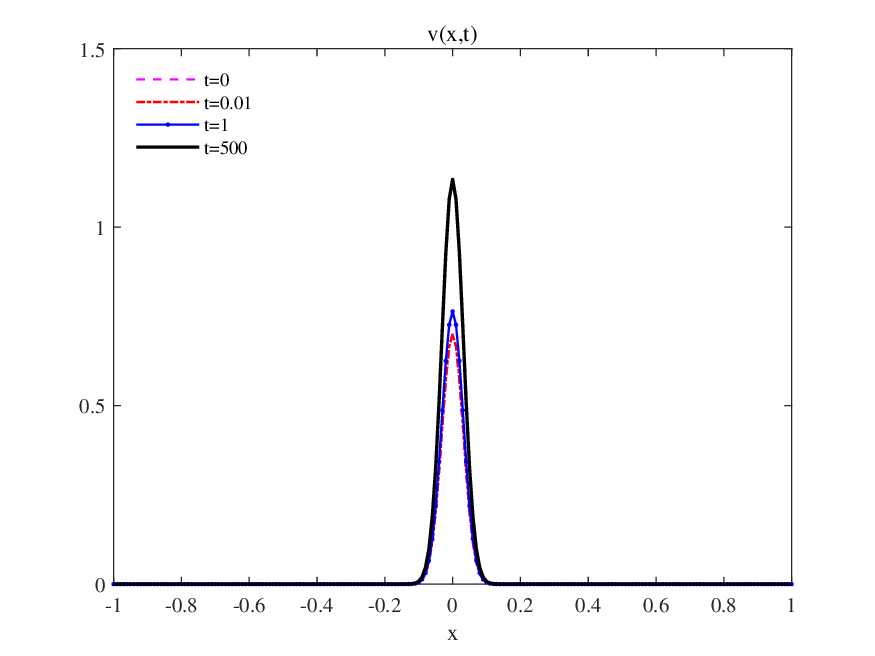}
        \caption*{Initial data $v_0=0.7e^{-\frac{5}{2} \chi x^2}$}
\end{subfigure}
\\ 
\caption{\textit{Dynamics of system (\ref{AA}) in 1-D with potential $A=-x^2$, Allee threshold $\theta=0.3$, aggressive ratio $c=2.5$ and conditional dispersal rate $\chi=200$.} We find either the $u$-species attains the nontrivial pattern with the $v-$species vanishes or $v$ survives with $u$ vanishes.}
\label{system2}
\end{figure}
\subsection{Discussion}
We have used the reduction method to construct and study the linear stability of localized solutions to the single species models (\ref{origin}) and competition models (\ref{14}) in the limit of an asymptotically large speed $\chi\gg 1$.  Our main contribution has been the rigorous analysis of the existence of localized patterns and their stability properties.  Under the technical assumptions (A1), (A2) and (H1), (H2), we showed that (\ref{origin}) admits many localized solutions when the potential $A$ has multiple maximum points.  In particular, there are two possible heights for every local bump.  Regarding the stability properties, we proved that once some local bump has the small height, the spike will be unstable.  We next focused on the analysis of the population model (\ref{14}).  On the one hand, we proved the non-coexistence of two competing species who follow the aggressive strategy and IFD strategy, respectively, for Allee threshold sufficiently small.  Moreover, we found that when the Allee threshold $\theta$ is small but independent of $\chi$, the species who follows the aggressive strategy can persist when all peaks have the large heights; while $\theta$ depending on $\chi$ is small, the aggressive strategy will lead to the extinction of species.  On the other hand, with the assumption that two species both follow the aggressive strategy, we showed that even though the localized patterns might coexist in local bumps, they are unstable.

We would like to mention that there are also some open problems that deserve future explorations.  While discussing the existence of interior spike steady states, we impose some technical assumptions on $A$; for instance, we assume that $A$ has $k$ non-degenerate maximum points.  Whether or not these assumptions can be removed remains an open problem.  Besides the stable interior spikes, we believe that (\ref{origin}) also admits the stable boundary spikes, and the rigorous analysis needs to be established.  Regarding the population system (\ref{14}), we only study the influence of large advection on the population evolution of interacting species.  The effect of small diffusion $d_1$ is apparently another delicate problem that deserves probing in the future.
\section*{Acknowledgments}

 We thank Professor M. Ward for stimulating discussions, critically reading the manuscript and  many critical suggestions. We also thank Professors C. Cosner and N. Rodriguez for  valuable comments.  The research of J. Wei is partially supported by NSERC of Canada.

\bibliographystyle{plain}
\bibliography{ref}
\end{document}